\documentclass[review,11pt]{elsarticle}
\usepackage{lipsum}
\usepackage{amsthm}
\makeatletter
\def\ps@pprintTitle{%
 \let\@oddhead\@empty
 \let\@evenhead\@empty
 \def\@oddfoot{}%
 \let\@evenfoot\@oddfoot}
\makeatother
\usepackage[margin=2.5cm]{geometry}
\usepackage{amsthm}
\usepackage{lineno,hyperref}
\usepackage{amsmath,amssymb,latexsym, amscd}
\usepackage{exscale, eps fig, graphics}
\usepackage{array}
\usepackage[usenames,dvipsnames]{color}
\usepackage{float}
\restylefloat{table}
\newcommand\numberthis{\addtocounter{equation}{1}\tag{\theequation}}
\usepackage{algorithmic}
\usepackage{placeins}
\usepackage{longtable}
\usepackage{makecell}
\usepackage[utf8]{inputenc}
\usepackage[justification=centering]{caption}
\usepackage{graphicx}
\usepackage{caption}
\usepackage{subcaption}
\usepackage{mathtools}
\usepackage{eucal}
\usepackage{mathrsfs}
\usepackage{fontenc}
\usepackage[T1]{fontenc}
\hypersetup{pdfauthor={Name}}

\newtheorem{thm}{Theorem}[section]
\newtheorem{lem}[thm]{Lemma}
\newtheorem{prop}[thm]{Proposition}
\theoremstyle{definition}
\newtheorem{remark}{Remark}
\newtheorem*{notation}{Notation}
\newtheorem*{acknowledgments}{Acknowledgments}
\newtheorem*{data availability}{Data availability}
\newtheorem*{conflict}{Conflict of interest statement}

\newtheorem{cor}[thm]{Corollary}

\newcommand{\BigO}[1]{\ensuremath{\operatorname{O}\left(#1\right)}}

\newcommand{\BigOm}[1]{\ensuremath{\operatorname{O}_{m}\left(#1\right)}}

\newcommand{\BigOkmld}[1]{\ensuremath{\operatorname{O}_{m,d,\ell}\left(#1\right)}}
\newcommand{\BigOmlamda}[1]{\ensuremath{\operatorname{O}_{m,\Lambda}\left(#1\right)}}

\renewcommand{\geq}{\geqslant}
\renewcommand{\leq}{\leqslant}

\renewcommand{\L}{\mathcal{L}}

\renewcommand{\r}{\rangle}

\DeclareMathOperator{\lcm}{lcm}

\begin{document}

\begin{frontmatter}

\title{Distribution of Farey fractions with $k$-free denominators}

\author[addr1]{Bittu Chahal}
 \ead{bittuchahal2397@gmail.com}
 
\author[addr2]{Tapas Chatterjee}
 \ead{tapasc@iitrpr.ac.in}

 \author[addr3]{Sneha Chaubey}
 \ead{sneha@iiitd.ac.in}

\address[addr1]{Department of Mathematics, IIT Indore, Madhya Pradesh 453552.}
\address[addr2]{Department of Mathematics, IIT Ropar, Punjab 140001.}
 \address[addr3]{Department of Mathematics, IIIT Delhi, New Delhi 110020.}


\begin{abstract}
We investigate the distributional properties of the sequence of Farey fractions with $k$-free denominators in residue classes, defined as
\[\mathscr{F}_{Q,k}^{(m)}:=\left\{\frac{a}{q}\ |\ 1\leq a\leq q\leq Q,\ \gcd(a,q)=1,\ q\ \text{is}\ k\text{-free}\ \&\ q\equiv b\pmod{m} \right\}.\]
We show that $\left(\mathscr{F}_{Q,k}^{(m)}\right)_{Q\ge 1}$ is equidistributed modulo one, and prove analogues of the classical results of Franel, Landau, and Niederreiter for $\left(\mathscr{F}_{Q,k}^{(m)}\right)_{Q\ge 1}$, particularly, deriving an equivalent form of the generalized Riemann hypothesis (GRH) for Dirichlet $L$-functions in terms of the distribution of $\left(\mathscr{F}_{Q,k}^{(m)}\right)_{Q\ge 1}$. 
Beyond examining the global distribution, we also study the local statistics of these sequences. We establish formulas for all levels ($\nu\ge 2$) of correlation measure. Specifically, we show the existence of the limiting pair  ($\nu=2$) correlation function and provide an explicit expression for it. Our results are based upon the estimation of weighted Weyl sums and weighted lattice point counting in restricted domains. 
\end{abstract}


\end{frontmatter}

\tableofcontents

\section{Introduction and main results}
Let $Q$ be a positive integer. The Farey sequence of order $Q$ is defined as follows:
\[\mathcal{F}_Q:=\left\{\frac{a}{q} : 1\leq a\leq q\leq Q,\ \gcd(a,q)=1 \right\}. \]

Let $k\geq 2$ be an integer. A number $n$ is said to be $k$-free if for every prime $p|n$, we have $p^k\nmid n$. It is well known that the density of $k$-free numbers is $1/\zeta(k)$, where $\zeta(s)$ represents the Riemann zeta function. There is a vast literature on the distribution of $k$-free numbers \cite{MR43120, MR2180456}.
In this article, we are interested in the distribution of Farey fractions whose denominators are $k$-free and that lie within an arithmetic progression. 
Denote
\[\mathscr{F}_{Q,k}^{(m)}:=\left\{\frac{a}{q}\ |\ 1\leq a\leq q\leq Q,\ \gcd(a,q)=1,\ q\ \text{is}\ k\text{-free}\ \&\ q\equiv b\pmod{m} \right\}, \numberthis\label{f_k}\]
where $m\in\mathbb{N},\ b\in\mathbb{Z}$ and $(b,m)=1$.

Equidistribution modulo one is concerned with the distribution of fractional parts of real numbers in $[0,1]$. 
A sequence $(x_n)_{n=1}^{\infty}$ of real numbers is said to be equidistributed or uniformly distributed modulo one, if for every interval $I\subseteq[0,1)$, we have
\[\lim_{N\to\infty}\frac{1}{N}\#\left\{1\leq n\leq N\ |\ \{x_n\}\in I \right\}=|I|, \]
where $\{x_n\}$ denotes the fractional part of $x_n$. 
The development of the theory of equidistribution began with the classical work of Weyl \cite{Weyl}, where he connected equidistribution with an exponential sum, also known as Weyl sum.
 The Weyl sums are central to various number-theoretic problems, including the zero-free region of the Riemann zeta function, the prime number theorem, and the Diophantine equations. The Weyl sums have been extensively studied in different forms by various authors. Specifically, the Weyl sum over the roots of quadratic congruences was studied in \cite{MR2926988, MR4137069}. The metric theory of Weyl sums appeared in \cite{MR4542717}. For more details and problems on the Weyl sums, one may refer to \cite{MR4843309, MR4102722, MR4298525} and references therein. In our first result, we establish an upper bound for the Weyl sum over Farey fractions with $k$-free denominators in residue classes. The Weyl sum for Farey fractions was dealt in \cite{MR922425, Fujii}.
\begin{thm}\label{Weyl sum}
For $r\in\mathbb{Z}\setminus\{0\}$, we have
    \[\sum_{\gamma\in \mathscr{F}_{Q,k}^{(m)} } e\left(r\gamma \right)=\BigOm{\min\left(Q,r^{\epsilon} \right)Q\exp{\left(-c\frac{(\log Q)^{3/5}}{(\log\log Q)^{1/5}} \right)}},\]
    where $c>0$ is some constant and $e(x)=e^{2\pi ix}$.
\end{thm}
The above theorem in conjunction with the Weyl criterion \cite[Theorem 2.1]{Weyl} immediately yields the following equidistribution result.
\begin{cor}\label{cor1}
    The Farey sequence $\left(\mathscr{F}_{Q,k}^{(m)}\right)_{Q\geq 1}$ is uniformly distributed modulo one.
\end{cor}
 Note that equidistribution modulo one is characterized as a qualitative asymptotic property; therefore, it is natural to study its corresponding quantitative aspect-namely, discrepancy, which is defined as follows:
For any $\alpha\in[0,1]$, let
$A(\alpha;{N})$ be the number of first $N$ terms of the sequence $(x_n)_{n=1}^{\infty}$ modulo one that do not exceed $\alpha$. Then the absolute discrepancy of the sequence $(x_n)_{n=1}^{\infty}$ is given by
\[D_{{N}}(x_1,\ldots, x_n)=\sup_{0\leq \alpha\leq 1}R_{{N}}(\alpha),\numberthis\label{D_1}\]
where
\[R_{{N}}(\alpha)=\left|\frac{A(\alpha;{N})}{{N}}-\alpha \right|.\numberthis\label{R_N}\]
 The classical work of Franel \cite{Franel} and Landau \cite{Landau} showed that the quantitative statement about the uniform distribution of Farey fractions and the Riemann hypothesis are equivalent. Denote $N(Q)=\#\mathscr{F}_Q$ and $\mathscr{F}_Q=\{\beta_1<\beta_2\cdots<\beta_{N(Q)}\}$. Then, Franel proved that 
the Riemann hypothesis is equivalent to the asymptotic formula
\[\sum_{i=1}^{N(Q)}R_{N(Q)}^2(\beta_i)=\BigO{Q^{-1+\epsilon}},\ \text{for all}\ \epsilon>0. \]
Note that by definition $R_{N(Q)}(\beta_j)=\left|\beta_j-\frac{j}{N(Q)}\right|.$
A similar version of Franel's result was proved by Landau \cite{Landau}, stating that the Riemann hypothesis is true if and only if, for all $\epsilon>0$,
\[\sum_{j=1}^{N(Q)}R_{N(Q)}(\beta_j)=\BigO{Q^{1/2+\epsilon}}.\]
We first derive an analogue of the above result for the sequence $\left(\mathscr{F}_{Q,k}^{(m)}\right)_{Q\geq 1}$. Denote $\mathcal{N}(Q,k,m)=\#\mathscr{F}_{Q,k}^{(m)}$ and $\mathscr{F}_{Q,k}^{(m)}=\{\gamma_1<\gamma_2\cdots<\gamma_{\mathcal{N}(Q,k,m)}\}$. 
\begin{thm}\label{thm2}
 Let $R_{\mathcal{N}(Q,k,m)}(\gamma_j)=\left|\gamma_j-\frac{j}{\mathcal{N}(Q,k,m)}\right|$. The generalized Riemann hypothesis (GRH) holds true if and only if, for all $\epsilon>0$, 
    \[\sum_{j=1}^{\mathcal{N}(Q,k,m)}R_{\mathcal{N}(Q,k,m)}(\gamma_j)=\BigOm{Q^{\frac{1}{2}+\epsilon}}. \]
\end{thm}
We next prove a closed-form formula for the second moment of the displacement of Farey fractions with $k$-free denominators. 
\begin{thm}\label{beta2}
   Let $R_{\mathcal{N}(Q,k,m)}(\gamma_j)=\left|\gamma_j-\frac{j}{\mathcal{N}(Q,k,m)}\right|$, and $M_q\left(x \right)=\sum_{\substack{n\leq x\\nq\equiv b\pmod{m}}}\mu(n)\mu_k(nq)^2$ for integers $b,m$ as in \eqref{f_k}, then, we have
   \[\sum_{j=1}^{\mathcal{N}(Q,k,m)}R_{\mathcal{N}(Q,k,m)}^2(\gamma_j)=\frac{1}{12\mathcal{N}(Q,k,m)}\left(\sum_{q_1, q_2\leq Q}M_{q_1}\left(\frac{Q}{q_1}\right)M_{q_2}\left(\frac{Q}{q_2}\right)\frac{(\gcd(q_1,q_2))^2}{q_1q_2}-1 \right).\]
 Moreover, the right-hand side above is
   bounded by \[\ll_{m} \left\{\begin{array}{cc}
   {\exp{\left(-c\frac{(\log Q)^{3/5}}{(\log\log Q)^{1/5}} \right)}},  & \mbox{unconditionally} ,\\
  {Q^{-1+\epsilon}},   & \mbox{on the GRH}.
\end{array}\right. \]
\end{thm}

The proof involves decomposing the weighted sum of Merten's function with congruence constraints in two different forms. To establish the bounds, we employ the Dirichlet hyperbola method alongside the non-trivial bounds for a twisted M\"{o}bius sum.

The foundational result on the equidistribution of irreducible fractions between $0$ and $1$, interpreted in terms of frequencies of certain almost periodic functions was given by Erd\H{o}s et al. \cite{Erdos}. Neville \cite{Neville} investigated the discrepancy of Farey fractions, proving that $D_{N(Q)}(\mathcal{F}_Q)\asymp \log Q/Q$. Subsequently, Niederreiter \cite{ Niederreiter} improved this result to $D_{N(Q)}(\mathcal{F}_Q)\asymp 1/Q$, and finally, Dress \cite{Dress} refined Niederreiter's result by showing that the discrepancy: $D_{N(Q)}(\mathcal{F}_Q)=1/Q$. 
Several authors \cite{MR2273359, MR2275343, Bchahal, MR3871604} have since studied the discrepancy of irreducible fractions in various contexts and with different congruence restrictions on denominators. In here,
we calculate the discrepancy for the sequence $\left(\mathscr{F}_{Q,k}^{(m)}\right)_{Q\ge 1}$ and establish bounds similar to those of Niederreiter, where the constants now depend on $m$. We prove the following analogue for $\mathscr{F}_{Q,k}^{(m)}$ , complementing Corollary \ref{cor1}.
\begin{thm}\label{thm1}
 For all $Q\geq 1$, we have
    \[D_{\mathcal{N}(Q,k)}\left(\mathscr{F}_{Q,k}^{(m)}\right)\asymp \frac{1}{Q}, \]
    where implied constants depend on $m$.
\end{thm}
The next results concern the fine-scale or local distribution of $\left(\mathscr{F}_{Q,k}^{(m)}\right)_{Q\ge 1}$ via $\nu$-level correlations. Let $\nu\geq 2$ be an integer and let $\mathcal{F}$ be a finite set of $\mathscr{N}$ elements in the unit interval $[0,1]$. The $\nu$-level correlation measure $\mathcal{S}_{\mathcal{F}}^{(\nu)}(\mathfrak{B})$ of a box $\mathfrak{B}\subset \mathbb{R}^{\nu-1}$ is defined as follows:
\[\frac{1}{\mathscr{N}}\#\left\{(x_1,\ldots,x_{\nu})\in \mathcal{F}^{\nu}: x_i\ \text{distinct},\ (x_1-x_2,\ldots,x_{\nu-1}-x_{\nu})\in \frac{1}{\mathscr{N}}\mathfrak{B}+\mathbb{Z}^{\nu-1} \right\}.\numberthis\label{v-tuple}\] 
The $\nu$-level correlation measure of an increasing sequence $(\mathcal{F}_n)_n$, for every box $\mathfrak{B}\subset \mathbb{R}^{\nu-1}$, is given (if it exists) by
\[\mathcal{S}^{(\nu)}(\mathfrak{B})=\lim_{n\to \infty }\mathcal{S}_{\mathcal{F}_n}^{(\nu)}(\mathfrak{B}).\]
The measure $\mathcal{S}^{(2)}$ is called the pair correlation measure. If
\[\mathcal{S}^{(\nu)}(\mathfrak{B})=\int_{\mathfrak{B}}g_{\nu}(x_1,\ldots,x_{\nu-1})dx_1\cdots dx_{\nu-1}, \numberthis\label{g(x)}\]
then $\mathfrak{g}_{\nu}$ is called the $\nu$-level correlation function of $(\mathcal{F}_n)_n$, and for $\nu=2$, it is called the pair correlation function. The $\nu$-level correlation is said to be Poissonian if $g_{\nu}(x)\equiv 1$. Poissonian behaviour of these fine-scale statistics can be seen as a pseudorandomness
property, since a sequence $(X_n)_{n\ge 1}$ of independent, identically distributed random variables with uniform distribution on $[0, 1)$ will almost surely have Poissonian correlations. The fine-scale statistics have been studied from mathematical point of view by Rudnick and Sarnak \cite{MR1628282} and then by Rudnick et al. \cite{MR1839285}, by studying the spacings between the fractional parts of the sequence $(\alpha n^d)_{n\geq 1}$, for an integer $d\geq 2$ and for a given irrational number $\alpha$. Subsequently, numerous authors \cite{MR1793613, MR2186997, MR2018926} studied the local spacing statistics of various sequences modulo one by investigated their correlation measure. For more on the study of the fine-scale statistics of sequences modulo one, we refer the reader to \cite{Aistleitner, Munsch, Shubin, Technau}.

The $\nu$-level correlation of Farey fractions was studied in \cite{BocaF}, where the authors prove the existence of the function $\mathcal{S}^{(\nu)}(\mathfrak{B})$ and derive an explicit expression for the pair correlation function of $(\mathcal{F}_Q)_Q$, which is given by
\[g(\lambda)=\frac{6}{\pi^2\lambda^2}\sum_{1\leq k<\frac{\pi^2\lambda}{3}}\phi(k)\log\frac{\pi^2\lambda}{3k}.\numberthis\label{g(lambda)}\]
Subsequently, several authors \cite{Siskaki, Bchahal, Lu, Lu1, Xiong, Zaharescu} have studied the pair correlation of Farey fractions with congruence constraints on the denominators. Further restrictions related to thin groups were examined by Lutsko in \cite{MR4458561}. In particular, in \cite{Chaubey}, the authors studied the pair correlation function for Farey fractions with square-free denominators. 

In the present article, we investigate if the $\nu$-level correlations of the sequence $\left(\mathscr{F}_{Q,k}^{(m)}\right)_{Q\ge 1}$ are Poissonian or not. Our primary aim is to compute the $\nu$-level correlation measure for all $\nu\ge 2$.  
To state our results, we first fix some notations and define certain one-to-one transformations. Let $A=(A_1,\ldots,A_{\nu-1}),\ B=(B_1,\ldots,B_{\nu-1})\in\mathbb{Z}_+^{\nu-1}$ such that $\gcd(A_j,B_j)=1$ for all $1\le j\le \nu-1$. For $\Lambda>0$ and $2\leq k, 1\leq m\in\mathbb{Z}$, we consider the one-to-one map defined as follows:
\[T_{A,B}(x,y)=\mathscr{C}(k,m)\left(\frac{B_1}{y(yA_1-xB_1)},\ldots,\frac{B_{\nu-1}}{y(yA_{\nu-1}-xB_{\nu-1})} \right),\numberthis\label{k63}\]
where 
\[\mathscr{C}(k,m)=\frac{1}{2\phi(m)L(k,\chi_0)}\prod_{p|m}\left(1-\frac{1}{p} \right)\prod_{\substack{p\\(p,m)=1}}\left(1-\frac{p^{k-1}-1}{p(p^k-1)}\right) \numberthis\label{C(k,m)}. \]
Here $\chi_0$ is the principal Dirichlet character modulo $m$. 
Let
\begin{align*}
  \Omega_{A,B,\Lambda,k}=&\left\{(x,y) : 0<x\leq y\leq 1,\ y\geq \frac{1}{\mathcal{C}{(\Lambda,k,m)}},\ 0< yA_j-xB_j\leq 1,\ 1\le j\le\nu-1\right\},  
\end{align*}
where
$\mathcal{C}{(\Lambda,k,m)}=\frac{\Lambda}{\mathscr{C}(k,m)}$.
We define another map $T$ on $\mathbb{R}^{\nu-1}$ and its inverse $T^{-1}$ as follows:
\[T(x_1,\ldots,x_{\nu-1})=(x_1-x_2, x_2-x_3,\ldots, x_{\nu-2}-x_{\nu-1},x_{\nu-1}),\]
\[T^{-1}(x_1,\ldots,x_{\nu-1})=(x_1+\cdots+x_{\nu-1}, x_2+\cdots+x_{\nu-1},\ldots, x_{\nu-2}+x_{\nu-1},x_{\nu-1}). \]
We also define a sublattice of $\mathbb{Z}^2$ as follows
\[\mathfrak{L}:=\left\{(x,y)\in\mathbb{Z}^2 : d_j^k|L_j^{(r)}(x,y),\ 0\leq j\leq\nu-1 \right\},  \]
where 
\[L_j(x,y)=\begin{cases}
 y,  &\mbox{if }\ j=0,\\
   yA_j-xB_j,  & \mbox{if}\ 1\leq j\leq\nu-1.
\end{cases} \]
Let $B=\{v_1, v_2\}$ be a basis of the lattice $\mathfrak{L}$. We denote the matrix generated by the basis $B$ by $M=\begin{pmatrix}
v_1^T & v_2^T 
\end{pmatrix}$. We define the map $\Phi : (\mathbb{Z}/m\mathbb{Z})^2\to(\mathbb{Z}/m\mathbb{Z})^2$ as $\Phi(x)=Mx$.
We are now ready to state our result on the $\nu$-level correlations. 
\begin{thm}\label{v correlation}
    Let $\nu\geq 2, k\ge 2$ be integers such that $\nu<k$. All $\nu$-level correlation measure of the sequence $\left(\mathscr{F}_{Q,k}^{(m)}\right)_{Q\geq 1}$ exist. For any box $\mathfrak{B}\subset(0,\Lambda)^{\nu-1}$, the $\nu$-level correlation measure is given by
\[ \mathcal{S}^{(\nu)}(\mathfrak{B})
   =\frac{1}{m^2\mathscr{C}(k,m)}\sum_{\substack{1\leq A_j\leq(\nu-1)\mathcal{C}^2{(\Lambda,k,m)}\\1\leq B_j\leq \nu\mathcal{C}^2{(\Lambda,k,m)}\\(A_j,B_j)=1\\ 1\leq j\leq\nu-1}}\mathscr{V}(m,k,A,B)\text{area}\left(\Omega_{A,B,\Lambda,k}\cap T_{A,B}^{-1}(T^{-1}\mathfrak{B}) \right), \]
 where 
 \[\mathscr{V}(m,k,A,B)=\sum_{\substack{r, d_j^k\geq 1 \\0\leq j\leq\nu-1\\ (m,rd_0d_1\cdots d_{\nu-1})=1}}\frac{\mu(r)\mu(d_0)\mu(d_1)\cdots\mu(d_{\nu-1})\#\mathcal{A}_{m,r}\#\ker\Phi}{r^2\det(\mathfrak{L})}, \]
 and $\mathcal{A}_{m,r}=\{(u,v)\in(\mathbb{Z}/m\mathbb{Z})^2 : L_j(ur,vr)\equiv b\pmod{m}, 0\leq j\leq\nu-1 \}$.
\end{thm}
Note that the summation $\mathscr{V}(m,k,A,B)$ is shown to be absolutely convergent.
\begin{remark}
    Recall that for the $\nu$-level correlation to be Poissonian, we must have $\mathcal{S}^{(\nu)}(\mathfrak{B})=\text{vol}(\mathfrak{B})$ for all boxes $\mathfrak{B}$. Using the above expression, we observe that the sequence $(\mathscr{F}_{Q,k}^{(m)})_{Q\ge 1}$, does not have Poissonian $\nu$-level correlations for any $\nu\ge2$. Indeed, let $\Lambda>0$ be a real number such that $m^2\Lambda^{-3(\nu-1)}> 2^{\nu-1}(\nu(\nu-1))^{\nu-1}(\mathscr{C}(k,m))^{-4\nu+3}\prod_{p}\left(1+\frac{1}{p^2}+\frac{2^{\nu}-1}{p^k}+\frac{2^{\nu}-1}{p^{k+1}} \right)$, and let $\mathfrak{B}=(0,\Lambda/2]^{\nu-1}$, then clearly $\mathcal{S}^{(\nu)}(\mathfrak{B})<\text{vol}(\mathfrak{B})$. 
    \end{remark}
    \begin{remark}
Particularly relevant to the work in this paper is the work of \cite{BocaF}, where the authors establish the $\nu$-level correlation for $(\mathscr{F}_Q)_{Q\ge 1}$. They reduce the problem of counting the \(\nu\)-tuple described in \eqref{v-tuple} to estimating an exponential sum.
This is achieved by expressing the Fourier series for the smooth real-valued function \(H\) with support contained in \(\mathfrak{B}\). Furthermore, they rewrite the exponential sum in terms of a M\"{o}bius sum and utilize the Poisson summation formula for the coefficients of the Fourier series. Given that the support of \(H\) is contained within \(\mathfrak{B}\), several changes of variables lead to the formulation of the \(\nu\)-level correlation measure.
However, in our case, the key difference lies in establishing estimates for weighted lattice point counting and deducing a formula for the exponential sum over Farey fractions whose denominators are \(k\)-free and lie within an arithmetic progression. 
In establishing estimates for weighted lattice point counting, the technical difficulty arises from the $k$-free conditions and arithmetic progressions. To overcome this difficulty, we first detect the $k$-free restrictions using M\"{o}bius sums and then consider a full-rank sublattice of $\mathbb{Z}^2$ subject to these divisibility restrictions. By a suitable transformation, among other steps, we reduce the problem of establishing estimates for weighted lattice point counting over a sublattice of $\mathbb{Z}^2$ to the corresponding problem over the full lattice $\mathbb{Z}^2$. We then handle this using \cite[Lemma 1]{Cobeli}.

\end{remark}
Our final result gives an explicit form for the pair \((\nu = 2)\) correlation measure of the sequence $\left(\mathscr{F}_{Q,k}^{(m)}\right)_{Q \geq 1}$.
\begin{thm}\label{main result}
The pair correlation function of the sequence $\left(\mathscr{F}_{Q,k}^{(m)}\right)_{Q\ge 1}$ exists and is given by
 \[\mathfrak{g}_{m,k}(\lambda)=\frac{6}{\lambda^2\pi^2\phi^2(m)}\sum_{1\leq n<\frac{\lambda}{\mathscr{C}(k,m)}}F_k(n)\log\left(\frac{\lambda}{n \mathscr{C}(k,m)} \right)\ \numberthis\label{g}\]
for any $\lambda\geq 0$, where $\mathscr{C}(k,m)$ is as in \eqref{C(k,m)}, and 
 \begin{align*}
    F_k(n)&=\sum_{\substack{\delta d_1d_2r=n\\(d_1d_2\delta,m)=1}}r\mu_k(\delta)^2\mu(d_1)\mu(d_2)\prod_{\substack{p\\(p,m)=1}}\left(1-\frac{\gcd(p^k,d_2\delta)}{p^{k-1}(p+1)} \right)\\&\times\prod_{\substack{p\\(p,m)=1}}\left(1-\frac{\gcd(p^k,d_1\delta)}{p^k+p^{k-1}-\gcd(p^k,d_2\delta)} \right).
 \end{align*}

 \end{thm}
\begin{figure}[ht]
\centering
\subfloat{
\includegraphics[width=8cm, height=5cm]{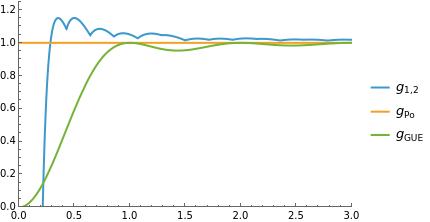}}
\subfloat{\includegraphics[width=8cm, height=5cm]{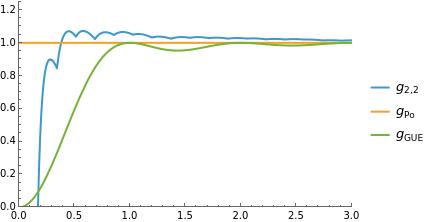}}
\caption{The graphs of pair correlation functions $\mathfrak{g}_{1,2}(\lambda), \mathfrak{g}_{2,2}(\lambda), g_{Po}(\lambda)\equiv 1$ and $g_{GUE}(\lambda)=1-\left(\frac{\sin \pi\lambda}{\pi\lambda} \right)^2$. }
\end{figure} 
\begin{remark}
    A key distinction in the argument presented for $k$-free versus square-free pair correlation measure is due to the following observation: If $n_1n_2$ is square-free then $(n_1,n_2)=1$. However, for $k\geq 3$, if $n_1n_2$ is $k$-free then $n_1$ and $n_2$ may or may not be coprime. As a result, the characteristic function for the $k$-free numbers $\mu_k(n_1n_2)^2$ cannot be separated when $k\geq 3$. This complexity necessitates a more careful analysis when establishing an asymptotic formula for counting weighted lattice points that satisfy specific coprimality conditions and $k$-free restrictions. As a result, the principal Dirichlet character yields the correlation measure, while for the non-principal character, we provide an estimate for the character sum twisted by a continuously differentiable function and the characteristic function for the \(k\)-free numbers. By applying this result, the sum over non-principal characters approaches zero as \(Q \to \infty\). 
\end{remark}


\begin{notation}
 For function $f,g :X\rightarrow\mathbb{R}$, defined on some set $X$, we write $f\ll g$ (or $\BigO{g(x)}$) to denote that there exists a constant $C>0$ such
that $|f(x)|\le C|g(x)|$ for all $x\in X$, with dependence on parameters denoted by subscripts. Moreover, let $f(x)\asymp g(x)$ denote that there exist constants $C_1$ and $C_2$ such that $C_1g(x)\leq f(x)\leq C_2g(x)$. The symbol $(a,b)=1$ denotes that $a$ and $b$ are coprime. We write $e(t)=\exp{(2\pi it)}$, and $\epsilon>0$ stands for an arbitrarily small positive real number. We denote for $x=(x_1,\ldots,x_n),\ y=(y_1,\ldots,y_n)\in\mathbb{R}^n,\ x\cdot y=x_1y_1+\cdots+x_ny_n$. The symbols $\zeta(s)$ and $L(s, \chi)$ denote the Riemann zeta function and the
Dirichlet L-function for the Dirichlet character $\chi$, respectively. We denote by $\lfloor x\rfloor$ the greatest integer less than or equal to $x$. For any set $\mathcal{U}$, $\#\mathcal{U}$ denotes the cardinality of the set $\mathcal{U}$.
\end{notation}

\begin{acknowledgments} 
The first author acknowledges support from the University Grants Commission, Department of Higher
Education, Government of India (GoI), under NTA Ref. no. 191620135578. The research conducted by the second and third authors is partially funded by core research grants CRG/2023/000804 and CRG/2023/001743 from the ANRF, formerly known as the Science and Engineering Research Board of the Department of Science and Technology (DST), GoI. 
\end{acknowledgments}
\section{Preliminaries}
In this section, we establish results that will be crucial in proving our main results. 

\subsection{Cardinality of the set $\mathscr{F}_{Q,k}^{(m)}$} We begin with estimating the cardinality $\mathcal{N}(Q,k,m)$ of the set $\mathscr{F}_{Q,k}^{(m)}$. 

\begin{prop}\label{prop1}
  Let $m$ and $b$ be positive integers. Then, we have 
   \[\mathcal{N}(Q,k,m)=Q^2\mathscr{C}(k,m)+\BigOm{Q^{\frac{2(2k-1)}{3k-2}}\exp{\left(-c\sqrt{\log Q} \right)}}, \]
   where $c>0$ is some constant and $\mathscr{C}(k,m)$ is as in \eqref{C(k,m)}.
\end{prop}
\begin{proof}
For fixed positive integers $m$ and $b$ with $(m,b)=1$, in view of the identity
\[\frac{1}{\phi(m)}\sum_{\chi\pmod{m}}\chi(n\bar{b})=\left\{\begin{array}{cc}
   1  & \mbox{if} \ n\equiv b\pmod{m},\\
  0   & \mbox{otherwise}, 
\end{array}\right. \numberthis\label{k67} \]
where $\bar{b}$ is such that $b\bar{b}\equiv 1\pmod{m}$, and by the definition of $\mathcal{N}(Q,k,m)$, we have 
\begin{align*}
\mathcal{N}(Q,k,m)
&=\frac{1}{\phi(m)}\sum_{\chi\pmod m}\bar{\chi}(b)\sum_{n\leq Q}\chi(n)\phi(n)\mu_k(n)^2.\numberthis\label{I7}
\end{align*}
The Dirichlet series of $\chi(n)\phi(n)\mu_k(n)^2$ is given by
\begin{align*}
    F(s)&=\sum_{n=1}^{\infty}\frac{\chi(n)\phi(n)\mu_k(n)^2}{n^s}
    =\frac{L(s-1,\chi)}{L(ks-k,\chi^k)}\prod_p\left(1+\frac{\chi(p^k)-\chi(p)p^{(k-1)(s-1)}}{p(p^{k(s-1)}-\chi(p^k))} \right),\numberthis\label{k19}
\end{align*}
which is absolutely convergent for $\Re(s)>2$ and has an analytic continuation to the half-plane $\Re(s)>1$ except for a simple pole at $s=2$ when $\chi=\chi_0$. For some fixed $\alpha=2+1/\log Q$ and the Dirichlet series $F(s)$, we apply Perron's formula (\cite{Tenenbaum}, Theorem 2, p. 132)
\[\sum_{\substack{n\leq Q}}\chi(n)\phi(n)\mu_k(n)^2=\frac{1}{2\pi i}\int_{\alpha-iT}^{\alpha+iT}F(s)\frac{Q^s}{s}ds+\BigO{R(T)}, \numberthis\label{k20}\]
    where
    \[R(T)\ll\frac{Q^\alpha}{T}\sum_{n=1}^\infty\frac{1}{n^{\alpha-1}|\log x/n|}\ll\frac{Q^2\log Q}{T}. \]
We use the Vinogradov-Korobov zero-free region for the Dirichlet $L$-functions modulo $m$ (see \cite{MR4732955}, Theorem 1.1) to estimate the integral in \eqref{k20}. We shift the line integral to the left of the line $\Re(s)=\alpha$, thereby replacing it by a rectangular contour with vertices $\alpha\pm iT$ and $\beta\pm iT$, where $\beta=1+1/k-c/k\sqrt{\log T}$. 

\textbf{Principal Dirichlet character:} We first consider the principal Dirichlet character $\chi=\chi_0$. Since the integrand in \eqref{k20} is holomorphic on and within this contour except for a pole at $s=2$. Thus, by Cauchy's residue theorem, we have
\[\frac{1}{2\pi i}\int_{\alpha-iT}^{\alpha+iT}F(s)\frac{Q^s}{s}ds=\frac{Q^2}{2L(k,\chi_0)}\prod_{p|m}\left(1-\frac{1}{p} \right)\prod_{\substack{p\\(p,m)=1}}\left(1-\frac{p^{k-1}-1}{p(p^k-1)}\right)+\sum_{j=1}^3I_j, \]
where $I_1$ and $I_3$ are integrals along the horizontal segments $[\alpha-iT,\beta-iT]$ and $[\beta+iT, \alpha+iT]$, respectively and $I_2$ is defined as the integral along the vertical segment $[\beta-iT, \beta+iT]$. In order to estimate the integrals $I_j$'s, we use the standard bounds for $\zeta(s)$ provided in \cite[page 47]{MR882550}, modulo multiplication by constant depending on $m$. Therefore,
\begin{align*}
    I_1, I_3&\ll_{m} \int_{\beta}^{\alpha}\frac{Q^{\sigma}|\zeta(\sigma-1+iT)|d\sigma}{|\sigma+iT||\zeta(k(\sigma-1+iT))|}
    \ll_{m} \frac{\log^2 T}{T}\left(\int_{\beta}^2Q^{\sigma}T^{1-\frac{\sigma}{2}}d\sigma + \int_2^{\alpha}Q^{\sigma}d\sigma\right)\ll_{m} \frac{Q^2\log^2 T}{T}.
\end{align*}
Next, we estimate the integral $I_2$ using the mean value estimate for $\zeta(s)/s$ \cite[Proposition 2.1]{Bittu}. 
\begin{align*}
    I_2&\ll_{m} Q^{\beta}\int_0^T\frac{|\zeta(\beta-1+it)|}{|\beta+it||\zeta(k\beta-k+ikt))|}dt
    \ll_{m} Q^{\beta}\log T\int_0^T\frac{|\zeta(\beta-1+it)|}{|\beta+it|}dt\ll_{m}Q^{\beta}T^{\frac{3}{2}-\beta}(\log T)^2.
\end{align*}
\textbf{Non-principal Dirichlet character:} We next consider the case for non-principal character $\chi$ (mod $m$). We continue with the contour defined above and use the bounds for Dirichlet $L$-function provided in (see \cite{MR551704}). Therefore
\begin{align*}
    I_1, I_3\ll_{m}& \int_{\beta}^{\alpha}\frac{Q^{\sigma}|L(\sigma-1+iT,\chi)|}{|\sigma+iT||L(k\sigma-k+ikT,\chi^k)|}d\sigma
    \ll_{m}\log T\left((\log T)^3\int_{\beta}^{\frac{3}{2}}\frac{Q^{\sigma}T^{\frac{127-73\sigma}{108}}}{T}d\sigma \right. \\ &+ \left. (\log T)^3\int_{\frac{3}{2}}^2\frac{Q^{\sigma}T^{\frac{35(2-\sigma)}{108}}}{T}d\sigma+\int_{2}^{\alpha}\frac{Q^{\sigma}\log T}{T}d\sigma \right)
    \ll_{m}\frac{Q^{\alpha}(\log T)^{2}}{T\log Q},
    \end{align*}
and using the bound $|L(k\sigma-k+ikt,\chi^k)|\gg_m 1/\log T$ (see \cite{Vaughan}) and \cite[Proposition 2.2]{Bittu}, we obtain
\begin{align*}
    I_2&\ll_{m} \int_{-T}^{T}\frac{|Q^{\beta+it}||L(\beta-1+it,\chi)|}{|\sigma+it||L(k\beta-k+ikt,\chi^k)|}dt
    \ll_{m} Q^{\beta}\log T\int_0^T\frac{|L(\beta-1+it,\chi)|}{|\beta-1+it|}dt
    \ll_{m} Q^{\beta}T^{\frac{3}{2}-\beta}\log^2 T.
\end{align*}
We choose optimally
\[T=Q^{\frac{2(k-1)}{3k-2}}\exp{\left(c\sqrt{\log Q} \right)}, \] 
By collecting all the above estimates, we obtain the required result.
\end{proof}

\subsection{Averages of weighted M\"{o}bius function}
\begin{prop}\label{mu}
  Let $b\in\mathbb{Z}$ and $d,l,m$ be positive integers. If $\xi_{d,k}(n)=\mu_k(nd)^2$ then for $x\geq 1$, we have
    \[\sum_{\substack{n\leq x\\(n,\ell)=1\\ n\equiv b\pmod{m}\\ }}\mu(n)\xi_{d,k}(n)\ll_{m} \left\{\begin{array}{cc}
   x\exp{\left(-c\frac{(\log x)^{3/5}}{(\log\log x)^{1/5}} \right)}\prod_{p|d}\left(\frac{\sqrt{p}}{\sqrt{p}-1}\right)\prod_{p|\ell}\left(\frac{\sqrt{p}}{\sqrt{p}-1} \right),  & \mbox{unconditionally} ,\\
  {x^{\frac{1}{2}+\epsilon}}\prod_{p|d}\left(1+\frac{1}{\sqrt{p}-1}\right)\prod_{p|\ell}\left(1-\frac{1}{\sqrt{p}} \right)^{-1},   & \mbox{on the GRH}.
\end{array}\right. \]
\end{prop}
\begin{proof}
It is easy to observe that if $d$ is not $k$-free, then the result follows trivially. Thus, we assume that $d$ is $k$-free. Using \eqref{k67}, we have
\begin{align*}
   \sum_{\substack{n\leq x\\(n,\ell)=1\\ n\equiv b\pmod{m} }}\mu(n)\xi_{d,k}(n)=\frac{1}{\phi(m)}\sum_{\chi}\chi(\bar{b})\sum_{\substack{n\leq x\\(n,\ell)=1 }}\chi(n)\mu(n)\xi_{d,k}(n). 
\end{align*}
 Note that $\xi_{d,k}(n)$ is a multiplicative function of $n$. Let $(n_1,n_2)=1$. If $n_1n_2d$ is $k$-free, then it is easy to observe that $n_1d$ and $n_2d$ are $k$-free. Conversely, suppose that $n_1d$ and $n_2d$ are $k$-free. We need to show that $n_1n_2d$ is also $k$-free. Suppose, for contradiction, that $n_1n_2d$ is not $k$-free; that is, there exists a prime $p$ such that $p^k|n_1n_2d$. Since $\gcd(n_1,n_2)=1$, it follows that either $p^k|n_1d$ or $p^k|n_2d$, which is a contradiction. This proves that $\xi_{d,k}(n)$ is a multiplicative function of $n$. The Dirichlet series of $\chi(n)\mu(n)\xi_{d,k}(n)$ is given by
    \begin{align*}
        F(s)&=\sum_{\substack{n=1\\ (n,\ell)=1}}^{\infty}\frac{\chi(n)\mu(n)\xi_{d,k}(n)}{n^s}=\prod_{\substack{p\\(p,\ell)=1}}\left(1-\frac{\chi(p)\xi_{d,k}(p)}{p^s} \right)\\
        &=\frac{1}{L(s,\chi)}\prod_{p|d}\left(1+\frac{\chi(p)(1-\xi_{d,k}(p))}{p^s}\left(1-\frac{\chi(p)}{p^s}\right)^{-1}\right)\prod_{p|\ell}\left(1-\frac{\chi(p)\xi_{d,k}(p)}{p^s} \right)^{-1}.
    \end{align*}
  In the last step, we used the fact that $\xi_{d,k}(p)=1$ if $(p,d)=1$.
The Dirichlet series $F(s)$ is absolutely convergent for $\Re(s)\geq\beta$, where $\beta=1-c/(\log T)^{2/3}(\log\log T)^{1/3}$.
    Employing Perron's formula (\cite{Tenenbaum}, Theorem 2, p. 132) for the Dirichlet series $F(s)$ with $\alpha=1+\frac{1}{\log x}$, we have
    \[\sum_{\substack{n\leq x\\ (n,\ell)=1}}\chi(n)\mu(n)\xi_{d,k}(n)=\frac{1}{2\pi i}\int_{\alpha-iT}^{\alpha+iT}F(s)\frac{x^s}{s}ds+\BigO{R(T)}, \]
    where
    \[R(T)\ll\frac{x^\alpha}{T}\sum_{n=1}^\infty\frac{1}{n^{\alpha}|\log x/n|}\ll\frac{x\log x}{T}. \numberthis\label{k9}\]
    In here, we bound the error term $R(T)$ as in Davenport (see \cite{Davenport}, p. 106-107). We next move the path of integration into a rectangular contour with line segments $[\alpha-iT, \alpha+iT],\ [\alpha+iT, \beta+iT],\ [\beta+iT, \beta-iT]$, and $[\beta-iT, \alpha-iT]$. For $\beta\leq\sigma\leq\alpha$, we have
    \begin{align*}
      \left|\prod_{p|\ell}\left(1-\frac{\chi(p)\xi_{d,k}(p)}{p^s} \right)^{-1} \right| 
      \leq\prod_{p|\ell}\left(\frac{\sqrt{p}}{\sqrt{p}-1} \right)\ \text{and}\ 
      \left|\prod_{p|d}\left(1+\frac{\chi(p)(1-\xi_{d,k}(p))}{p^s-\chi(p)}\right) \right|\leq \prod_{p|d}\left(1+\frac{1}{\sqrt{p}-1}\right).
    \end{align*}
    By Cauchy's theorem, we have
    \[\frac{1}{2\pi i}\int_{\alpha-iT}^{\alpha+iT}F(s)\frac{x^s}{s}ds=\frac{1}{2\pi i}\left(\int_{\alpha-iT}^{\beta-iT}+\int_{\beta-iT}^{\beta+iT}+\int_{\beta+iT}^{\alpha+iT} \right)F(s)\frac{x^s}{s}ds:=I_1+I_2+I_3. \]
    We first estimate the integrals $I_1$ and $I_3$:
    \begin{align*}
    I_1, I_3
    &\ll_{m} \frac{x\log T}{T\log x}\prod_{p|\ell}\left(1-\frac{1}{\sqrt{p}} \right)^{-1}\prod_{p|d}\left(1+\frac{1}{\sqrt{p}-1}\right). 
    \end{align*}
    The integral $I_2$ is estimated as
    \begin{align*}
     I_2 
     &\ll_{m} x^{\beta}(\log T)^2\prod_{p|d}\left(1+\frac{1}{\sqrt{p}-1}\right)\prod_{p|\ell}\left(1-\frac{1}{\sqrt{p}} \right)^{-1}.    \end{align*}
We collect all the above estimate and take $ T=\exp{\left(\frac{c(\log x)^{\frac{3}{5}}}{(\log\log x)^{\frac{1}{5}}}\right)}$. This completes the proof unconditionally. Assuming GRH, the Dirichlet series $F(s)$ is absolutely convergent for $\Re(s)>1/2$. By using Perron's formula with $\alpha=1+\frac{1}{\log x}$ and $\beta=\frac{1}{2}+\epsilon$, and proceeding in a similar manner as in the unconditional case, we obtain the proof under GRH. This completes the proof of Proposition \ref{mu}.
    \end{proof}

\begin{prop}\label{Prop1}
 Let $b\in\mathbb{Z}$, and let $d,l,m$ be positive integers. Suppose $d$ is $k$-free and $\xi_{d,k}(n)=\mu_k(nd)^2$. For $x\geq 2$, we have
   \begin{align*}
   \sum_{\substack{n\leq x\\(n,\ell)=1\\ n\equiv b\pmod{m}}}\frac{\xi_{d,k}(n)}{n}=&\mathcal{M}_{m,d,l}(x) +\BigOkmld{x^{\frac{-2(k-1)}{3k-2}}\exp{\left(-c\sqrt{\log x} \right)}},
   \end{align*}
   where $c$ is some positive constant and
   \begin{align*}
      \mathcal{M}_{m,d,l}(x)&= \left(\log x +\sum_{\substack{p|d\\(p,m)=1}}\left(\frac{-k\log p}{p^k-1}+\frac{\log p}{p-1}-\log p\sum_{\substack{j=1}}^{k-1}\frac{j\xi_{d,k}(p^j)}{p^j}  \left(1+\sum_{j=1}^{k-1}\frac{\xi_{d,k}(p^j)}{p^j} \right)^{-1} \right) \right. \\ &+ \left.  \log p\sum_{\substack{p|\ell\\(p,m)=1}}\sum_{j=1}^{k-1}\frac{j\xi_{d,k}(p^j)}{p^j}\left(1+\sum_{j=1}^{k-1}\frac{\xi_{d,k}(p^j)}{p^j} \right)^{-1}+\gamma-k\frac{L^{\prime}(k,\chi_0)}{L(k,\chi_0)}+\sum_{p|m}\frac{\log p}{p-1}   \right)\\&\times\frac{1}{L(k,\chi_0)}\prod_{p|m}\left(1-\frac{1}{p}\right)\prod_{\substack{p|d\\(p,m)=1}}\left(1-\frac{1}{p^{k}} \right)^{-1}\left(1-\frac{1}{p} \right)\prod_{\substack{p|d\\(p,m)=1}}\left(1+\sum_{j=1}^{k-1}\frac{\xi_{d,k}(p^j)}{p^j} \right)\\&\times\prod_{\substack{p|\ell\\(p,m)=1}}\left(1+\sum_{j=1}^{k-1}\frac{\xi_{d,k}(p^j)}{p^j} \right)^{-1}.
   \end{align*}
\end{prop}
\begin{proof}
We take into account \eqref{k67} to obtain
\begin{align*}
   \sum_{\substack{n\leq x\\(n,\ell)=1\\ n\equiv b\pmod{m} }}\frac{\xi_{d,k}(n)}{n}=\frac{1}{\phi(m)}\sum_{\chi}\chi(\bar{b})\sum_{\substack{n\leq x\\(n,\ell)=1 }}\frac{\chi(n)\xi_{d,k}(n)}{n}. 
\end{align*}
   The Dirichlet series of $\frac{\xi_{d,k}(n)}{n}$ is as follows:
   \begin{align*}
       F(s)&=\sum_{n=1}^{\infty}\frac{\xi_{d,k}(n)\chi(n)}{n^{s+1}}=\frac{L(s+1,\chi)}{L(ks+k,\chi^k)}\prod_{p|d}\left(1-\frac{\chi(p)}{p^{s+1}} \right)\left(1-\frac{\chi(p^k)}{p^{k(s+1)}}\right)^{-1}\\&\times\prod_{p|d}\left(1+\sum_{j=1}^{k-1}\frac{\xi_{d,k}(p^{j})\chi(p^j)}{p^{j(s+1)}} \right)\prod_{p|\ell}\left(1+\sum_{j=1}^{k-1}\frac{\xi_{d,k}(p^j)\chi(p^j)}{p^{j(s+1)}} \right)^{-1}.
   \end{align*}
Note that $F(s)$ is absolutely convergent for $\Re(s)>0$ and it can be analytically continued to the half-plane $\Re(s)=\beta>-1+\frac{1}{k}-\frac{c}{\sqrt{\log T}}$ except for a pole at $s=0$ when $\chi=\chi_0$. For some fixed $\alpha=1/\log x$ and the Dirichlet series $F(s)$, we apply Perron's formula (\cite{Tenenbaum}, Theorem 2, p. 132)
\[\sum_{\substack{n\leq x\\(n,\ell)=1}}\frac{\xi_{d,k}(n)\chi(n)}{n}=\frac{1}{2\pi i}\int_{\alpha-iT}^{\alpha+iT}F(s)\frac{x^s}{s}ds+\BigO{R(T)},\numberthis\label{k21} \]
    where
    \[R(T)\ll\frac{x^\alpha}{T}\sum_{n=1}^\infty\frac{1}{n^{\alpha+1}|\log x/n|}\ll\frac{\log x}{T}. \]
To estimate the integral on the right hand side of \eqref{k21}, we shift the line of integral into a rectangular contour with vertices $\alpha\pm iT$ and $\beta\pm iT$. We first consider the principal character. In this case, the integrand in \eqref{k21} has a pole of order $2$ at $s=0$. Denote
\begin{align*}
    Z(s)&=\frac{x^sL(s+1,\chi)}{sL(ks+k,\chi^k)}\prod_{p|d}\left(1-\frac{\chi(p)}{p^{s+1}} \right)\left(1-\frac{\chi(p^k)}{p^{k(s+1)}}\right)^{-1}\left(1+\sum_{j=1}^{k-1}\frac{\xi_{d,k}(p^j)\chi(p^j)}{p^{j(s+1)}} \right)\\&\times\prod_{p|\ell}\left(1+\sum_{j=1}^{k-1}\frac{\xi_{d,k}(p^j)\chi(p^j)}{p^{j(s+1)}} \right)^{-1} .
\end{align*}
 By Cauchy's residue theorem, we have
\[\frac{1}{2\pi i}\int_{\alpha-iT}^{\alpha+iT}F(s)\frac{x^s}{s}ds=\text{Res}_{s=0}Z(s)+\sum_{i=1}^{3}I_i, \]
where $I_1$ and $I_3$ are integrals along horizontal segments $[\alpha-iT,\beta-iT]$ and $[\beta+iT,\alpha+iT]$, respectively and $I_2$ is the integral along vertical segment $[\beta-iT,\beta+iT]$. The first term in the above identity is the residue of the second order pole of $Z(s)$ at $s=0$, and is given by $\mathcal{M}_{m,d,l}(x)$.  
We use the standard bounds for the Riemann zeta function $\zeta(s)$ \cite[page 47]{MR882550}, modulo multiplication by constants depending on $d$ and $\ell$
\begin{align*}
    I_1, I_3 
    &\ll_{m,d,\ell} \frac{\log T}{T}\left(\int_{\beta}^0x^{\sigma}T^{\frac{-\sigma}{2}}d\sigma+\int_{0}^{\alpha}x^{\sigma}\log T d\sigma \right)
    \ll_{m,d,\ell} \frac{(\log T)^2}{T\log x}.
\end{align*}
We next estimate the integral $I_2$ using \cite[Proposition 2.1]{Bittu} 
\begin{align*}
    I_2 \ll_{m,d,\ell} x^{\beta}\log T \int_{0}^{T}\frac{|\zeta(\beta+1+it)|}{|\beta+it|}dt
    \ll_{m,d,\ell} x^{\beta}T^{\frac{-1}{2}-\beta}(\log T)^2.
\end{align*}
We next consider the case for the non-principal character $\chi\ne\chi_0$. We continue with the contour defined above. Using the bounds for the Dirichlet $L$-function modulo $m$ and \cite[Proposition 2.2]{Bittu}, we obtain
\[I_1, I_3\ll_{m,d,\ell} \frac{(\log T)^2}{T\log x}\ \text{and}\ I_2\ll_{m,d,\ell} x^{\beta}T^{\frac{-1}{2}-\beta}(\log T)^2. \]
We collect all the above estimate and take optimally $T=x^{\frac{2(k-1)}{3k-2}}\exp\left(c\sqrt{\log x} \right)$. This completes the proof of Proposition \ref{Prop1}.
\end{proof}

\begin{prop}\label{mu product}
 For $x\geq 1$, we have
    \[\sum_{n\leq x}\mu_k(n)^2\prod_{p|n}\left(1-\frac{1}{\sqrt{p}} \right)^{-1}=\frac{x}{\zeta(k)}\prod_{p}\left(1+\frac{p^{k-1}-1}{(\sqrt{p}-1)(p^k-1)} \right)+\BigO{x^{\frac{k}{3k-2}}\exp{\left(-c\sqrt{\log x} \right)}}. \]
\end{prop}
\begin{proof}
The proof is similar to Proposition \ref{mu}. 
\end{proof}
\subsection{Weighted $k$-free Farey sums} We next expand Farey sums for Farey fractions in $\mathscr{F}_{Q,k}^{(m)}$ using the M\"{o}bius and the $k$-free M\"{o}bius function. This is helpful in the reduction and estimation of the exponential sums for $\mathscr{F}_{Q,k}^{(m)}$ required for the proof of Theorem \ref{thm2}. 
\begin{lem}\label{f(gamma)}
    Assume that $f$ is any complex-valued function defined on the interval $[0,1]$, and let $\gamma_i\in\mathscr{F}_{Q,k}^{(m)}\ \text{for}\ 1\leq i\leq \mathcal{N}(Q,k,m)$. Then, we have
    \[\sum_{j=1}^{\mathcal{N}(Q,k,m)}f(\gamma_j)=\sum_{q\leq Q}M_q\left(\frac{Q}{q}\right)\sum_{a\leq q}f\left(\frac{a}{q}\right), \]
    where \[ M_q(x)=\sum_{\substack{n\leq x\\qn\equiv b\pmod{m}}}\mu(n)\mu_k(qn)^2 .\]
\end{lem}
\begin{proof}
    We write
    \begin{align*}
     \sum_{j=1}^{\mathcal{N}(Q,k,m)}f(\gamma_j)&=\sum_{\substack{q\leq Q\\q\equiv b\pmod{m}}}\mu_k(q)^2\sum_{\substack{a\leq q\\(a,q)=1}}f\left(\frac{a}{q}\right)\\
     &=\sum_{\substack{q\leq Q\\q\equiv b\pmod{m}}}\mu_k(q)^2\sum_{\substack{a\leq q}}f\left(\frac{a}{q}\right)\sum_{\substack{d|a\\d|q}}\mu(d)
     =\sum_{d\leq Q}\mu(d)\sum_{\substack{q\leq Q\\q\equiv b\pmod{m}\\ d|q}}\mu_k(q)^2\sum_{\substack{\substack{a\leq q\\ d|a}}}f\left(\frac{a}{q}\right)\\
     &=\sum_{d\leq Q}\mu(d)\sum_{\substack{q\leq\frac{Q}{d}\\qd\equiv b\pmod{m}}}\mu_k(qd)^2\sum_{a\leq q}f\left(\frac{a}{q}\right)=\sum_{q\leq Q}M_q\left(\frac{Q}{q}\right)\sum_{a\leq q}f\left(\frac{a}{q}\right).
    \end{align*}
    \end{proof}

\begin{lem}
\label{prop7}
    Let $f(x)=x-\lfloor x\rfloor-\frac{1}{2}$ and $M_n$ as in Lemma \ref{f(gamma)}. For any real number $u\in[0,1]$ lying between two successive Farey fractions $\gamma_v$ and $\gamma_{v+1}$ in $\mathscr{F}_{Q,k}^{(m)}$, we have 
    
    \[\sum_{j=1}^{\mathcal{N}(Q,k,m)}f(u+\gamma_j)=\mathcal{N}(Q,k,m)u-v-\frac{1}{2}. \]
\end{lem}
\begin{proof}
  Similar to the proof of Lemma \ref{f(gamma)}, we can write
  \begin{align*}
      \sum_{j=1}^{\mathcal{N}(Q,k,m)}f(u+\gamma_j)&=\sum_{n\leq Q}f(nu)M_n\left(\frac{Q}{n}\right)
      =\sum_{n\leq Q}\left(nu-\lfloor nu\rfloor-\frac{1}{2}\right)M_n\left(\frac{Q}{n}\right)
      \\&=\mathcal{N}(Q,k,m)u-\sum_{n\leq Q}\lfloor nu\rfloor M_n\left(\frac{Q}{n}\right)-\frac{1}{2}.
  \end{align*}
  Note that the sum on the right-hand side of the above equation counts the number of fractions in $\mathscr{F}_{Q,k}^{(m)}$ less than or equal to $u$. Therefore between $\gamma_v$ and $\gamma_{v+1}$ the above sum is equal to $\mathcal{N}(Q,k,m)u-v-\frac{1}{2}$.
\end{proof}

\subsection{Weighted lattice point counting}
The results in this section are essential in the derivation of the $\nu$-level correlation function and the explicit expression for the pair correlation function. These results involve a variation of the following result of \cite{Cobeli} for lattice point counting in bounded domains. 
\begin{lem}[Lemma 1, \cite{Cobeli}]\label{Zaharescu lemma}
    Let $\Omega\subset [1, R]^2$ be a bounded region and assume that $f$ is a continuously differentiable function on $\Omega,$ then
    \begin{align*}
       \sum_{\substack{(a,b)\in \Omega\cap\mathbb{Z}^2\\}}f(a,b)&=\iint_{\Omega}f(x,y)dxdy+ \BigO{\left(\left\|{\frac{\partial f}{\partial x}}\right\|_{\infty}+\left\|{\frac{\partial f}{\partial y}}\right\|_{\infty}\right)\text{Area}(\Omega)\log R}\\&+\BigO{\left\|f \right\|_{\infty}(R+\text{length}(\delta\Omega)\log R)}. 
    \end{align*}
\end{lem}
We prove weighted version of the above result consisting of coprimality constraints and twisted by M\"{o}bius functions. This involves several significant modifications. In particular, we need to deal with
the extra M\"{o}bius twists in the sums and handle the extra coprimality conditions by carefully reducing the regions using several change of variables.

\begin{lem}\label{key lemma}
Let $R>1$ be a real number and let $\delta_1$ and $\delta_2$ be $k$-free numbers. Let $\Omega\subset [1, R]^2$ be a bounded region and assume that $f$ is a continuously differentiable function on $\Omega.$ For any positive integers $r_1$ and $r_2$, we have 
   \[\sum_{\substack{(a,b)\in \Omega\cap\mathbb{Z}^2\\(a,r_1)=(b,r_2)=(a,b)=1}}\mu_k(a\delta_1)^2\mu_k(b\delta_2)^2f(a,b)=\frac{6P_{r_1,r_2}^{k}(\delta_1,\delta_2)}{\pi^2}\iint_{\Omega}f(x,y)dxdy+E(r_1,r_2),\numberthis\label{2.3}\] 
   where
   \begin{align*}
       P_{r_1,r_2}^{k}(\delta_1,\delta_2)=&\frac{\phi(r_1)\phi(r_2)}{r_1r_2}\prod_{p|r_1r_2}\left(1-\frac{1}{p^2}\right)^{-1} \prod_{\substack{p|r_1\\(p,r_2)=1}}\left(1-\frac{\gcd(p^k,\delta_2)}{p^k} \right)\prod_{\substack{p|r_2\\(p,r_1)=1}}\left(1-\frac{\gcd(p^k,\delta_1)}{p^{k}} \right) \\&\times\prod_{\substack{p\\(p,r_1r_2)=1}}\left(1-\frac{\gcd(p^k,\delta_2)}{p^{k-1}(p+1)} \right)\left(1-\frac{\gcd(p^k,\delta_1)}{p^{k-1}(p+1)}\left(1-\frac{\gcd(p^k,\delta_2)}{p^{k-1}(p+1)} \right)^{-1} \right)
   \end{align*} 
   and
   \[ E(r_1,r_2)\ll_k \left(\tau\left(r_1 \right)\left\|{\frac{\partial f}{\partial x}}\right\|_{\infty}+\tau\left(r_2 \right)\left\|{\frac{\partial f}{\partial y}}\right\|_{\infty}\right)\text{Area}(\Omega)R^{\frac{1}{k}}\log^2R+R^{1+\frac{1}{k}}\log^2R\left\|f \right\|_{\infty}\left(\tau(r_1)+\tau(r_2) \right).\]
\end{lem}
\begin{proof}
We have
    \begin{align*}
        M:&=\sum_{\substack{(a,b)\in\Omega\cap\mathbb{Z}^2\\(a,r_1)=(b,r_2)=(a,b)=1}}\mu_k(a\delta_1)^2\mu_k(b\delta_2)^2f(a,b)
        =\sum_{\substack{d_1^k\leq R\delta_1\\ d_2^k\leq R\delta_2}}\mu(d_1)\mu(d_2)\sum_{\substack{(a,b)\in\Omega\cap\mathbb{Z}^2\\(a,r_1)=(b,r_2)=(a,b)=1\\d_1^k|a\delta_1, d_2^k|b\delta_2}}f(a,b).
    \end{align*}
 Using the fact that $a|bc$ if and only if $\frac{a}{\gcd(a,c)}|b$, the above identity can be expressed as   
 \begin{align*}
   M&=\sum_{\substack{d_1^k\leq R\delta_1\\ d_2^k\leq R\delta_2}}\mu(d_1)\mu(d_2)\sum_{\substack{(a,b)\in\Omega\cap\mathbb{Z}^2\\(a,r_1)=(b,r_2)=(a,b)=1\\\frac{d_1^k}{\gcd(d_1^k,\delta_1)}|a, \frac{d_2^k}{\gcd(d_2^k,\delta_2)}|b}}f(a,b) \\
   &=\sum_{\substack{d_1^k\leq R\delta_1, d_2^k\leq R\delta_2\\ \left(\frac{d_1^k}{\gcd(d_1^k,\delta_1)},r_1\right)=\left(\frac{d_2^k}{\gcd(d_2^k,\delta_2)},r_2\right)=1\\\left(\frac{d_1^k}{\gcd(d_1^k,\delta_1)},\frac{d_2^k}{\gcd(d_2^k,\delta_2)}\right)=1}}\mu(d_1)\mu(d_2)\sum_{\substack{(a_1,b_1)\in\Omega_{(d_1,d_2)}\cap\mathbb{Z}^2\\\left(a_1,\frac{r_1d_2^k}{\gcd(d_2^k,\delta_2)}\right)=1=\left(b_1,\frac{r_2d_1^k}{\gcd(d_1^k,\delta_1)}\right)\\(a_1,b_1)=1}}g(a_1,b_1),\numberthis\label{k56}
 \end{align*}
 where $g(a_1,b_1)=f\left(\frac{d_1^ka_1}{\gcd(d_1^k,\delta_1)},\frac{d_2^kb_1}{\gcd(d_2^k,\delta_2)}\right)$ and
 \[\Omega_{(d_1,d_2)}=\left\{(x,y)\ :\ x\in\frac{\gcd(d_1^k,\delta_1)}{d_1^k}[1, R],\ y\in\frac{\gcd(d_2^k,\delta_2)}{d_2^k}[1, R] \right\}.\]
We first estimate the inner sum in the above identity. Therefore
\begin{align*}
M^{(1)}:&=\sum_{\substack{(a_1,b_1)\in\Omega_{(d_1,d_2)}\cap\mathbb{Z}^2\\\left(a_1,\frac{r_1d_2^k}{\gcd(d_2^k,\delta_2)}\right)=1=\left(b_1,\frac{r_2d_1^k}{\gcd(d_1^k,\delta_1)}\right)\\(a_1,b_1)=1}}g(a_1,b_1)\\
&=\sum_{\substack{d\leq R\min\left(\frac{\gcd(d_1^k,\delta_1)}{d_1^k},\frac{\gcd(d_2^k,\delta_2)}{d_2^k} \right)\\\left(d,\frac{r_1r_2d_1^kd_2^k}{\gcd(d_1^k,\delta_1)\gcd(d_2^k,\delta_2)}\right)=1}} \mu(d)\sum_{\substack{(a_2,b_2)\in\frac{1}{d}\Omega_{(d_1,d_2)}\cap\mathbb{Z}^2\\\left(a_2,\frac{r_1d_2^k}{\gcd(d_2^k,\delta_2)}\right)=1=\left(b_2,\frac{r_2d_1^k}{\gcd(d_1^k,\delta_1)}\right)}}g(da_2,db_2)\\
&=\sum_{\substack{d\leq R\min\left(\frac{\gcd(d_1^k,\delta_1)}{d_1^k},\frac{\gcd(d_2^k,\delta_2)}{d_2^k} \right)\\\left(d,\frac{r_1r_2d_1^kd_2^k}{\gcd(d_1^k,\delta_1)\gcd(d_2^k,\delta_2)}\right)=1}} \mu(d)\sum_{\substack{(a_2,b_2)\in\frac{1}{d}\Omega_{(d_1,d_2)}\cap\mathbb{Z}^2}}g(da_2,db_2)\sum_{\substack{s|a_2\\s|\frac{r_1d_2^k}{\gcd(d_2^k,\delta_2)}}}\mu(s)\sum_{\substack{t|b_2\\t|\frac{r_2d_1^k}{\gcd(d_1^k,\delta_1)}}}\mu(t)\\
&=\sum_{\substack{d\leq R\min\left(\frac{\gcd(d_1^k,\delta_1)}{d_1^k},\frac{\gcd(d_2^k,\delta_2)}{d_2^k} \right)\\\left(d,\frac{r_1r_2d_1^kd_2^k}{\gcd(d_1^k,\delta_1)\gcd(d_2^k,\delta_2)}\right)=1}} \mu(d)\sum_{\substack{s|\frac{r_1d_2^k}{\gcd(d_2^k,\delta_2)}}}\mu(s)\sum_{\substack{t|\frac{r_2d_1^k}{\gcd(d_1^k,\delta_1)}}}\mu(t)\sum_{\substack{(a_3,b_3)\in\Gamma\cap\mathbb{Z}^2}}h(a_3,b_3),\numberthis\label{k52}
\end{align*}
where $h(a_3,b_3)=f\left(\frac{dsd_1^ka_3}{\gcd(d_1^k,\delta_1)},\frac{dtd_2^kb_3}{\gcd(d_2^k,\delta_2)}\right)$ and
\[\Gamma=\left\{(x,y)\ :\ x\in\frac{\gcd(d_1^k,\delta_1)}{dsd_1^k}[1, R],\ y\in\frac{\gcd(d_2^k,\delta_2)}{dtd_2^k}[1, R] \right\}. \]
We use Lemma \ref{Zaharescu lemma} to estimate the innermost sum in \eqref{k52}
\begin{align*}
\sum_{\substack{(a_3,b_3)\in\Gamma\cap\mathbb{Z}^2}}&h(a_3,b_3)=\iint_{\Gamma}h(x,y)dxdy\\&+\BigO{\left(\left\|{\frac{\partial h}{\partial x}}\right\|_{\infty}+\left\|{\frac{\partial h}{\partial y}}\right\|_{\infty}\right)\text{Area}(\Gamma)+\left\|h \right\|_{\infty}(1+\text{length}(\partial \Gamma))}\\
&=\frac{\gcd(d_1^k,\delta_1)\gcd(d_2^k,\delta_2)}{std^2d_1^kd_2^k}\iint_{\Omega}f(x,y)dxdy+\BigO{\left\|f \right\|_{\infty}\frac{R}{d}\left(\frac{\gcd(d_1^k,\delta_1)}{sd_1^k}+\frac{\gcd(d_2^k,\delta_2)}{td_2^k}\right)}\\&+\BigO{\left(\frac{\gcd(d_2^k,\delta_2)}{dtd_2^k}\left\|{\frac{\partial f}{\partial x}}\right\|_{\infty}+\frac{\gcd(d_1^k,\delta_1)}{dsd_1^k}\left\|{\frac{\partial f}{\partial y}}\right\|_{\infty}\right)\text{Area}(\Omega)}.
\end{align*}
By invoking the above estimate into \eqref{k52}, we obtain
\begin{align*}
    M^{(1)}&=\frac{\gcd(d_1^k,\delta_1)\gcd(d_2^k,\delta_2)}{d_1^kd_2^k}\sum_{\substack{d\leq R\min\left(\frac{\gcd(d_1^k,\delta_1)}{d_1^k},\frac{\gcd(d_2^k,\delta_2)}{d_2^k} \right)\\\left(d,\frac{r_1r_2d_1^kd_2^k}{\gcd(d_1^k,\delta_1)\gcd(d_2^k,\delta_2)}\right)=1}} \frac{\mu(d)}{d^2}\sum_{\substack{s|\frac{r_1d_2^k}{\gcd(d_2^k,\delta_2)}, t|\frac{r_2d_1^k}{\gcd(d_1^k,\delta_1)}}}\frac{\mu(s)\mu(t)}{st}\\&\times\iint_{\Omega}f(x,y)dxdy+\BigO{\left(\tau\left(\frac{r_1d_2^k}{\gcd(d_2^k,\delta_2)} \right)\frac{\gcd(d_2^k,\delta_2)}{d_2^k}\left\|{\frac{\partial f}{\partial x}}\right\|_{\infty}\right)\text{Area}(\Omega)\log^2R}\\&+\BigO{\left(\tau\left(\frac{r_2d_1^k}{\gcd(d_1^k,\delta_1)} \right)\frac{\gcd(d_1^k,\delta_1)}{d_1^k}\left\|{\frac{\partial f}{\partial y}}\right\|_{\infty}\right)\text{Area}(\Omega)\log^2R}\\&+\BigO{R\log^2R\left\|f \right\|_{\infty}\left(\tau\left(\frac{r_2d_1^k}{\gcd(d_1^k,\delta_1)} \right)\frac{\gcd(d_1^k,\delta_1)}{d_1^k}+\tau\left(\frac{r_1d_2^k}{\gcd(d_2^k,\delta_2)} \right)\frac{\gcd(d_2^k,\delta_2)}{d_2^k}\right)}.\numberthis\label{k53}
\end{align*}
We next estimate the summation in \eqref{k53}
\begin{align*}
    M^{(11)}:&=\sum_{\substack{d\leq R\min\left(\frac{\gcd(d_1^k,\delta_1)}{d_1^k},\frac{\gcd(d_2^k,\delta_2)}{d_2^k} \right)\\\left(d,\frac{r_1r_2d_1^kd_2^k}{\gcd(d_1^k,\delta_1)\gcd(d_2^k,\delta_2)}\right)=1}} \frac{\mu(d)}{d^2}\sum_{\substack{s|\frac{r_1d_2^k}{\gcd(d_2^k,\delta_2)}, t|\frac{r_2d_1^k}{\gcd(d_1^k,\delta_1)}}}\frac{\mu(s)\mu(t)}{st}\\
    &=\sum_{\substack{d=1\\\left(d,\frac{r_1r_2d_1^kd_2^k}{\gcd(d_1^k,\delta_1)\gcd(d_2^k,\delta_2)}\right)=1}}^{\infty} \frac{\mu(d)}{d^2}\prod_{p|\frac{r_1d_2^k}{\gcd(d_2^k,\delta_2)}}\left(1-\frac{1}{p} \right)\prod_{p|\frac{r_2d_1^k}{\gcd(d_1^k,\delta_1)}}\left(1-\frac{1}{p} \right)+\BigO{\frac{\max(d_1^k, d_2^k)}{R}}\\
    &=\prod_{p|\frac{r_1d_2^k}{\gcd(d_2^k,\delta_2)}}\left(1-\frac{1}{p} \right)\prod_{p|\frac{r_2d_1^k}{\gcd(d_1^k,\delta_1)}}\left(1-\frac{1}{p} \right)\prod_{\substack{p\\\left(p,\frac{r_1r_2d_1^kd_2^k}{\gcd(d_1^k,\delta_1)\gcd(d_2^k,\delta_2)}\right)=1}}\left(1-\frac{1}{p^2}\right)+\BigO{\frac{\max(d_1^k, d_2^k)}{R}}. \numberthis\label{k54}
\end{align*}
The above estimate in conjunction with \eqref{k53} and \eqref{k56} gives
\begin{align*}
    &M=\frac{1}{\zeta(2)}\iint_{\Omega}f(x,y)dxdy\sum_{\substack{d_1^k\leq R\delta_1, d_2^k\leq R\delta_2\\ \left(\frac{d_1^k}{\gcd(d_1^k,\delta_1)},r_1\right)=\left(\frac{d_2^k}{\gcd(d_2^k,\delta_2)},r_2\right)=1\\\left(\frac{d_1^k}{\gcd(d_1^k,\delta_1)},\frac{d_2^k}{\gcd(d_2^k,\delta_2)}\right)=1}}\frac{\mu(d_1)\mu(d_2)\gcd(d_1^k,\delta_1)\gcd(d_2^k,\delta_2)}{d_1^kd_2^k}\\&\times\prod_{p|\frac{r_1d_2^k}{\gcd(d_2^k,\delta_2)}}\left(1-\frac{1}{p} \right)\prod_{p|\frac{r_2d_1^k}{\gcd(d_1^k,\delta_1)}}\left(1-\frac{1}{p} \right) \prod_{p|\frac{r_1r_2d_1^kd_2^k}{\gcd(d_1^k,\delta_1)\gcd(d_2^k,\delta_2)}}\left(1-\frac{1}{p^2} \right)^{-1}
    \\&+\BigO{\left(\tau\left(r_1 \right)\left\|{\frac{\partial f}{\partial x}}\right\|_{\infty}+\tau\left(r_2 \right)\left\|{\frac{\partial f}{\partial y}}\right\|_{\infty}\right)\text{Area}(\Omega)R^{\frac{1}{k}}\log^2R} +\BigO{R^{1+\frac{1}{k}}\log^2R\left\|f \right\|_{\infty}\left(\tau(r_1)+\tau(r_2) \right)}\\
    &=\frac{1}{\zeta(2)}\prod_{p|r_1r_2}\left(1-\frac{1}{p^2} \right)^{-1}\prod_{p|r_1}\left(1-\frac{1}{p}\right)\prod_{p|r_2}\left(1-\frac{1}{p}\right)\iint_{\Omega}f(x,y)dxdy\sum_{\substack{d_1, d_2=1\\ \left(\frac{d_1^k}{\gcd(d_1^k,\delta_1)},r_1\right)=\left(\frac{d_2^k}{\gcd(d_2^k,\delta_2)},r_2\right)=1\\\left(\frac{d_1^k}{\gcd(d_1^k,\delta_1)},\frac{d_2^k}{\gcd(d_2^k,\delta_2)}\right)=1}}^{\infty}\\&\times\frac{\mu(d_1)\mu(d_2)\gcd(d_1^k,\delta_1)\gcd(d_2^k,\delta_2)}{d_1^kd_2^k}\prod_{\substack{p|\frac{d_1^k}{\gcd(d_1^k,\delta_1)}\\(p,r_2)=1}}\left(1-\frac{1}{p} \right)\prod_{\substack{p|\frac{d_2^k}{\gcd(d_2^k,\delta_2)}\\(p,r_1)=1}}\left(1-\frac{1}{p} \right)\\&\times \prod_{\substack{p|\frac{d_1^kd_2^k}{\gcd(d_1^k,\delta_1)\gcd(d_2^k,\delta_2)}\\(p,r_1r_2)=1}}\left(1-\frac{1}{p^2} \right)^{-1}+E(r_1,r_2).\numberthis\label{k55}
\end{align*}
We next estimate the summation in \eqref{k55}; let us denote it by $M_{r_1,r_2}$. Since $\delta_1$ and $\delta_2$ are $k$-free, it follows that $\left(\frac{d_1^k}{\gcd(d_1^k,\delta_1)},\frac{d_2^k}{\gcd(d_2^k,\delta_2)}\right)=1$ if and only if $(d_1,d_2)=1$; that $p|\frac{d_1^k}{\gcd(d_1^k,\delta_1)}$ if and only if $p|d_1$; and that $p|\frac{d_2^k}{\gcd(d_2^k,\delta_2)}$ if and only if $p|d_2$. Therefore the sum in \eqref{k55} becomes

\begin{align*}
   M_{r_1, r_2}:
   &=\sum_{\substack{d_1=1\\(d_1,r_1)=1}}^{\infty}\frac{\mu(d_1)\gcd(d_1^k,\delta_1)}{d_1^k}\prod_{\substack{p|d_1\\(p,r_2)=1}}\left(1-\frac{1}{p} \right)\prod_{\substack{p|d_1\\(p,r_1r_2)=1}}\left(1-\frac{1}{p^2} \right)^{-1}\\&\times\sum_{\substack{d_2=1\\(d_2,r_2d_1)=1}}^{\infty}\frac{\mu(d_2)\gcd(d_2^k,\delta_2)}{d_2^k}\prod_{\substack{p|d_2\\(p,r_1)=1}}\left(1-\frac{1}{p} \right)\prod_{\substack{p|d_2\\\left(p,r_1r_2d_1\right)=1}}\left(1-\frac{1}{p^2} \right)^{-1}+\BigO{\frac{1}{R^{1-\frac{1}{k}}}}\\
   &=\prod_{\substack{p\\(p,r_1r_2)=1}}\left(1-\frac{\gcd(p^k,\delta_2)}{p^{k-1}(p+1)} \right)\left(1-\frac{\gcd(p^k,\delta_1)}{p^{k-1}(p+1)}\left(1-\frac{\gcd(p^k,\delta_2)}{p^{k-1}(p+1)} \right)^{-1} \right)\\&\times\prod_{\substack{p|r_2\\(p,r_1)=1}}\left(1-\frac{\gcd(p^k,\delta_1)}{p^{k}} \right)\prod_{\substack{p|r_1\\(p,r_2)=1}}\left(1-\frac{\gcd(p^k,\delta_2)}{p^k} \right)+\BigO{\frac{1}{R^{1-\frac{1}{k}}}}.
\end{align*}
Inserting the above estimate into \eqref{k55} completes the proof of Lemma \ref{key lemma}.
\end{proof}

\subsection{Sublattice of $\mathbb{Z}^2$}
Let $d_j,\ 0\leq j\leq\nu-1$ be positive integers. Let $\mathfrak{L}$ be a sublattice of $\mathbb{Z}^2$ defined by
\[\mathfrak{L}:=\left\{(x,y)\in\mathbb{Z}^2 : d_j^k|L_j^{(r)}(x,y),\ 0\leq j\leq\nu-1 \right\}, \numberthis\label{lattice 1} \]
where $L_j^{(r)}(x,y)=r(yA_j-xB_j), 0\leq j\leq\nu-1$ and $r, A_j, B_j$ are positive integers.

\begin{prop}\label{prop-9}
  Let $\mathfrak{L}$ be a sublattice of $\mathbb{Z}^2$ defined in \eqref{lattice 1}. Then $\mathfrak{L}$ has rank $2$. 
\end{prop}
\begin{proof}
    Denote $D=\lcm(d_0^k,d_1^k,\ldots,d_{\nu-1}^k)$. Let $D\mathbb{Z}^2=\{(Dx,Dy) : x,y\in\mathbb{Z} \}$ be a sublattice of $\mathbb{Z}^2$. Let $(w_1,w_2)\in D\mathbb{Z}^2$ and write $w_1=Dx, w_2=Dy$, where $x,y\in\mathbb{Z}$. Then $L_j^{(r)}(w_1,w_2)=Dr(yA_j-xB_j), 0\leq j\leq\nu-1$. Since $d_j^k|D$ for every $0\leq j\leq\nu-1$, it follows that $d_j^k|L_j^{(r)}(w_1,w_2)$ for every $0\leq j\leq\nu-1$. Thus $(w_1,w_2)\in\mathfrak{L}$ and consequently $D\mathbb{Z}^2\subseteq\mathfrak{L}\subseteq\mathbb{Z}^2$. Therefore, $[\mathbb{Z}^2:\mathfrak{L}]\leq [\mathbb{Z}^2:D\mathbb{Z}^2]=D^2<\infty$. Hence, $\mathfrak{L}$ has finite index in $\mathbb{Z}^2$ and therefore $\mathfrak{L}$ is a rank-$2$ sublattice of $\mathbb{Z}^2$.
\end{proof}

Let $B=\{v_1, v_2\}$ be a basis of the lattice $\mathfrak{L}$. We denote the matrix generated by the basis $B$ by $M=\begin{pmatrix}
v_1^T & v_2^T 
\end{pmatrix}$. This implies that $\mathfrak{L}=M\mathbb{Z}^2$ and $\det(\mathfrak{L})=\det(M)$.

\begin{prop}\label{prop-10}
    Let $\Omega\subset\mathbb{R}^2$ be a bounded region. Then $Y\in\Omega\cap\mathfrak{L}$ if and only if $M^{-1}Y\in M^{-1}\Omega\cap\mathbb{Z}^2$.
\end{prop}
\begin{proof}
   Let $Y\in\Omega\cap\mathfrak{L}$. Then $Y\in\Omega$ and $Y\in\mathfrak{L}$. Since $Y\in\mathfrak{L}$, we have $Y=Mn$ for some $n\in\mathbb{Z}^2$. This in turn implies that $Mn\in\Omega$ and consequently $M^{-1}Y\in M^{-1}\Omega$. Moreover, since $Y=Mn$ with $n\in\mathbb{Z}^2$, we have $M^{-1}Y=n\in\mathbb{Z}^2$. Therefore $M^{-1}Y\in M^{-1}\Omega\cap\mathbb{Z}^2$. Conversely, let $M^{-1}Y\in M^{-1}\Omega\cap\mathbb{Z}^2$. This implies that $Y\in\Omega$ and $Y\in M\mathbb{Z}^2=\mathfrak{L}$. Thus, $Y\in\Omega\cap\mathfrak{L}$.
\end{proof}

\begin{remark}\label{rem-4}
    Note that the map $\psi : M^{-1}\Omega\cap\mathbb{Z}^2\to \Omega\cap\mathfrak{L}$ defined by $\psi(z)=Mz$ is a bijection.
\end{remark}

\begin{prop}\label{prop-12}
 Let $\mathfrak{L}$ be as in \eqref{lattice 1}. Then, we have
 \[\det(\mathfrak{L})\geq \frac{\lcm(q_0,q_1,\ldots,q_{\nu-1})}{A_1\cdots A_{\nu-1}}, \]
 where $q_j=\frac{d_j^k}{\gcd(d_j^k,r)},\ 0\leq j\leq\nu-1$.
\end{prop}
\begin{proof}
  Note that $d_j^k|r(yA_j-xB_j)$ if and only if $q_j|r(yA_j-xB_j),\ 0\leq j\leq\nu-1$. Thus, the sublattice $\mathfrak{L}$ can be expressed as $\mathfrak{L}=\cap_{j=0}^{\nu-1}\mathfrak{L}_j$, where \[\mathfrak{L}_j:=\left\{(x,y)\in\mathbb{Z}^2 : d_j^k|L_j^{(r)}(x,y)\right\}. \]   
 For each $0\leq j\leq\nu-1$, consider the map
 \[\Theta_j : \mathbb{Z}^2\to \mathbb{Z}/q_j\mathbb{Z}, \hspace{4mm} \Theta_j((x,y))=yA_j-xB_j\pmod{q_j}.\] It is readily seen that $\ker\Theta_j=\mathfrak{L}_j$. Therefore, by the first isomorphism theorem $\mathbb{Z}^2/\mathfrak{L}_j\cong \text{Im}(\Theta_j),\ 0\leq j\leq\nu-1$. Consequently, $[\mathbb{Z}^2:\mathfrak{L}_j]=|\text{Im}(\Theta_j)|$. Let $h_j=\gcd(q_j,A_j,B_j)$. Then $\text{Im}(\Theta_j)=h_j(\mathbb{Z}/q_j\mathbb{Z})$. Thus, $\#\text{Im}(\Theta_j)=\#h_j(\mathbb{Z}/q_j\mathbb{Z})=\frac{q_j}{h_j}$ and hence $[\mathbb{Z}^2:\mathfrak{L}_j]=\frac{q_j}{h_j}$. By Proposition \ref{prop-9}, $\mathfrak{L}_j$ is a rank-$2$ sublattice of $\mathbb{Z}^2$. Hence, $[\mathbb{Z}^2:\mathfrak{L}_j]=\det(\mathfrak{L}_j)$. Moreover, since $\mathfrak{L}\subseteq\mathfrak{L}_j\subseteq\mathbb{Z}^2$, the multiplicativity of the indices gives $[\mathbb{Z}^2:\mathfrak{L}]=[\mathbb{Z}^2:\mathfrak{L}_j][\mathfrak{L}_j:\mathfrak{L}]$. Hence, $[\mathbb{Z}^2:\mathfrak{L}_j]|[\mathbb{Z}^2:\mathfrak{L}]$ for every $0\leq j\leq\nu-1$. It follows that
 \[\lcm_{0\leq j\leq\nu-1}\left(\frac{q_j}{\gcd(q_j,A_j,B_j)}\right). \]
 Consequently, we obtain
 \[[\mathbb{Z}^2:\mathfrak{L}]\geq\lcm_{0\leq j\leq\nu-1}\left(\frac{q_j}{\gcd(q_j,A_j,B_j)}\right)\geq  \frac{\lcm(q_0,q_1,\ldots,q_{\nu-1})}{A_1\cdots A_{\nu-1}}. \]
\end{proof}
Denote \[\mathcal{A}_{m,r}=\{(u,v)\in(\mathbb{Z}/m\mathbb{Z})^2 : L_j^{(r)}(u,v)\equiv b\pmod{m}, 1\leq j\leq\nu-1 \}.\]
We define the map $\Phi : (\mathbb{Z}/m\mathbb{Z})^2\to(\mathbb{Z}/m\mathbb{Z})^2$ as $\Phi(x)=Mx$.
\begin{prop}\label{prop-13}
Let $m\geq 1, k\geq 2, A_j, B_j\geq 1$ be positive integers. Then
    \[\mathcal{V}(m,k,A_j,B_j)=\sum_{\substack{r, d_0^k\leq Q\\d_j^k\leq (QA_j-B_j)\\1\leq j\leq\nu-1\\ (m,d_0d_1\cdots d_{\nu-1})=1}}\frac{\mu(r)\mu(d_0)\mu(d_1)\cdots\mu(d_{\nu-1})\#\mathcal{A}_{m,r}\#\ker\Phi}{r^2\det(\mathfrak{L})}<\infty\ \text{as}\ Q\to\infty. \]
\end{prop}
\begin{proof}
  We consider $\mathcal{V}$ and apply Proposition \ref{prop-12} to $\det(\mathfrak{L})$. We obtain
  \begin{align*}
      \mathcal{V}(m,k,A_j,B_j)=\sum_{\substack{r, d_j\geq 1\\1\leq j\leq\nu-1\\ (m,d_0d_1\cdots d_{\nu-1})=1}}\frac{\mu(r)\mu(d_0)\mu(d_1)\cdots\mu(d_{\nu-1})\#\mathcal{A}_{m,r}\#\ker\Phi}{r^2\det(\mathfrak{L})}+\BigOm{Q^{\frac{-1}{2}+\epsilon}}.
  \end{align*}
We next estimate the main term in the above identity.
  \begin{align*}
     \mathcal{V}_1(m,k,A_j,B_j)\ll_m\sum_{\substack{r, d_j\geq 1\\1\leq j\leq\nu-1\\ (m,d_0d_1\cdots d_{\nu-1})=1}}\frac{|\mu(r)\mu(d_0)\mu(d_1)\cdots\mu(d_{\nu-1})|}{r^2\lcm(q_0,q_1,\ldots,q_{\nu-1})}=\sum_{n=1}^{\infty}F(n), \numberthis\label{k103}
  \end{align*}
  where $q_j=\frac{d_j^k}{\gcd(d_j^k,r)},\ 0\leq j\leq\nu-1$ and
  \[F(n)=\sum_{\substack{rd_0d_1\cdots d_{\nu-1}=n\\ r,d_0,d_1,\ldots,d_{\nu-1}\geq 1}}\frac{|\mu(r)\mu(d_0)\mu(d_1)\cdots\mu(d_{\nu-1})|}{r^2\lcm(q_0,q_1,\ldots,q_{\nu-1})}. \]
Let $n_1$ and $n_2$ be positive integers such that $\gcd(n_1,n_2)=1$. We have
\begin{align*}
    F(n_1n_2)&=\sum_{\substack{rd_0d_1\cdots d_{\nu-1}=n_1n_2 \\ r,d_0,d_1,\ldots,d_{\nu-1}\geq 1}}\frac{|\mu(r)\mu(d_0)\mu(d_1)\cdots\mu(d_{\nu-1})|}{r^2\lcm(q_0,q_1,\ldots,q_{\nu-1})}\\
    &=\sum_{\substack{r_1d_0^{(1)}d_1^{(1)}\cdots d_{\nu-1}^{(1)}=n_1 \\ r_1,d_0^{(1)},d_1^{(1)},\ldots,d_{\nu-1}^{(1)}\geq 1}}\frac{|\mu(r_1)\mu(d_0^{(1)})\mu(d_1^{(1)})\cdots\mu(d_{\nu-1}^{(1)})|}{r_1^2\lcm(q_0^{(1)},q_1^{(1)},\ldots,q_{\nu-1}^{(1)})}\\&\times\sum_{\substack{r_2d_0^{(2)}d_1^{(2)}\cdots d_{\nu-1}^{(2)}=n_2 \\ r_2,d_0^{(2)},d_1^{(2)},\ldots,d_{\nu-1}^{(2)}\geq 1}}\frac{|\mu(r_2)\mu(d_0^{(2)})\mu(d_1^{(2)})\cdots\mu(d_{\nu-1}^{(2)})|}{r_2^2\lcm(q_0^{(2)},q_1^{(2)},\ldots,q_{\nu-1}^{(2)})}
    =F(n_1)F(n_2).
\end{align*}
Therefore, $F(n)$ is a multiplicative function and the sum on the right-hand side of \eqref{k103} can be expressed as
\begin{align*}
   \sum_{n=1}^{\infty}F(n)=\prod_{p}\sum_{n=0}^{\infty}F(p^n). 
\end{align*}
By the definition of $F$, 
\[F(1)=1,\ F(p)=\frac{1}{p^2}+\frac{\nu}{p^k},\ F(p^{\nu+1})=\frac{1}{p^{k+1}},\] and
\[
F(p^n)=\begin{cases}
 \binom{\nu}{n}\cdot\frac{1}{p^k}+\binom{\nu}{n-1}\cdot\frac{1}{p^{k+1}},  &\mbox{if }\ 2\leq n\leq \nu,\\
  0,   & \mbox{if}\ n>\nu+1.
\end{cases} \]
Thus, we have
\begin{align*}
  \sum_{n=1}^{\infty}F(n)&=\prod_{p}\left(1+\frac{1}{p^2}+\frac{p\nu+1}{p^{k+1}}+\frac{1}{p^k}\sum_{n=1}^{\nu}\binom{\nu}{n}+\frac{1}{p^{k+1}}\sum_{n=2}^{\nu}\binom{\nu}{n-1}\right)\\
  &=\prod_{p}\left(1+\frac{1}{p^2}+\frac{2^{\nu}-1}{p^k}+\frac{2^{\nu}-1}{p^{k+1}}\right).
\end{align*}
Since $k\geq 2$, the product on the right hand side of the above identity is absolutely convergent. This completes the proof of Proposition \ref{prop-13}.
\end{proof}

\subsection{Weighted character sums} We will need the following estimate on weighted character sums to deal with the contribution coming from non-principal Dirichlet characters modulo $m$ in the computations for Theorems \ref{v correlation} and \ref{main result}.
\begin{prop}\label{prop15}
  Let $R>1,\ M$ and $\Lambda$ be positive real numbers and let $\delta$ be a positive integer. Suppose $\chi$ is a non-principal Dirichlet character modulo $m$ and $f$ is a continuously differentiable function with Supp$(f)\subset(0,\Lambda)$. Then for any integer $r\geq 1$, we have
  \[\sum_{\substack{a\leq R\\(a,r)=1}}\mu_k(a\delta)^2\chi(a)f\left(\frac{M}{a}\right)=\BigOmlamda{\tau(r)R^{\frac{1}{k}} \log R}. \]
\end{prop}
\begin{proof}
    We have
    \begin{align*}
        \sum_{\substack{a\leq R\\(a,r)=1}}\mu_k(a\delta)^2\chi(a)f\left(\frac{M}{a}\right)&=\sum_{\substack{a\leq R\\(a,r)=1}}\chi(a)f\left(\frac{M}{a}\right)\sum_{d^k|a\delta}\mu(d)\\
        &=\sum_{\substack{d^k\leq R\delta\\\left(\frac{d^k}{\gcd(d^k,\delta)},r\right)=1}}\mu(d)\chi\left(\frac{d^k}{\gcd(d^k,\delta)}\right)\sum_{\substack{a\leq \frac{R\gcd(d^k,\delta)}{d^k}\\(a,r)=1}}\chi(a)f\left(\frac{M\gcd(d^k,\delta)}{d^ka}\right).\\
    \end{align*}
In the last step, we used the fact that $a|bc$ if and only if $\frac{a}{\gcd(a,c)}|b$. Since $\delta$ is $k$-free $-$ otherwise the result would follow trivially $-$ it follows that $\left(\frac{d^k}{\gcd(d^k,\delta)},r\right)=1$ if and only if $(d,r)=1$. Therefore
\begin{align*}
 \sum_{\substack{a\leq R\\(a,r)=1}}\mu_k(a\delta)^2\chi(a)f\left(\frac{M}{a}\right)&= \sum_{\substack{d^k\leq R\delta\\\left(d,r\right)=1}}\mu(d)\chi\left(\frac{d^k}{\gcd(d^k,\delta)}\right)\sum_{\substack{a\leq \frac{R\gcd(d^k,\delta)}{d^k}\\(a,r)=1}}\chi(a)f\left(\frac{M\gcd(d^k,\delta)}{d^ka}\right)\\
 &=\sum_{s|r}\mu(s)\chi(s)\sum_{\substack{d^k\leq R\delta\\\left(d,r\right)=1}}\mu(d)\chi\left(\frac{d^k}{\gcd(d^k,\delta)}\right)\\&\times\sum_{\substack{a\leq \frac{R\gcd(d^k,\delta)}{sd^k}}}\chi(a)f\left(\frac{M\gcd(d^k,\delta)}{sd^ka}\right).\numberthis\label{1}
\end{align*}
To estimate the inner-most sum, we apply Abel summation formula
    \begin{align*}
     \sum_{\substack{a\leq \frac{R\gcd(d^k,\delta)}{sd^k}}}\chi(a)f\left(\frac{M\gcd(d^k,\delta)}{d^ksa}\right)&= f\left(\frac{M}{R}\right) \sum_{\substack{a\leq \frac{R\gcd(d^k,\delta)}{sd^k}}}\chi(a)+\int_{1}^{\frac{R\gcd(d^k,\delta)}{sd^k}}\\&\times\sum_{a\leq x}\chi(a)f^{\prime}\left(\frac{M\gcd(d^k,\delta)}{d^ksx}\right)\frac{M\gcd(d^k,\delta)dx}{d^ksx^2}\ll_{m,\Lambda}\log R. 
    \end{align*}
  The above estimate in conjunction with \eqref{1} gives the required result.   
\end{proof}
    \section{Weyl sum}
In this section, we study the equidistribution of the sequence $\left(\mathfrak{F}_{Q,k}^{(m)}\right)_Q$ by establishing an estimate for its associated Weyl sum.
\begin{proof}[Proof of Theorem \ref{Weyl sum}]
     We have
    \begin{align*}
        \sum_{\gamma\in \mathfrak{F}_{Q,k}^{(m)}}e(r\gamma)&=\sum_{\substack{q\leq Q\\q\equiv b\pmod{m}}}\mu_k(q)^2\sum_{\substack{1\leq a\leq q\\(a,q)=1}}e\left(\frac{ar}{q} \right)
    =\sum_{\substack{q\leq Q\\q\equiv b\pmod{m}}}\mu_k(q)^2\sum_{1\leq a\leq q}e\left(\frac{ar}{q} \right)\sum_{d|\gcd(a,q)}\mu(d)\\
    &=\sum_{d\leq Q}\mu(d) \sum_{\substack{q\leq \frac{Q}{d}\\qd\equiv b\pmod{m}}}\mu_k(qd)^2\sum_{1\leq a\leq q}e\left(\frac{ar}{q} \right)\\
    &=\sum_{d\leq Q}\mu(d)\sum_{\substack{q\leq \frac{Q}{d}\\qd\equiv b\pmod{m} \\q|r}}q\mu_k(qd)^2
    =\sum_{\substack{q\leq Q \\q|r}}q\sum_{\substack{d\leq\frac{Q}{q}\\qd\equiv b\pmod{m}}}\mu(d)\mu_k(qd)^2.
    \end{align*}
    We use Proposition \ref{mu} to estimate the inner sum above, and we find that
    \begin{align*}
    \sum_{\gamma\in \mathfrak{F}_{Q,k}^{(m)}}e(r\gamma)&\ll_{m} Q\exp{\left(-c\frac{(\log Q)^{3/5}}{(\log\log Q)^{1/5}} \right)}\sum_{\substack{q\leq Q\\q|r}}\mu_k(q)^2 \prod_{p|q}\left(1-\frac{1}{\sqrt{p}} \right)^{-1}.\numberthis\label{e_1}
   \end{align*}
 If $r$ is small with respect to $Q$ then using $\prod_{p|q}\left(1-\frac{1}{\sqrt{p}} \right)^{-1}\leq 2^{\omega(q)}$, where $\omega(q)$ counts the distinct prime divisors of $q$, we obtain
 \[\sum_{\substack{q\leq Q\\q|r}}\mu_k(q)^2 \prod_{p|q}\left(1-\frac{1}{\sqrt{p}} \right)^{-1}\ll \sum_{q|r}2^{\omega(q)}\ll r^{\epsilon}. \]
 If $r$ is large with respect to $Q$ then using Lemma \ref{mu product}, we obtain
 \[\sum_{\substack{q\leq Q\\q|r}}\mu_k(q)^2 \prod_{p|q}\left(1-\frac{1}{\sqrt{p}} \right)^{-1}\ll Q. \]
 Combining the above estimates with \eqref{e_1} yields
 \[ \sum_{\gamma\in \mathfrak{F}_{Q,k}^{(m)}}e(r\gamma) \ll_{m} \min\left(Q,r^{\epsilon} \right)Q\exp{\left(-c\frac{(\log Q)^{3/5}}{(\log\log Q)^{1/5}} \right)}. \]
   This completes the proof of Theorem \ref{Weyl sum}.
\end{proof}

\section{GRH and Farey fractions}
In this section, we develop an equivalent criterion for the Generalized Riemann Hypothesis in terms of the distribution of $\mathfrak{F}_{Q,k}^{(m)}$.
\subsection{Proof of Theorem \ref{thm2}}
We first assume that
\[\sum_{j=1}^{\mathcal{N}(Q,k,m)}R_{\mathcal{N}(Q,k,m)}(\gamma_j)=\BigOm{Q^{\frac{1}{2}+\epsilon}}. \]
We apply Lemma \ref{f(gamma)} with $f(x)=e(x)$ and obtain
\begin{align*}
   \sum_{v=1}^{\mathcal{N}(Q,k,m)}e(\gamma_v)&=\sum_{q\leq Q}M_q\left(\frac{Q}{q}\right)\sum_{a\leq q}e\left(\frac{a}{q}\right)=M_1(Q), 
\end{align*}
where $M_1(Q)=M(Q)=\sum_{\substack{n\leq Q\\n\equiv b\pmod{m}}}\mu(n)$.
In the last step, we used the following identity
\[\sum_{a\leq q}e\left(\frac{a}{q}\right)=\left\{\begin{array}{cc}
   1,  & \mbox{if} \ q=1,\\
  0,   & \mbox{otherwise}. 
\end{array}\right. \]
We have
\begin{align*}
    M(Q)&=\sum_{v=1}^{\mathcal{N}(Q,k,m)}e(\gamma_v)
    =\sum_{v=1}^{\mathcal{N}(Q,k,m)}e\left(\gamma_v-\frac{v}{\mathcal{N}(Q,k,m)}+\frac{v}{\mathcal{N}(Q,k,m)} \right)\\
    &=\sum_{v=1}^{\mathcal{N}(Q,k,m)}e\left(\frac{v}{\mathcal{N}(Q,k,m)} \right)(e(R_{\mathcal{N}(Q,k,m)}(\gamma_v))-1)+\sum_{v=1}^{\mathcal{N}(Q,k,m)}e\left(\frac{v}{\mathcal{N}(Q,k,m)} \right).
\end{align*}
This yields
\begin{align*}
    |M(Q)|&\leq \sum_{v=1}^{\mathcal{N}(Q,k,m)}|e(R_{\mathcal{N}(Q,k,m)}(\gamma_v))-1|
    \leq 2\pi \sum_{v=1}^{\mathcal{N}(Q,k,m)}R_{\mathcal{N}(Q,k,m)}(\gamma_v)\ll_{m} Q^{\frac{1}{2}+\epsilon}.
\end{align*}
Thus, GRH holds. For the converse, assume that GRH is true.  We apply Lemma \ref{f(gamma)} with $f(x)=x-\lfloor x\rfloor-\frac{1}{2}$. We have
\begin{align*}
    G(u)=\sum_{v=1}^{\mathcal{N}(Q,k,m)}f(u+\gamma_v)&=\sum_{q\leq Q}M_q\left(\frac{Q}{q}\right)\sum_{a\leq q}f\left(u+\frac{a}{q}\right)
    =\sum_{q\leq Q}M_q\left(\frac{Q}{q}\right)f(qu).
\end{align*}
We denote
\[I:=\int_0^1(G(u))^2du. \numberthis\label{k16}\]

Case-I: If
\[G(u)=\sum_{q\leq Q}M_q\left(\frac{Q}{q}\right)f(qu).\]
Substituting in \eqref{k16}, we have
\begin{align*}
I&=\int_0^1 \sum_{q_1\leq Q}f(q_1u)M_{q_1}\left(\frac{Q}{q_1}\right)\sum_{q_2\leq Q}f(q_2u)M_{q_2}\left(\frac{Q}{q_2}\right)du\\
    &=\sum_{q_1, q_2\leq Q}M_{q_1}\left(\frac{Q}{q_1}\right)M_{q_2}\left(\frac{Q}{q_2}\right)\int_0^1f(q_1u)f(q_2u)du.
    \numberthis\label{3}
\end{align*}
 The above integral is estimated as in \cite[p. 266-267]{Edwards} which yields
\[\int_0^1f(q_1u)f(q_2u)du=\frac{(\gcd(q_1,q_2))^2}{12q_1q_2}. \]
Hence, the above estimate with \eqref{3} gives
\begin{align*}
  I&=\frac{1}{12}\sum_{q_1, q_2\leq Q}M_{q_1}\left(\frac{Q}{q_1}\right)M_{q_2}\left(\frac{Q}{q_2}\right)\frac{(\gcd(q_1,q_2))^2}{q_1q_2} . \numberthis\label{4}  
\end{align*}
If GRH holds, then by employing Proposition \ref{mu}, we obtain
\begin{align*}
    M_q\left(\frac{Q}{q}\right)&=\sum_{\substack{d\leq\frac{Q}{q}\\qd\equiv b\pmod{m}}}\mu(d)\mu_k(qd)^2\ll_{m} \left(\frac{Q}{q}\right)^{\frac{1}{2}+\epsilon}\prod_{p|q}\left(1+\frac{1}{\sqrt{p}-1} \right).
\end{align*}
The above estimate with \eqref{4} yields
\begin{align*}
    I\leq CQ^{1+2\epsilon}\sum_{q_1, q_2\leq Q}\frac{(\gcd(q_1,q_2))^2}{(q_1q_2)^{\frac{3}{2}+\epsilon}}\leq CQ^{1+2\epsilon}\sum_{\delta\leq Q}\frac{1}{\delta^{1+\epsilon}}\sum_{\substack{q_1, q_2\leq \frac{Q}{\delta}\\ (q_1,q_2)=1}}\frac{1}{(q_1q_2)^{\frac{3}{2}+\epsilon}}\leq CQ^{1+2\epsilon},\numberthis\label{5}
\end{align*}
where $C>0$ is constant depending on $m$.


Case-II: Next, we apply Lemma \ref{prop7}, which implies that between $\gamma_v$ and $\gamma_{v+1}$, the value of $G$ is given by the closed form formula $G(u)=-1/2+\mathcal{N}(Q,k,m)u-v$. Therefore
\begin{align*}
    I&=\sum_{v=1}^{\mathcal{N}(Q,k,m)}\int_{\gamma_{v-1}}^{\gamma_v}\left(\frac{1}{2}+u\mathcal{N}(Q,k,m)-v\right)^2du\\
    &=\frac{1}{3\mathcal{N}(Q,k,m)}\sum_{v=1}^{\mathcal{N}(Q,k,m)}\left(\left(\gamma_v\mathcal{N}(Q,k,m)-v+\frac{1}{2}\right)^3-\left(\gamma_{v-1}\mathcal{N}(Q,k,m)-v+\frac{1}{2}\right)^3\right)\\
    &=\frac{1}{3\mathcal{N}(Q,k,m)}\sum_{v=1}^{\mathcal{N}(Q,k,m)}\left(\left(R_{\mathcal{N}(Q,k,m)}(\gamma_v)\mathcal{N}(Q,k,m)+\frac{1}{2}\right)^3-\left(R_{\mathcal{N}(Q,k,m)}(\gamma_v)\mathcal{N}(Q,k,m)-\frac{1}{2}\right)^3\right)\\
    &=\mathcal{N}(Q,k,m)\sum_{v=1}^{\mathcal{N}(Q,k,m)}(R_{\mathcal{N}(Q,k,m)}(\gamma_v))^2+\frac{1}{12}.\numberthis\label{6}
\end{align*}
The above estimate with \eqref{5} gives
\[\mathcal{N}(Q,k,m)\sum_{v=1}^{\mathcal{N}(Q,k,m)}(R_{\mathcal{N}(Q,k,m)}(\gamma_v))^2<CQ^{1+2\epsilon}, \]
where $C$ is a constant depending on $\epsilon$. By the Schwarz inequality, we have
\begin{align*}
\sum_{v=1}^{\mathcal{N}(Q,k,m)}R_{\mathcal{N}(Q,k,m)}(\gamma_v)&\leq \left(\sum_{v=1}^{\mathcal{N}(Q,k,m)}1\right)^{1/2}\left(\sum_{v=1}^{\mathcal{N}(Q,k,m)}(R_{\mathcal{N}(Q,k,m)}(\gamma_v))^2\right)^{1/2}\\ & 
\leq \left(\mathcal{N}(Q,k,m)\sum_{v=1}^{\mathcal{N}(Q,k,m)}(R_{\mathcal{N}(Q,k,m)}(\gamma_v))^2\right)^{1/2}\leq C^{1/2}Q^{1/2+\epsilon}.
\end{align*}
This completes the proof of Theorem \ref{thm2}.



\subsection{Proof of Theorem \ref{beta2}}
Employing \eqref{4} and \eqref{6}, we have
\begin{align*}
    \mathcal{N}(Q,k,m)\sum_{v=1}^{\mathcal{N}(Q,k,m)}(R_{\mathcal{N}(Q,k,m)}(\gamma_v))^2+\frac{1}{12}&=\frac{1}{12}\sum_{q_1, q_2\leq Q}M_{q_1}\left(\frac{Q}{q_1}\right)M_{q_2}\left(\frac{Q}{q_2}\right)\frac{(\gcd(q_1,q_2))^2}{q_1q_2}.
    \end{align*}
    Therefore, we have
    \begin{align*}
    \sum_{v=1}^{\mathcal{N}(Q,k,m)}(R_{\mathcal{N}(Q,k,m)}(\gamma_v))^2&=\frac{1}{12\mathcal{N}(Q,k,m)}\left(\sum_{q_1, q_2\leq Q}M_{q_1}\left(\frac{Q}{q_1}\right)M_{q_2}\left(\frac{Q}{q_2}\right)\frac{(\gcd(q_1,q_2))^2}{q_1q_2}-1\right).
\end{align*}
This completes the proof of the first part of Theorem \ref{beta2}.
In order to prove second part, we use the above identity and obtain
\begin{align*}
  \sum_{v=1}^{\mathcal{N}(Q,k,m)}(R_{\mathcal{N}(Q,k,m)}(\gamma_v))^2
    &=\frac{1}{12\mathcal{N}(Q,k,m)}\left(\sum_{q_1\leq Q}\sum_{\substack{d_1\leq \frac{Q}{q_1}\\q_1d_1\equiv b\pmod{m}}}\mu(d_1)\mu_k(q_1d_1)^2\sum_{q_2\leq Q}\right. \\ &\times \left.\sum_{\substack{d_2\leq \frac{Q}{q_2}\\q_2d_2\equiv b\pmod{m}}}\mu(d_2)\mu_k(q_2d_2)^2 \frac{(\gcd(q_1,q_2))^2}{q_1q_2}-1 \right).\numberthis\label{26}
\end{align*}
Let $\gcd(q_1,q_2)=\delta$ so that $q_1=q_1^{\prime}\delta$ and $q_2=q_2^{\prime}\delta$ with $(q_1^{\prime},q_2^{\prime})=1$. The above identity can be expressed as
\begin{align*}
\sum_{v=1}^{\mathcal{N}(Q,k,m)}(R_{\mathcal{N}(Q,k,m)}(\gamma_v))^2&=\frac{1}{12\mathcal{N}(Q,k,m)}\left(\sum_{\delta\leq Q}\sum_{q_1^{\prime}\leq \frac{Q}{\delta}}\frac{1}{q_1^{\prime}}\sum_{\substack{d_1\leq \frac{Q}{q_1^{\prime}\delta}\\q_1^{\prime}d_1\delta\equiv b\pmod{m}}}\mu(d_1)\mu_k(q_1^{\prime}d_1\delta)^2\right. \\ &\times \left.\sum_{\substack{q_2^{\prime}\leq \frac{Q}{\delta}\\(q_1^{\prime},q_2^{\prime})=1}}\frac{1}{q_2^{\prime}}\sum_{\substack{d_2\leq \frac{Q}{q_2^{\prime}\delta}\\q_2^{\prime}d_2\delta\equiv b\pmod{m}}}\mu(d_2)\mu_k(q_2^{\prime}d_2\delta)^2 -1 \right).\numberthis\label{k22}
\end{align*}
We apply Dirichlet hyperbola method to estimate the inner sum on the above identity
\begin{align*}
    S:&=\sum_{\substack{q\leq \frac{Q}{\delta}\\(q,l)=1}}\frac{1}{q}\sum_{\substack{d\leq \frac{Q}{q\delta}\\qd\delta\equiv b\pmod{m}}}\mu(d)\mu_k(qd\delta)^2\\
    &=\sum_{\substack{q\leq \sqrt{\frac{Q}{\delta}}\\(q,l)=1}}\frac{1}{q}\sum_{\substack{d\leq \frac{Q}{q\delta}\\qd\delta\equiv b\pmod{m}}}\mu(d)\mu_k(qd\delta)^2+\sum_{d\leq\sqrt{\frac{Q}{\delta}}}\mu(d)\sum_{\substack{q\leq\frac{Q}{d\delta}\\qd\delta\equiv b\pmod{m}\\ (q,l)=1}}\frac{\mu_k(qd\delta)^2}{q}\\&-\sum_{\substack{q\leq \sqrt{\frac{Q}{\delta}}\\(q,l)=1}}\frac{1}{q}\sum_{\substack{d\leq\sqrt{\frac{Q}{\delta}}\\qd\delta\equiv b\pmod{m}}}\mu(d)\mu_k(qd\delta)^2.
\end{align*}
Employing Proposition \ref{mu} to the inner sum in the first and last terms, and Proposition \ref{Prop1} to the inner sum in the second term of the above identity, we obtain
\begin{align*}
    S&\ll_{m} \frac{Q}{\delta}\sum_{\substack{q\leq \sqrt{\frac{Q}{\delta}}\\(q,l)=1}}\frac{1}{q^2} \exp{\left(-c\frac{(\log Q/q\delta)^{3/5}}{(\log\log Q/q\delta)^{1/5}} \right)}\prod_{p|q\delta}\left(\frac{\sqrt{p}}{\sqrt{p}-1}\right)\prod_{p|\ell}\left(1-\frac{1}{\sqrt{p}}\right)^{-1}\\ &\ll_{m} \frac{Q}{\delta}\exp{\left(-c\frac{(\log Q/\delta)^{3/5}}{(\log\log Q/\delta)^{1/5}} \right)}\prod_{p|\delta}\left(\frac{\sqrt{p}}{\sqrt{p}-1}\right)\prod_{p|\ell}\left(1-\frac{1}{\sqrt{p}}\right)^{-1}.
\end{align*}
Inserting the above estimate into \eqref{k22} gives
\begin{align*}
   \sum_{v=1}^{\mathcal{N}(Q,k,m)}(R_{\mathcal{N}(Q,k,m)}(\gamma_v))^2&\ll_{m} \frac{Q^2}{\mathcal{N}(Q,k,m)} \sum_{\delta\leq Q}\frac{1}{\delta^{2-\epsilon}}\exp{\left(-c\frac{(\log Q/\delta)^{3/5}}{(\log\log Q/\delta)^{1/5}} \right)}\\
   &\ll_{m} \sum_{\delta\leq\sqrt{Q}}\frac{1}{\delta^{2-\epsilon}}\exp{\left(-c\frac{(\log Q/\delta)^{\frac{3}{5}}}{(\log\log Q/\delta)^{\frac{1}{5}}} \right)}\\&+\sum_{\sqrt{Q}<\delta\leq Q}\frac{1}{\delta^{2-\epsilon}}\exp{\left(-c\frac{(\log Q/\delta)^{\frac{3}{5}}}{(\log\log Q/\delta)^{\frac{1}{5}}} \right)}
   \ll_{m} \exp{\left(-c\frac{(\log Q)^{\frac{3}{5}}}{(\log\log Q)^{\frac{1}{5}}} \right)}.
\end{align*}
This completes the proof unconditionally. We now estimate the sum on the right-hand side of \eqref{k22} under the assumption of the GRH. Assuming GRH, we apply Proposition \ref{mu}. Therefore, 
\begin{align*}
   \sum_{v=1}^{\mathcal{N}(Q,k,m)}(R_{\mathcal{N}(Q,k,m)}(\gamma_v))^2&\ll_{m} \frac{Q^{1+\epsilon}}{\mathcal{N}(Q,k,m)} \sum_{\delta\leq Q}\frac{1}{\delta^{1+\epsilon}}\sum_{q_1^{\prime}\leq \frac{Q}{\delta}}\frac{1}{(q_1^{\prime})^{\frac{3}{2}+\epsilon}}\sum_{q_2^{\prime}\leq \frac{Q}{\delta}}\frac{1}{(q_2^{\prime})^{\frac{3}{2}+\epsilon}}\ll_{m} Q^{-1+\epsilon}.
\end{align*}
This completes the proof of Theorem \ref{beta2}.

\section{Discrepancy}

\subsection{Proof of Theorem \ref{thm1}}
Let $\epsilon>0$ be arbitrarily small, and set $\alpha=1/Q-\epsilon$ to obtain a lower bound for $D_{\mathcal{N}(Q,k,m)}\left(\mathscr{F}_{Q,k}^{(m)}\right)$. By the definition of $A(\alpha;\mathcal{N}(Q,k,m))$, we have $A(1/Q-\epsilon;\mathcal{N}(Q,k,m))=0$.
By \eqref{D_1} and \eqref{R_N}, we get
\[D_{\mathcal{N}(Q,k,m)}\left(\mathscr{F}_{Q,k}^{(m)}\right)\geq R_{\mathcal{N}(Q,k,m)}(\alpha)=R_{\mathcal{N}(Q,k,m)}\left(\frac{1}{Q}-\epsilon\right)=\frac{1}{Q}-\epsilon \]
for all $\epsilon>0$. Since $\epsilon>0$ is arbitrary, one can thus deduce that
\[D_{\mathcal{N}(Q,k,m)}\left(\mathscr{F}_{Q,k}^{(m)}\right)\geq \frac{1}{Q}. \]
We next estimate the upper bound for the discrepancy. For any $\alpha\in [0,1]$, we write
\begin{align*}
    \mathcal{A}(\alpha;\mathcal{N}(Q,k,m))-\alpha\mathcal{N}(Q,k,m)
    &=\sum_{\substack{q\leq Q\\q\equiv b\pmod{m}}}\mu_k(q)^2\sum_{\substack{a\leq q\alpha\\(a,q)=1}}1-\alpha\sum_{\substack{q\leq Q\\q\equiv b\pmod{m}}}\mu_k(q)^2\sum_{\substack{a\leq q\\(a,q)=1}}1\\
    &=\sum_{\substack{q\leq Q \\q\equiv b\pmod{m}}}\mu_k(q)^2\left(\sum_{a\leq q\alpha}1-\alpha\sum_{a\leq q}1\right)\sum_{\substack{d|a\\d|q}}\mu(d)\\
    &=\sum_{d\leq Q}\mu(d)\sum_{\substack{q\leq\frac{Q}{d}\\qd\equiv b\pmod{m}}}\mu_k(qd)^2(\lfloor q\alpha\rfloor-\alpha\lfloor q\rfloor)\\
    &=-\sum_{d\leq Q}\mu(d)\sum_{\substack{q\leq\frac{Q}{d
    }\\qd\equiv b\pmod{m}}}\mu_k(qd)^2\{q\alpha \}.
\end{align*}
Next, we take the modulus of both sides. Therefore,
\begin{align*}
    |\mathcal{A}(\alpha;\mathcal{N}(Q,k,m))-\alpha\mathcal{N}(Q,k,m)|&=\left|\sum_{d\leq Q}\mu(d)\sum_{\substack{q\leq\frac{Q}{d
    }\\qd\equiv b\pmod{m}}}\mu_k(qd)^2\{q\alpha \} \right|\\
    &\ll\sum_{q\leq Q}\mu_k(q)^2\left|\sum_{\substack{d\leq\frac{Q}{q
    }\\qd\equiv b\pmod{m}}}\mu(d)\mu_k(qd)^2\right|.\numberthis\label{k5}
\end{align*}
By employing Proposition \ref{mu}, the above sum can be expressed as
    \begin{align*}
|\mathcal{A}(\alpha;\mathcal{N}(Q,k,m))-\alpha\mathcal{N}(Q,k,m)|&\ll_{m} \sum_{q\leq Q}\mu_k(q)^2\frac{Q}{q}\prod_{p|q}\left(1+\frac{1}{\sqrt{p}-1} \right)\exp{(-c\sqrt{\log (Q/q)})}\\
        &\ll_{m} \sum_{q\leq Q}\mu_k(q)^2\prod_{p|q}\left(1+\frac{1}{\sqrt{p}-1} \right)\sum_{d\leq \frac{Q}{q}}\exp{(-c\sqrt{\log d})}\\
        &\ll_{m} \sum_{d\leq Q}\exp{(-c\sqrt{\log d})}\sum_{q\leq \frac{Q}{d}}\mu_k(q)^2\prod_{p|q}\left(1+\frac{1}{\sqrt{p}-1} \right).
    \end{align*}
    To estimate the inner-sum, we apply Proposition \ref{mu product} and obtain
    \begin{align*}
        |\mathcal{A}(\alpha;\mathcal{N}(Q,k,m))-\alpha\mathcal{N}(Q,k,m)|&\ll_{m} \sum_{d\leq Q}\exp{(-c\sqrt{\log d})}\frac{Q}{d\zeta(k)}\prod_{p}\left(1+\frac{p^{k-1}-1}{(p^{\frac{1}{2}}-1)(p^k-1)} \right)\\&\ll_{m} \frac{Q}{\zeta(k)}\sum_{d\leq Q}\frac{1}{d\exp{(c\sqrt{\log d})}}\ll_{m} Q.\numberthis\label{k14}
    \end{align*}
Therefore,
\[R_{\mathcal{N}(Q,k,m)}(\alpha)=\frac{1}{\mathcal{N}(Q,k,m)}\left|A(\alpha;\mathcal{N}(Q,k,m))-\alpha \mathcal{N}(Q,k,m)\right|\ll_{m}\frac{1}{Q}, \]
uniformly in $\alpha\in[0,1]$. This completes the proof of Theorem \ref{thm1}.
\section{$\nu$-level correlations }
The following section discusses the $\nu$-level correlation measure of the sequence $\left(\mathfrak{F}_{Q,k}^{(m)}\right)_Q$. We begin by establishing a closed-form formula for the exponential sum over the Farey fractions whose denominators are $k$-free and lie in an arithmetic progression.
\begin{lem}\label{exp nu}
   Let $r\in \mathbb{Z},$ we have
   \[\sum_{\gamma\in \mathfrak{F}_{Q,k}^{(m)}}e(r\gamma)=\sum_{\substack{q\leq Q\\q|r}}qM_q\left(\frac{Q}{q} \right), \]
   where $M_q(x)=\sum_{\substack{d\leq x\\qd\equiv b\pmod{m}}}\mu(d)\mu_k(qd)^2$.
\end{lem}
\begin{proof}
    We have
    \begin{align*}
        \sum_{\gamma\in \mathfrak{F}_{Q,k}^{(m)}}e(r\gamma)&=\sum_{\substack{q\leq Q\\q\equiv b\pmod{m}}}\mu_k(q)^2\sum_{\substack{1\leq a\leq q\\(a,q)=1}}e\left(\frac{ar}{q} \right)
    =\sum_{\substack{q\leq Q\\q\equiv b\pmod{m}}}\mu_k(q)^2\sum_{1\leq a\leq q}e\left(\frac{ar}{q} \right)\sum_{d|\gcd(a,q)}\mu(d)\\
    &=\sum_{d\leq Q}\mu(d) \sum_{\substack{q\leq \frac{Q}{d}\\qd\equiv b\pmod{m}}}\mu_k(qd)^2\sum_{1\leq a\leq q}e\left(\frac{ar}{q} \right)
    =\sum_{\substack{q\leq Q\\ q|r}}qM_q\left(\frac{Q}{q} \right).
    \end{align*}
\end{proof}

\subsection{Proof of Theorem \ref{v correlation}}
In order to establish the $\nu$-level correlation measure for the sequence of Farey fractions with $k$-free denominators $q$ that run through a given arithmetic progression, we need to estimate, for any positive real number $\Lambda$, the quantity
\begin{align*}
   \mathcal{S}_{\mathfrak{F}_{Q,k}^{(m)}}^{\nu}(\Lambda)=&\frac{1}{\mathcal{N}{(Q,k,m)}}\#\left\{(\gamma_1,\ldots,\gamma_{\nu})\in \left(\mathfrak{F}_{Q,k}^{(m)}\right)^{\nu}: \gamma_i\ \text{distinct},\ \right. \\ &  \left. (\gamma_1-\gamma_2,\ldots,\gamma_{\nu-1}-\gamma_{\nu})\in\frac{\mathfrak{B}}{\mathcal{N}{(Q,k,m)}}+\mathbb{Z}^{\nu-1}\right\}. 
\end{align*}
To estimate this, we build upon the ideas introduced in \cite{BocaF} making several necessary and technical modifications on the way. 
For a smooth real valued function $H$ on $\mathbb{R}^{\nu-1}$ such that Supp$(H)\subset(0,\Lambda)^{\nu-1}$, define
\[f(y)=\sum_{r\in \mathbb{Z}^{\nu-1}}H(\mathcal{N}{(Q,k,m)}(y+r)),\ y\in\mathbb{R}^{\nu-1}, \]
and
\[S^{(\nu)}_{Q,k}=\sum_{\gamma_i\in\mathfrak{F}_{Q,k}^{(m)}, \text{distinct}}f(\gamma_1-\gamma_2,\ldots,\gamma_{\nu-1}-\gamma_{\nu}). \numberthis\label{h1}\]
Since Supp$H\subset(0,\Lambda)$, the condition $\gamma_i\ne\gamma_j$ for $i\ne j$ can be removed for $Q$ large enough that $\mathcal{N}({Q,k,m})>\Lambda.$ Let
 \[f(y)=\sum_{r\in \mathbb{Z}^{\nu-1}}c_re(r\cdot y)\]
 be the Fourier series expansion of $f$, with the Fourier coefficients 
 \begin{align*}
     c_r&=\int_{[0,1)^{\nu-1}}f(x)e(-r\cdot x)dx
     =\frac{1}{(\mathcal{N}{(Q,k,m))^{\nu-1}}}\widehat{H}\left(\frac{r}{\mathcal{N}{(Q,k,m)}}\right),\numberthis\label{Fourier coef}
 \end{align*}
 where $\widehat{H}$ is the Fourier transform of $H.$ Then by \eqref{h1}, we have
\begin{align*}
     S^{(\nu)}_{Q,k}&=\sum_{\gamma_1,\ldots,\gamma_{\nu}\in \mathfrak{F}_{Q,k}^{(m)}}f(\gamma_1-\gamma_2,\ldots,\gamma_{\nu-1}-\gamma_{\nu})
     =\sum_{\substack{\gamma_1,\ldots,\gamma_{\nu}\in \mathfrak{F}_{Q,k}^{(m)}\\r_1,\ldots,r_{\nu-1}\in\mathbb{Z}}}c_re(r\cdot(\gamma_1-\gamma_2,\ldots,\gamma_{\nu-1}-\gamma_{\nu}))\\
     &=\sum_{\substack{\gamma_1,\ldots,\gamma_{\nu}\in \mathfrak{F}_{Q,k}^{(m)}\\r_1,\ldots,r_{\nu-1}\in\mathbb{Z}}}c_re(r_1\gamma_1)e((r_2-r_1)\gamma_2)\ldots e((r_{\nu-1}-r_{\nu-2})\gamma_{\nu-1})e(r_{\nu-1}\gamma_{\nu}).\numberthis\label{k66}
 \end{align*}
By applying Lemma \ref{exp nu} to the above identity yields
\begin{align*}
    S^{(\nu)}_{Q,k}
    &=\sum_{1\leq d_1,\ldots,d_{\nu}\leq Q}d_1\cdots d_{\nu}M_{d_1}\left(\frac{Q}{d_1} \right)\cdots M_{d_{\nu}}\left(\frac{Q}{d_{\nu}} \right)\sum_{\substack{d_1|r_1\\d_2|r_2-r_1\\\ldots\\d_{\nu-1}|r_{\nu-1}-r_{\nu-2}\\d_{\nu}|r_{\nu-1}}}c_r.
\end{align*}
The divisibility conditions in the inner-sum of the above identity can be expressed as
\begin{align*}
    r_1&=l_1d_1\\
    r_2&=l_1d_1+l_2d_2\\
    &\cdots\\
    r_{\nu-1}&=l_1d_1+\cdots+l_{\nu-1}d_{\nu-1}=l_{\nu}d_{\nu}
\end{align*}
for some $l_1,\ldots,l_{\nu}\in\mathbb{Z}$. We denote $d=(d_1,\ldots,d_{\nu-1})\in\square_Q^{\nu-1}:=[1,Q]^{\nu-1}\cap\mathbb{Z}^{\nu-1}$, $l=(l_1,\ldots,l_{\nu-1})$. We obtain
\begin{align*}
   S^{(\nu)}_{Q,k}&= \sum_{d\in\square_Q^{\nu-1}}d_1\cdots d_{\nu-1}M_{d_1}\left(\frac{Q}{d_1} \right)\cdots M_{d_{\nu-1}}\left(\frac{Q}{d_{\nu-1}} \right)\\&\times\sum_{l\in\mathbb{Z}^{\nu-1}}c_{d_1l_1,d_1l_1+d_2l_2,\ldots,d_1l_1+\cdots+d_{\nu-1}l_{\nu-1}}\sum_{d_{\nu}|d_1l_1+\cdots+d_{\nu-1}l_{\nu-1}}d_{\nu}M_{d_{\nu}}\left(\frac{Q}{d_{\nu}} \right).\numberthis\label{h3}
\end{align*}
By using \eqref{Fourier coef} and Lemma \ref{exp nu}, the two inner sums in \eqref{h3} take the form
\begin{align*}
&\sum_{\substack{l\in\mathbb{Z}^{\nu-1}\\\gamma\in\mathfrak{F}_{Q,k}^{(m)}}}c_{d_1l_1,d_1l_1+d_2l_2,\ldots,d_1l_1+\cdots+d_{\nu-1}l_{\nu-1}}e(-\gamma d\cdot l)\\
&=\sum_{\substack{l\in\mathbb{Z}^{\nu-1}\\\gamma\in\mathfrak{F}_{Q,k}^{(m)}}}\int_{\mathbb{R}^{\nu-1}}e\left(-\sum_{i=1}^{\nu-1}d_il_i(x_i+\cdots+x_{\nu-1})\right)H(\mathcal{N}{(Q,k,m)}(x_1,\ldots,x_{\nu-2},x_{\nu-1}-\gamma))dx.\numberthis\label{h2}
\end{align*}
We take $y_i=d_i(x_i+\cdots+x_{\nu-1}),\ i=1,\ldots,\nu-1$ with $y=(y_1,\ldots,y_{\nu-1})\in\mathbb{R}^{\nu-1}$ and set
\[H_{\mathcal{N}{(Q,k,m)};d,\gamma}(y)=H\left(\mathcal{N}{(Q,k,m)}\left(\frac{y_1}{d_1}-\frac{y_2}{d_2} \right),\ldots,\mathcal{N}{(Q,k,m)}\left(\frac{y_{\nu-2}}{d_{\nu-2}}-\frac{y_{\nu-1}}{d_{\nu-1}} \right),\mathcal{N}{(Q,k,m)}\left(\frac{y_{\nu-1}}{d_{\nu-1}}-\gamma \right) \right). \]
Therefore, the identity in \eqref{h2} can be expressed as follows
\begin{align*}
    \frac{1}{d_1\cdots d_{\nu-1}}&\sum_{\gamma\in\mathfrak{F}_{Q,k}^{(m)}}\sum_{l\in\mathbb{Z}^{\nu-1}}\int_{\mathbb{R}^{\nu-1}}e(-l\cdot y)H_{\mathcal{N}{(Q,k,m)};d,\gamma}(y)dy
    =\frac{1}{d_1\cdots d_{\nu-1}}\sum_{\gamma\in\mathfrak{F}_{Q,k}^{(m)}}\sum_{l\in\mathbb{Z}^{\nu-1}}\widehat{H}_{\mathcal{N}{(Q,k,m)};d,\gamma}(l).
\end{align*}
Employing the Poisson summation formula to the inner sum of the above identity and inserting it back into \eqref{h3}, we obtain
\begin{align*}
   S^{(\nu)}_{Q,k}&= \sum_{d\leq\square_Q^{\nu-1}}M_{d_1}\left(\frac{Q}{d_1} \right)\cdots M_{d_{\nu-1}}\left(\frac{Q}{d_{\nu-1}} \right)\sum_{\gamma\in\mathfrak{F}_{Q,k}^{(m)}}\sum_{l\in\mathbb{Z}^{\nu-1}}H_{\mathcal{N}{(Q,k,m)};d,\gamma}(l). 
\end{align*}
As Supp$H\subset(0,\Lambda)^{\nu-1}$, we have
\[0<\mathcal{N}{(Q,k,m)}\left(\frac{l_j}{d_j}-\frac{l_{j+1}}{d_{j+1}} \right)<\Lambda^{\prime},\ j=1,\ldots,\nu-2. \]
The above inequalities implies $l_jd_{j+1}-l_{j+1}d_j\geq 1$ and
\[\Lambda>\frac{\mathcal{N}{(Q,k,m)}(l_jd_{j+1}-l_{j+1}d_j)}{d_jd_{j+1}}\geq \frac{\mathcal{N}{(Q,k,m)}}{d_jd_{j+1}}. \]
Therefore, for all $Q\geq Q_0(\Lambda)$ using above inequality we get
\[\frac{Q^2}{d_jd_{j+1}}=\frac{Q^2}{\mathcal{N}{(Q,k,m)}}\cdot \frac{\mathcal{N}{(Q,k,m)}}{d_jd_{j+1}}<\frac{Q^2\Lambda}{\mathcal{N}{(Q,k,m)}}<\frac{\Lambda}{\mathscr{C}(k,m)}=:\mathcal{C}{(\Lambda,k,m)}. \]
Note that both $Q/d_j\ge 1$ and $Q/d_{j+1}\ge 1$. Therefore, for all $Q\geq Q_0(\Lambda)$, we have
$1\leq \frac{Q}{d_j}\leq\mathcal{C}{(\Lambda,k,m)},\ j=1\ldots,\nu-1.$
Similarly, we obtain
\[\frac{Q}{q}\leq\mathcal{C}{(\Lambda,k,m)}.\numberthis\label{h5} \]
Hence,
\begin{align*}
    S^{(\nu)}_{Q,k}
    &=\sum_{\substack{1\leq n_j\leq \mathcal{C}{(\Lambda,k,m)}\\1\leq j\leq\nu-1}}\mu(n_1)\cdots\mu(n_{\nu-1})\sum_{\substack{1\leq d_j\leq Q/n_j\\n_jd_j\equiv b\pmod{m}\\1\leq j\leq\nu-1}}\mu_k(n_1d_1)^2\cdots\mu_k(n_{\nu-1}d_{\nu-1})^2\\&\times\sum_{l\in\mathbb{Z}^{\nu-1}}\sum_{\substack{a/q\in\mathfrak{F}_{Q,k}^{(m)}\\q\geq Q/\mathcal{C}{(\Lambda,k,m)}}}H_{\mathcal{N}{(Q,k,m)};d,\gamma}(l).\numberthis\label{h4}
\end{align*}
We set $\Delta_j=ql_j-ad_j$ for $j=1,\ldots,\nu-1$. Consequently, $l_j$ is uniquely determined as $l_j=\frac{\Delta_j+ad_j}{q}$. This in turn implies that
\[\frac{l_j}{d_j}-\frac{l_{j+1}}{d_{j+1}}=\frac{\Delta_j+ad_j}{qd_j}-\frac{\Delta_{j+1}+ad_{j+1}}{qd_{j+1}}=\frac{1}{q}\left(\frac{\Delta_j}{d_j}-\frac{\Delta_{j+1}}{d_{j+1}} \right),\ j=1,\ldots,\nu-2. \]
Moreover,
\[\frac{l_{\nu-1}}{d_{\nu-1}}-\frac{a}{q}=\frac{\Delta_{\nu-1}}{qd_{\nu-1}}. \]
Also, $d_j$ satisfy the congruence $d_j\equiv -\bar{a}\Delta_j\pmod{q},\ j=1,\ldots,\nu-1$, where $1\leq\bar{a}\leq q$ such that $a\bar{a}=1\pmod{q}$. 
Since Supp$H\subset(0,\Lambda)^{\nu-1}$, we get
\[0<\frac{\mathcal{N}{(Q,k,m)\Delta_j}}{qd_j}
=\mathcal{N}{(Q,k,m)}\left(\frac{l_{j}}{d_{j}}-\frac{l_{j+1}}{d_{j+1}} \right)+\cdots+\mathcal{N}{(Q,k,m)}\left(\frac{l_{\nu-1}}{d_{\nu-1}}-\frac{a}{q} \right)<(\nu-j)\Lambda^{\prime}. \]
For $Q\geq Q_0(\Lambda)$, the above inequalities give
\[1\leq\Delta_j\leq\frac{qd_j(\nu-j)\Lambda}{\mathcal{N}{(Q,k,m)}}\leq\frac{Q^2(\nu-j)\Lambda}{\mathcal{N}{(Q,k,m)}}\leq (\nu-j)\mathcal{C}{(\Lambda,k,m)}; \]
thus, we have $1\leq\Delta_1,\ldots,\Delta_{\nu-1}\leq (\nu-1)\mathcal{C}{(\Lambda,k,m)}$. Therefore, \eqref{h4} becomes
\begin{align*}
  S^{(\nu)}_{Q,k}&= \sum_{\substack{1\leq n_j\leq \mathcal{C}{(\Lambda,k,m)}\\1\leq j\leq\nu-1}}\mu(n_1)\cdots\mu(n_{\nu-1})\sum_{\substack{1\leq\Delta_j\leq(\nu-1)\mathcal{C}{(\Lambda,k,m)}\\1\leq j\leq\nu-1}}\sum_{\substack{1\leq d_j\leq Q/n_j\\n_jd_j\equiv b\pmod{m}\\1\leq j\leq\nu-1}}\mu_k(n_1d_1)^2\cdots\mu_k(n_{\nu-1}d_{\nu-1})^2\\
  &\times\sum_{\substack{a/q\in\mathfrak{F}_{Q,k}^{(m)}, q\geq Q/\mathcal{C}{(\Lambda,k,m)}\\d_j\equiv -\bar{a}\Delta_j\pmod{q}}}H\left(\frac{\mathcal{N}{(Q,k,m)}}{q}\left(\frac{\Delta_1}{d_1}-\frac{\Delta_2}{d_2},\ldots,\frac{\Delta_{\nu-2}}{d_{\nu-2}}-\frac{\Delta_{\nu-1}}{d_{\nu-1}},\frac{\Delta_{\nu-1}}{d_{\nu-1}} \right) \right).
\end{align*}
We simplify the above expression by employing the linear transformation $T$ defined in \eqref{k63}. We set $\Tilde{H}=H\circ T$, which is smooth and Supp$\Tilde{H}\subset(0,(\nu-1)\Lambda]\times\cdots\times(0,\Lambda]$. The above identity then becomes
\begin{align*}
  S^{(\nu)}_{Q,k}&= \sum_{\substack{1\leq n_j\leq \mathcal{C}{(\Lambda,k,m)}\\1\leq j\leq\nu-1}}\mu(n_1)\cdots\mu(n_{\nu-1})\sum_{\substack{1\leq\Delta_j\leq(\nu-1)\mathcal{C}{(\Lambda,k,m)}\\1\leq j\leq\nu-1}}\sum_{\substack{1\leq d_j\leq Q/n_j\\n_jd_j\equiv b\pmod{m}\\1\leq j\leq\nu-1}}\mu_k(n_1d_1)^2\cdots\mu_k(n_{\nu-1}d_{\nu-1})^2\\
  &\times\sum_{\substack{a/q\in\mathfrak{F}_{Q,k}^{(m)}, q\geq Q/\mathcal{C}{(\Lambda,k,m)}\\ d_j=-\bar{a}\Delta_j\pmod{q}}} \Tilde{H}\left(\frac{\mathcal{N}{(Q,k,m)}}{q}\left(\frac{\Delta_1}{d_1},\frac{\Delta_2}{d_2},\ldots,\frac{\Delta_{\nu-1}}{d_{\nu-1}} \right) \right). 
\end{align*}
We define
$e_j=\frac{d_j+\bar{a}\Delta_j}{q},\ j=1,\ldots,\nu-1.$
Note that $e_j$ is an integer since $d_j=-\bar{a}\Delta_j\pmod{q}$. As $d_j,\bar{a}$, and $ \Delta_j$ are all integers, it follows that $e_j\geq 1$. Moreover, using \eqref{h5}, we obtain $1\leq e_j\leq \nu\mathcal{C}{(\Lambda,k,m)},\ j=1,\ldots,\nu-1$. For each value of $e_j$, with $a,q,$ and $ \Delta_j$ fixed, we obtain a unique value of $d_j$; in particular, $d_j=qe_j-\bar{a}\Delta_j$. Also, with fixed $e_j$ and $\Delta_j$ and variable $a/q\in\mathfrak{F}_{Q,k}^{(m)}$, in order for $d_j$ to belong to the set $\{1,\ldots,\lfloor Q/n_j\rfloor \}$, $a$ and $q$ must satisfy $\frac{Q}{n_j\mathcal{C}{(\Lambda,k,m)}}\leq qe_j-\bar{a}\Delta_j\leq\frac{Q}{n_j} $.
We consider the region
\[\Omega_{n,e,k,\Delta}=\left\{(x,y): 0<x\leq y\leq 1,\ y\geq \frac{1}{\mathcal{C}{(\Lambda,k,m)}},\ \frac{n_j^{-1}}{\mathcal{C}{(\Lambda,k,m)}}\leq ye_j-x\Delta_j\leq\frac{1}{n_j}, 1\leq j\leq\nu-1 \right\}. \]
We next set the functions $f_{k,e,\Delta}, f_{k,e,\Delta}^{(j)}$ defined on $\Omega_{n,e,k,\Delta}$ as follows
\[f_{k,e,\Delta}(x,y)=\Tilde{H}\left(f_{k,e,\Delta}^{(1)}(x,y),\ldots,f_{k,e,\Delta}^{(\nu-1)}(x,y)\right)\ \text{and}\ f_{k,e,\Delta}^{(j)}(x,y)=\frac{\mathcal{N}{(Q,k,m)}\Delta_j}{y(ye_j-x\Delta_j)},\ j=1,\ldots,\nu-1. \]
We also set $a^{\prime}=\bar{a}$ and note that $a^{\prime}/q\in\mathfrak{F}_{Q,k}^{(m)}$ with $q\geq Q/\mathcal{C}{(\Lambda,k,m)}$ as $a/q\in\mathfrak{F}_{Q,k}^{(m)}$. Therefore
\begin{align*}
   S^{(\nu)}_{Q,k}&= \sum_{\substack{1\leq n_j\leq \mathcal{C}{(\Lambda,k,m)}\\1\leq j\leq\nu-1}}\mu(n_1)\cdots\mu(n_{\nu-1})\sum_{\substack{1\leq\Delta_j\leq(\nu-1)\mathcal{C}{(\Lambda,k,m)}\\1\leq e_j\leq \nu\mathcal{C}{(\Lambda,k,m)}\\1\leq j\leq\nu-1}}\sum_{\substack{(a^{\prime},q)\in Q\Omega_{n,e,k,\Delta}\\ (a^{\prime},q)=1, \mu_k(q)^2=1\\q\equiv b\pmod{m}\\\mu_k(qA_j-a^{\prime}B_j)^2=1\\qA_j-a^{\prime}B_j\equiv b\pmod{m}\\ 1\leq j\leq\nu-1}}f_{k,e,\Delta}(a^{\prime},q)\\
   &= \sum_{\substack{1\leq n_j\leq \mathcal{C}{(\Lambda,k,m)}\\1\leq j\leq\nu-1}}\mu(n_1)\cdots\mu(n_{\nu-1})\sum_{\substack{1\leq\Delta_j\leq(\nu-1)\mathcal{C}{(\Lambda,k,m)}\\1\leq e_j\leq \nu\mathcal{C}{(\Lambda,k,m)}\\1\leq j\leq\nu-1}}\sum_{\substack{(a^{\prime},q)\in Q\Omega_{n,e,k,\Delta}\\ (a^{\prime},q)=1, \mu_k(L_j(a^{\prime},q))^2=1\\L_j(a^{\prime},q)\equiv b\pmod{m}\\ 0\leq j\leq\nu-1}}f_{k,e,\Delta}(a^{\prime},q), \numberthis\label{k65}
\end{align*}
where
\[L_j(a^{\prime},q)=\begin{cases}
 q,  &\mbox{if }\ j=0,\\
   qA_j-a^{\prime}B_j,  & \mbox{if}\ 1\leq j\leq\nu-1,
\end{cases} \] 
and $A_j=n_je_j, B_j=n_j\Delta_j$.
We estimate the innermost sum in the above identity.
\begin{align*}
   S^{(\nu)}_{Q,k}(f)&= \sum_{\substack{(a^{\prime},q)\in Q\Omega_{n,e,k,\Delta}\\ (a^{\prime},q)=1, \mu_k(L_j(a^{\prime},q))^2=1\\L_j(a^{\prime},q)\equiv b\pmod{m}\\ 0\leq j\leq\nu-1}}f_{k,e,\Delta}(a^{\prime},q)=\sum_{\substack{(a^{\prime},q)\in Q\Omega_{n,e,k,\Delta}\\ (a^{\prime},q)=1, (a^{\prime},q)\in\mathcal{A}_m \\\mu_k(L_j(a^{\prime},q))^2=1\\ 0\leq j\leq\nu-1}}f_{k,e,\Delta}(a^{\prime},q),
\end{align*}
where $\mathcal{A}_m=\{(u,v)\in(\mathbb{Z}/m\mathbb{Z})^2 : L_j(u,v)\equiv b\pmod{m}, 0\leq j\leq\nu-1 \}$. We use M\"{o}bius to detect the $k$-free conditions.
\begin{align*}
 S^{(\nu)}_{Q,k}(f)&=\sum_{\substack{(a^{\prime},q)\in Q\Omega_{n,e,k,\Delta}\\ (a^{\prime},q)=1, (a^{\prime},q)\in\mathcal{A}_m}}f_{k,e,\Delta}(a^{\prime},q)\sum_{d_0^k|L_0(a^{\prime},q)}\mu(d_0)\sum_{d_1^k|L_1(a^{\prime},q)}\mu(d_1)\cdots \sum_{d_{\nu-1}^k|L_{\nu-1}(a^{\prime},q)}\mu(d_{\nu-1})\\
 &=\sum_{\substack{d_0^k\leq Q\\d_j\leq (QA_j-B_j)\\1\leq j\leq\nu-1}}\mu(d_0)\mu(d_1)\cdots\mu(d_{\nu-1})\sum_{\substack{(a^{\prime},q)\in Q\Omega_{n,e,k,\Delta}\\ (a^{\prime},q)=1, (a^{\prime},q)\in\mathcal{A}_m\\d_j^k|L_j(a^{\prime},q), 0\leq j\leq\nu-1}}f_{k,e,\Delta}(a^{\prime},q)\\
 &=\sum_{\substack{d_0^k\leq Q\\d_j\leq (QA_j-B_j)\\1\leq j\leq\nu-1}}\mu(d_0)\mu(d_1)\cdots\mu(d_{\nu-1})\sum_{\substack{(a^{\prime},q)\in Q\Omega_{n,e,k,\Delta}\\  (a^{\prime},q)\in\mathcal{A}_m\\d_j^k|L_j(a^{\prime},q), 0\leq j\leq\nu-1}}f_{k,e,\Delta}(a^{\prime},q)\sum_{\substack{r|a^{\prime}\\ r|q}}\mu(r)\\
 &=\sum_{\substack{r, d_0^k\leq Q\\d_j\leq (QA_j-B_j)\\1\leq j\leq\nu-1}}\mu(r)\mu(d_0)\mu(d_1)\cdots\mu(d_{\nu-1})\sum_{\substack{(a,q)\in Q\Omega_{n,e,k,\Delta}^{(r)}\\  (a,q)\in\mathcal{A}_{m,r}\\d_j^k|L_j^{(r)}(a,q), 0\leq j\leq\nu-1}}f_{k,e,\Delta}^{(1)}(a,q), \numberthis\label{k101}
\end{align*}
where $f_{k,e,\Delta}^{(1)}(a,q)=f_{k,e,\Delta}(ra,rq)$,\ $\mathcal{A}_{m,r}=\{(u,v)\in(\mathbb{Z}/m\mathbb{Z})^2 : L_j^{(r)}(u,v)\equiv b\pmod{m}, 0\leq j\leq\nu-1 \}$,
\[L_j^{(r)}(a^{\prime},q)=\begin{cases}
 qr,  &\mbox{if }\ j=0,\\
   r(qA_j-a^{\prime}B_j),  & \mbox{if}\ 1\leq j\leq\nu-1,
\end{cases} \] 
 and
 \[\Omega_{n,e,k,\Delta}^{(r)}=\left\{(x,y): 0<x\leq y\leq \frac{1}{r},\ y\geq \frac{r^{-1}}{\mathcal{C}{(\Lambda,k,m)}},\ \frac{(rn_j)^{-1}}{\mathcal{C}{(\Lambda,k,m)}}\leq ye_j-x\Delta_j\leq\frac{1}{rn_j}, 1\leq j\leq\nu-1 \right\}. \]
 Since $\gcd(b,m)=1$ and $L_j^{(r)}(a,q)\equiv b\pmod{m}\ \text{for}\ 0\leq j\leq\nu-1 \}$, it follows that $\gcd(L_j^{(r)}(a,q),m)=1$ for every $0\leq j\leq\nu-1$. Consequently, we have $\gcd(m,rd_0d_1\cdots d_{\nu-1})=1$. Therefore, \eqref{k101} can be expressed as
\begin{align*}
  S^{(\nu)}_{Q,k}(f)&= \sum_{\substack{r, d_0^k\leq Q\\d_j\leq (QA_j-B_j)\\1\leq j\leq\nu-1\\ \gcd(m,rd_0d_1\cdots d_{\nu-1})=1}}\mu(r)\mu(d_0)\mu(d_1)\cdots\mu(d_{\nu-1})T(r,d_0,\ldots,d_{\nu-1}), \numberthis\label{k102}
\end{align*}
where
\[T(r,d_0,\ldots,d_{\nu-1}):=\sum_{\substack{(a,q)\in Q\Omega_{n,e,k,\Delta}^{(r)}\\  (a,q)\in\mathcal{A}_{m,r}\\d_j^k|L_j^{(r)}(a,q), 0\leq j\leq\nu-1}}f_{k,e,\Delta}^{(1)}(a,q).\numberthis\label{T} \]
We define the lattice as follows
\[\mathfrak{L}=\left\{(x,y)\in\mathbb{Z}^2 : d_j^k|L_j^{(r)}(x,y),\ 0\leq j\leq\nu-1 \right\}.  \]
Therefore, we have
\begin{align*}
  T(r,d_0,\ldots,d_{\nu-1}):=\sum_{\substack{(a,q)\in Q\Omega_{n,e,k,\Delta}^{(r)}\cap\mathfrak{L}\\  (a,q)\in\mathcal{A}_{m,r}}}f_{k,e,\Delta}^{(1)}(a,q)= \sum_{(u,v)\in\mathcal{A}_{m,r}} \sum_{\substack{(a,q)\in Q\Omega_{n,e,k,\Delta}^{(r)}\cap\mathfrak{L}\\  (a,q)\equiv (u,v)\pmod{m}}}f_{k,e,\Delta}^{(1)}(a,q).
\end{align*}
Let $B=\{v_1, v_2\}$ be a basis of the lattice $\mathfrak{L}$. We denote the matrix generated by the basis $B$ by $M=\begin{pmatrix}
v_1^T & v_2^T 
\end{pmatrix}$. Using Proposition \ref{prop-10} and Remark \ref{rem-4}, the innermost sum in the above identity can be expressed as
\begin{align*}
   T(r,d_0,\ldots,d_{\nu-1})=\sum_{(u,v)\in\mathcal{A}_{m,r}} \sum_{\substack{(a,q)\in M^{-1}Q\Omega_{n,e,k,\Delta}^{(r)}\cap\mathbb{Z}^2\\  \left(M\begin{pmatrix}
a & q 
\end{pmatrix}^T\right)^T\equiv (u,v)\pmod{m}}}f_{k,e,\Delta}^{(1)}\left(\left(M\begin{pmatrix}
a & q 
\end{pmatrix}^T\right)^T\right). 
\end{align*}
We define the map $\Phi : (\mathbb{Z}/m\mathbb{Z})^2\to(\mathbb{Z}/m\mathbb{Z})^2$ as $\Phi(x)=Mx$. The congruence $\left(M\begin{pmatrix}
a & q 
\end{pmatrix}^T\right)^T\equiv (u,v)\pmod{m}$ has no solution or a whole coset of the $\ker\Phi$ is a solution. Let $w_0$ be a solution of the above congruence then $\left(M\begin{pmatrix}
a & q 
\end{pmatrix}^T\right)^T\equiv (u,v)\pmod{m}$ if and only if $(a,q)\equiv w_0+\ell\pmod{m},\ \ell\in \ker\Phi$. Therefore, we obtain
\begin{align*}
   T(r,d_0,\ldots,d_{\nu-1})=\sum_{(u,v)\in\mathcal{A}_{m,r}}\sum_{\ell\in\ker\Phi} \sum_{\substack{(a,q)\in M^{-1}Q\Omega_{n,e,k,\Delta}^{(r)}\cap\mathbb{Z}^2\\ (a,q)\equiv w_0+\ell\pmod{m} }}f_{k,e,\Delta}^{(1)}\left(\left(M\begin{pmatrix}
a & q 
\end{pmatrix}^T\right)^T\right). 
\end{align*}
 To estimate the inner sum in the above identity, we apply Lemma \ref{Zaharescu lemma}
 \begin{align*}
  &T(r,d_0,\ldots,d_{\nu-1})=\frac{1}{m^2}\sum_{(u,v)\in\mathcal{A}_{m,r}}\sum_{\ell\in\ker\Phi}\iint_{M^{-1}Q\Omega_{n,e,k,\Delta}^{(r)}}f_{k,e,\Delta}^{(1)}\left(\left(M\begin{pmatrix}
x & y 
\end{pmatrix}^T\right)^T\right) dxdy\\ &+\BigOm{\left(\left\|{\frac{\partial f_1}{\partial x}}\right\|_{\infty}+\left\|{\frac{\partial f_2}{\partial y}}\right\|_{\infty}\right)\text{Area}(M^{-1}Q\Omega_{n,e,k,\Delta}^{(r)})}+\BigOm{\|f_1\|_{\infty}(1+\text{length}(\partial M^{-1}Q\Omega_{n,e,k,\Delta}^{(r)})}\\
&=\frac{\#\mathcal{A}_{m,r}\#\ker\Phi}{m^2}\iint_{M^{-1}Q\Omega_{n,e,k,\Delta}^{(r)}}f_{k,e,\Delta}^{(1)}\left(\left(M\begin{pmatrix}
x & y 
\end{pmatrix}^T\right)^T\right) dxdy\\ &+\BigOm{\left(\left\|{\frac{\partial f_1}{\partial x}}\right\|_{\infty}+\left\|{\frac{\partial f_2}{\partial y}}\right\|_{\infty}\right)\frac{\text{Area}(Q\Omega_{n,e,k,\Delta}^{(r)})}{\det(\mathfrak{L})}}+\BigOm{\|f_1\|_{\infty}\cdot\frac{Q}{r}}.
 \end{align*}
By the following change of variables 
\[\begin{pmatrix}
rx\\ ry 
\end{pmatrix}=M^{-1}\begin{pmatrix}
x_1 \\ y_1 
\end{pmatrix},\hspace{5mm}  dxdy=\frac{dx_1dy_1}{r^2\det(\mathfrak{L})},\]
 we obtain
\begin{align*}
  T(r,d_0,\ldots,d_{\nu-1})&=\frac{\#\mathcal{A}_{m,r}\#\ker\Phi}{r^2m^2\det(\Lambda)}\iint_{Q\Omega_{n,e,k,\Delta}}f_{k,e,\Delta}(x_1,y_1) dx_1dy_1\\& +\BigOm{\left(\left\|{\frac{\partial f_1}{\partial x}}\right\|_{\infty}+\left\|{\frac{\partial f_2}{\partial y}}\right\|_{\infty}\right)\frac{Q^2}{r\det(\mathfrak{L})}}+\BigOm{\|f_1\|_{\infty}\cdot\frac{Q}{r}}. 
\end{align*}
Inserting the above estimate into \eqref{k102}, we have
\begin{align*}
   S^{(\nu)}_{Q,k}(f)&= \frac{1}{m^2}\sum_{\substack{r, d_0^k\leq Q\\d_j^k\leq (QA_j-B_j)\\1\leq j\leq\nu-1\\ (m,rd_0d_1\cdots d_{\nu-1})=1}}\frac{\mu(r)\mu(d_0)\mu(d_1)\cdots\mu(d_{\nu-1})\#\mathcal{A}_{m,r}\#\ker\Phi}{r^2\det(\mathfrak{L})}\\ &\times\iint_{Q\Omega_{n,e,k,\Delta}}f_{k,e,\Delta}(x_1,y_1) dx_1dy_1+\BigOm{Q^{1+\frac{\nu}{k}}\log Q}. \numberthis\label{k64}
\end{align*}
By invoking the above estimates in \eqref{k65}, and making the change of variables $(u,v)=(Qx,Qy)$ in the main term of \eqref{k64}, we obtain
\begin{align*}
    S^{(\nu)}_{Q,k}&= \frac{Q^2}{m^2}\sum_{\substack{1\leq n_j\leq \mathcal{C}{(\Lambda,k,m)}\\ 1\leq j\leq\nu-1}}\mu(n_1)\cdots\mu(n_{\nu-1})\sum_{\substack{1\leq\Delta_j\leq(\nu-1)\mathcal{C}{(\Lambda,k,m)}\\1\leq e_j\leq \nu\mathcal{C}{(\Lambda,k,m)}\\ 1\leq j\leq\nu-1}}\mathcal{V}(m,k,A_j,B_j)\mathcal{I}_{k}(e,\Delta)\\&+\BigOm{Q^{1+\frac{\nu}{k}}\log Q},\numberthis\label{h6}
\end{align*}
where
\[\mathcal{V}(m,k,A_j,B_j)=\sum_{\substack{r, d_0^k\leq Q\\d_j^k\leq (QA_j-B_j)\\1\leq j\leq\nu-1\\ \gcd(m,rd_0d_1\cdots d_{\nu-1})=1}}\frac{\mu(r)\mu(d_0)\mu(d_1)\cdots\mu(d_{\nu-1})\#\mathcal{A}_{m,r}\#\ker\Phi}{r^2\det(\mathfrak{L})} \]
and
\[\mathcal{I}_{k}(e,\Delta)=\iint_{\Omega_{n,e,k,\Delta}}g_{k,e,\Delta}(x,y)dxdy, \numberthis\label{h7}\]
\[g_{k,e,\Delta}(x,y)=\Tilde{H}\left(g_{k,e,\Delta}^{(1)}(x,y),\ldots,g_{k,e,\Delta}^{(\nu-1)}(x,y)\right)\ \text{and}\ g_{k,e,\Delta}^{(j)}(x,y)=\frac{\mathcal{N}{(Q,k,m)}\Delta_j}{Q^2y(ye_j-x\Delta_j)},\  j=1,\ldots,\nu-1. \]
Employing Proposition \ref{prop1} and the inequality
\[|\Tilde{H}(v)-\Tilde{H}(w)|\leq \|\Tilde{H}^{\prime} \||v-w|\leq 2\|H^{\prime}\| |v-w|, \]
we observe that \eqref{h6} holds true when $g_{k,e,\Delta}^{(j)}$ is replaced by
\[g_{k,e,\Delta}^{(j)}(x,y)=\frac{\mathscr{C}(k,m)\Delta_j}{y(ye_j-x\Delta_j)},\ j=1,\ldots,\nu-1, \]
in the formula for $g_{k,e,\Delta}$. Therefore
\begin{align*}
    S^{(\nu)}_{Q,k}&= \frac{Q^2}{m^2}\sum_{\substack{1\leq n_j\leq \mathcal{C}{(\Lambda,k,m)}\\1\leq j\leq\nu-1}}\mu(n_1)\cdots\mu(n_{\nu-1})\sum_{\substack{1\leq\Delta_j\leq(\nu-1)\mathcal{C}{(\Lambda,k,m)}\\1\leq e_j\leq \nu\mathcal{C}{(\Lambda,k,m)}\\1\leq j\leq\nu-1}}\mathcal{V}(m,k,A_j,B_j)\mathcal{I}_{k}(e,\Delta)+\BigOm{Q^{1+\frac{\nu}{k}}\log Q},
\end{align*}
where $\mathcal{I}_{k}(r,e,\Delta)$ is as in \eqref{h7}. We also note that the region can be extended to
\[\Tilde{\Omega}_{n,e,k,\Delta}=\left\{(x,y): 0<x\leq y\leq 1,\ y\geq \frac{1}{\mathcal{C}{(\Lambda,k,m)}},\ 0< ye_j-x\Delta_j\leq\frac{1}{n_j},\ \Psi_k(n_j(ye_j-x\Delta_j))=1 \right\}. \]
If $(x,y)\in\Tilde{\Omega}_{n,e,k,\Delta}\setminus\Omega_{n,e,k,\Delta} $, there is some $j$ such that $|ye_j-x\Delta_j|<1/n_j\mathcal{C}{(\Lambda,k,m)}$, which implies that
\[|g_{k,e,\Delta}^{(j)}(x,y)|\geq n_j\Delta_j\mathscr{C}(k,m)\mathcal{C}{(\Lambda,k,m)} \geq \mathscr{C}(k,m)\mathcal{C}{(\Lambda,k,m)}=\Lambda. \]
This in turn implies that $g_{k,e,\Delta}=0$ on $\Tilde{\Omega}_{n,e,k,\Delta}\setminus\Omega_{n,e,k,\Delta}$. Therefore
\begin{align*}
    S^{(\nu)}_{Q,k}&= \frac{Q^2}{m^2}\sum_{1\leq n_j\leq \mathcal{C}{(\Lambda,k,m)}}\mu(n_1)\cdots\mu(n_{\nu-1})\sum_{\substack{1\leq\Delta_j\leq(\nu-1)\mathcal{C}{(\Lambda,k,m)}\\1\leq e_j\leq \nu\mathcal{C}{(\Lambda,k,m)}}}\mathcal{V}(m,k,A_j,B_j)\\&\times\iint_{\Tilde{\Omega}_{n,e,k,\Delta}}g_{k,e,\Delta}(x,y)dxdy+\BigOm{Q^{1+\frac{\nu}{k}}\log Q}.\numberthis\label{h8} 
\end{align*}

On taking $A_j=e_jn_j,\ B_j=\Delta_jn_j,\ A=(A_1,\ldots,A_{\nu-1}) $ and $B=(B_1,\ldots,B_{\nu-1})$ and considering the region $\Omega_{A,B,\Lambda,k}$ and map $T_{A,B}$. We set
\[\mathcal{I}_{k,\Lambda}(A,B)=\iint_{\Omega_{A,B,\Lambda,k}}\Tilde{H}\circ T_{A,B}. \]Therefore \eqref{h8} becomes
\begin{align*}
    S^{(\nu)}_{Q,k}&= \frac{Q^2}{m^2}\sum_{\substack{1\leq A_j\leq\nu\mathcal{C}^2{(\Lambda,k,m)}\\1\leq B_j\leq (\nu-1)\mathcal{C}^2{(\Lambda,k,m)}}}\sum_{n_j|\gcd(A_j,B_j)}\mu(n_1)\cdots\mu(n_{\nu-1})\mathcal{V}(m,k,A_j,B_j)\mathcal{I}_{k,\Lambda}(A,B)\\&+\BigOm{Q^{1+\frac{\nu}{k}}\log Q}\\
    &=\frac{Q^2}{m^2}\sum_{\substack{1\leq A_j\leq\nu\mathcal{C}^2{(\Lambda,k,m)}\\1\leq B_j\leq (\nu-1)\mathcal{C}^2{(\Lambda,k,m)}\\ (A_j,B_j)=1\\1\leq j\leq\nu-1}}\mathscr{V}(m,k,A,B)\mathcal{I}_{k,\Lambda}(A,B)+\BigOm{Q^{1+\frac{\nu}{k}}\log Q}+\BigOm{Q^{\frac{3}{2}+\epsilon}}.
\end{align*}
Using Proposition \ref{prop1} with the above formula yields
\begin{align*}
  \frac{S^{(\nu)}_{Q,k}}{\mathcal{N}{(Q,k,m)}}&=\frac{1}{m^2\mathscr{C}(k,m)}\sum_{\substack{1\leq A_j\leq\nu\mathcal{C}^2{(\Lambda,k,m)}\\1\leq B_j\leq (\nu-1)\mathcal{C}^2{(\Lambda,k,m)}\\(A_j,B_j)=1\\ 1\leq j\leq\nu-1}}\mathscr{V}(m,k,A,B)\mathcal{I}_{k,\Lambda}(A,B)+\BigOm{Q^{\frac{-1}{2}+\epsilon}}\\&+\BigOm{Q^{-1+\frac{\nu}{k}}\log Q}.  
\end{align*}

By the standard approximation argument, we next approximate the characteristic function of a box $\mathfrak{B}\in(0,\Lambda)^{\nu-1}$ from below and above by the smooth functions with compact support in $(0,\Lambda)$
\begin{align*}
   \mathcal{S}^{\nu}(\mathfrak{B})&=\lim_{Q\to\infty}\frac{S^{(\nu)}_{Q,k}}{\mathcal{N}{(Q,k,m)}}=\frac{1}{m^2\mathscr{C}(k,m)}\sum_{\substack{1\leq A_j\leq(\nu-1)\mathcal{C}^2{(\Lambda,k)}\\1\leq B_j\leq \nu\mathcal{C}^2{(\Lambda,k)}\\(A_j,B_j)=1\\ 1\leq j\leq\nu-1}}\mathscr{V}(m,k,A,B)\iint_{\Omega_{A,B,\Lambda,k}}\chi_{\mathcal{B}}\circ T\circ T_{A,B} \\
   &=\frac{1}{m^2\mathscr{C}(k,m)}\sum_{\substack{1\leq A_j\leq(\nu-1)\mathcal{C}^2{(\Lambda,k)}\\1\leq B_j\leq \nu\mathcal{C}^2{(\Lambda,k)}\\(A_j,B_j)=1\\ 1\leq j\leq\nu-1}}\mathscr{V}(m,k,A,B)\text{area}\left(\Omega_{A,B,\Lambda,k}\cap T_{A,B}^{-1}(T^{-1}\mathfrak{B}) \right).
\end{align*}
This completes the proof of Theorem \ref{v correlation}.
\section{Pair correlation }
In this final section, we prove Theorem \ref{main result} by combining estimates on exponential sums over elements in $\mathfrak{F}_{Q,k}^{(m)}$, and invoking the key Lemma \ref{key lemma} on weighted lattice point counting.
  \subsection{Proof of Theorem \ref{main result}}
 To prove Theorem \ref{main result}, we need to estimate, for any positive real number $\Lambda$, the quantity
\[S_{\mathfrak{F}_{Q,k}^{(m)}}^{(2)}(\Lambda)=\frac{1}{\mathcal{N}{(Q,k,m)}}\#\{(\gamma_1,\gamma_2)\in \left(\mathfrak{F}_{Q,k}^{(m)}\right)^2: \gamma_1\ne\gamma_2, \gamma_1-\gamma_2\in\frac{1}{\mathcal{N}{(Q,k,m)}}(0,\Lambda)+\mathbb{Z}\},\numberthis\label{S lambda}\]
as $Q\to \infty.$ 
Let $H$ be any continuously differentiable function with Supp\ $H\subset(0,\Lambda)$. To estimate \eqref{S lambda}, we consider \eqref{k66} with $\nu=2$. We obtain
\[S_{Q,k}^{(2)}
     =\sum_{r\in \mathbb{Z}}c_r\sum_{\gamma_1\in \mathfrak{F}_{Q,k}^{(m)}}e(r\gamma_1)\sum_{\gamma_2\in \mathfrak{F}_{Q,k}^{(m)}}e(r\gamma_2).\numberthis\label{exp}\]
     We employ Lemma \ref{exp nu} into the above identity and express it as
\begin{align*}
    S_{Q,k}^{(2)}&=\sum_{r\in\mathbb{Z}}c_r\sum_{d_1,d_2\leq Q}\mu(d_1)\mu(d_2)\sum_{\substack{q_1\leq\frac{Q}{d_1}, q_2\leq\frac{Q}{d_2}\\ [q_1,q_2]|r\\ q_1d_1\equiv b\pmod{m}\\ q_2d_2\equiv b\pmod{m} }}q_1q_2\mu_k(q_1d_1)^2\mu_k(q_2d_2)^2\\
&=\sum_{d_1,d_2\leq Q}\mu(d_1)\mu(d_2)\sum_{\substack{q_1\leq\frac{Q}{d_1}, q_2\leq\frac{Q}{d_2}\\ q_1d_1\equiv b\pmod{m}\\ q_2d_2\equiv b\pmod{m} }}q_1q_2\mu_k(q_1d_1)^2\mu_k(q_2d_2)^2\sum_{r\in\mathbb{Z}}c_{r[q_1,q_2]},\numberthis\label{k35}
\end{align*}
where $[q_1,q_2]$ is least common multiple of $q_1$ and $q_2$. By using \cite[(3.4)]{Chaubey}, we obtain
\begin{align*}
\sum_{r\in\mathbb{Z}}c_{[q_1,q_2]r}&
    =\sum_{r\in\mathbb{Z}}\frac{1}{[q_1,q_2]}H\left(\frac{r\mathcal{N}{(Q,k,m)}}{[q_1,q_2]} \right).\numberthis\label{k36}
\end{align*}
Using the above identity into \eqref{k35}, we obtain
\begin{align*}
   S_{Q,k}^{(2)} &= \sum_{d_1,d_2\leq Q}\mu(d_1)\mu(d_2)\sum_{\substack{q_1\leq\frac{Q}{d_1}, q_2\leq\frac{Q}{ d_2}\\ q_1d_1\equiv b\pmod{m}\\ q_2d_2\equiv b\pmod{m} }}\gcd(q_1,q_2)\mu_k(q_1d_1)^2\mu_k(q_2d_2)^2\sum_{r\in\mathbb{Z}}H\left(\frac{r\mathcal{N}{(Q,k,m)}}{[q_1,q_2]} \right)\\
   &=\sum_{\delta\leq Q}\delta \sum_{d_1,d_2\leq \frac{Q}{\delta}}\mu(d_1)\mu(d_2)\sum_{\substack{q_1\leq\frac{Q}{\delta d_1}, q_2\leq\frac{Q}{\delta d_2}\\ q_1\delta d_1\equiv b\pmod{m}\\ q_2\delta d_2\equiv b\pmod{m} \\ (q_1,q_2)=1}}\mu_k(q_1d_1\delta)^2\mu_k(q_2d_2\delta)^2\sum_{r\in\mathbb{Z}}H\left(\frac{r\mathcal{N}{(Q,k,m)}}{q_1q_2\delta} \right).\numberthis\label{k51}
\end{align*}
 For the non-zero contribution from $H$, using the fact that Supp$H \subset (0,\Lambda)$ and Proposition \ref{prop1}, one must have
\[0<\frac{\mathcal{N}{(Q,k,m)}r}{q_1q_2\delta}<\Lambda,\numberthis\label{nonzero contribution}\] which implies 
\[\delta d_1d_2r<\frac{\Lambda}{{\mathscr{C}(k,m)}}=:\mathcal{C}{(\Lambda,k,m)}.\]
By applying the above estimate and observing that
\[H\left(\frac{\mathcal{N}{(Q,k,m)}r}{q_1
q_2\delta} \right)=H\left(\frac{Q^2\mathscr{C}(k,m)r}{q_1q_2\delta} \right)+\BigOm{\frac{r}{q_1q_2\delta}Q^{\frac{2(2k-1)}{3k-2}}},\]
the sum in \eqref{k51} can be expressed as 
\begin{align*}
    S_{Q,k}^{(2)} &=\sum_{\substack{d_1,d_2,\delta, r\geq 1\\\delta d_1d_2r<\mathcal{C}{(\Lambda,k,m)}}}\delta \mu(d_1)\mu(d_2)\sum_{\substack{q_1\leq\frac{Q}{\delta d_1}, q_2\leq\frac{Q}{\delta d_2}\\ q_1\delta d_1\equiv b\pmod{m}\\ q_2\delta d_2\equiv b\pmod{m} \\ (q_1,q_2)=1}}\mu_k(q_1d_1\delta)^2\mu_k(q_2d_2\delta)^2H\left(\frac{Q^2\mathscr{C}(k,m)r}{q_1q_2\delta} \right)\\ &+\BigOm{Q^{\frac{2(2k-1)}{3k-2}}(\log Q)^2}\\
    &=\frac{1}{\phi^2(m)}\sum_{\substack{\chi(\text{mod}\ {m})\\\chi^{\prime}(\text{mod}\ {m})}}\sum_{\substack{d_1, d_2,\delta, r\geq 1\\\delta d_1d_2r<{\mathcal{C}(\Lambda,k,m)}}}\delta {\chi}(\delta d_1\bar{b}){\chi^{\prime}}(\delta d_2\bar{b})\mu_k(\delta)^2\mu(d_1)\mu(d_2)\sum_{\substack{q_1\leq\frac{Q}{\delta d_1}, q_2\leq\frac{Q}{\delta d_2}\\(q_1,q_2)=1 }}\chi(q_1)\chi^{\prime}(q_2)\\&\times\mu_k(q_1d_1\delta)^2\mu_k(q_2d_2\delta)^2 H\left(\frac{Q^2\mathscr{C}(k,m)r}{q_1q_2\delta} \right)+\BigOm{Q^{\frac{2(2k-1)}{3k-2}}(\log Q)^2}.\numberthis\label{k61}
\end{align*}
Next, we deal with the cases of principal and non-principal characters separately.

Case-I: If $\chi=\chi_0$ and $\chi^{\prime}=\chi_0^{\prime}$ then we have
\begin{align*}
   S_{Q,k}^{(2)}(\chi_0,\chi_0^{\prime})&=\sum_{\substack{d_1, d_2,\delta, r\geq 1\\\delta d_1d_2r<{\mathcal{C}(\Lambda,k,m)}\\(d_1d_2\delta,m)=1}}\delta\mu_k(\delta)^2\mu(d_1)\mu(d_2)\\&\times\sum_{\substack{q_1\leq\frac{Q}{\delta d_1}, q_2\leq\frac{Q}{\delta d_2}\\ (q_1q_2,m)=1=(q_1,q_2) }}\mu_k(q_1d_1\delta)^2\mu_k(q_2d_2\delta)^2 H\left(\frac{Q^2\mathscr{C}(k,m)r}{q_1q_2\delta} \right)\numberthis\label{k57}.  
\end{align*}
To estimate the inner sum in the above identity, we employ Lemma \ref{key lemma} which counts the $k$-free lattice points with some weight and congruence constraints. Note that, since Supp $H\subset(0,\Lambda)$, for the non-zero contribution from $H$, one has $0<\frac{Q^2\mathscr{C}(k,m)r}{x_1x_2\delta}<\Lambda$. For $0<x_1\leq \frac{Q}{\delta d_1}$ and $0<x_2\leq \frac{Q}{\delta d_2}$, we obtain
\[\frac{1}{x_1}\leq \frac{\mathcal{C}{(\Lambda,k,m)}}{rd_2Q}\ \text{and}\ \frac{1}{x_2}\leq \frac{\mathcal{C}{(\Lambda,k,m)}}{rd_1Q}. \numberthis\label{x_1}\]
Using \eqref{x_1} and the necessary condition for the non-zero contribution of $H$,  we get
\[\left|\frac{\partial H}{\partial x_1}(x_1,x_2)\right|\ll\frac{1}{Q}\ \text{and}\ \left|\frac{\partial H}{\partial x_2}(x_1,x_2)\right|\ll\frac{1}{Q}.\]
Hence
\[\|DH\|_{\infty}\ll\frac{1}{Q}.\]
Employing Lemma \ref{key lemma} with $r_1=r_2=m,\ \delta_1=d_1\delta,\ \delta_2=d_2\delta$, and $f(a,b)=H\left(\frac{Q^2\mathscr{C}(k,m)r}{ab\delta} \right)$, the inner-sum in \eqref{k57} is expressed as
\begin{align*}
    \sum_{\substack{q_1\leq\frac{Q}{\delta d_1}, q_2\leq\frac{Q}{\delta d_2}\\ (q_1q_2,m)=1, (q_1,q_2)=1 }}&\mu_k(q_1d_1\delta)^2\mu_k(q_2d_2\delta)^2H\left(\frac{Q^2\mathscr{C}(k,m)r}{q_1q_2\delta} \right)\\&=\frac{6P_{m,m}^{k}(d_1\delta,d_2\delta)}{\pi^2}\int_0^{\frac{Q}{\delta d_1}}\int_0^{\frac{Q}{\delta d_2}}  H\left(\frac{Q^2\mathscr{C}(k,m)r}{xy\delta} \right)dxdy+\BigOm{\tau(m)Q^{1+\frac{1}{k}}\log^2Q}.\numberthis\label{k4}
\end{align*}
Using the fact that Supp$H\subset(0,\lambda)$ and by a suitable change of variable the integral in \eqref{k4} can be expressed as
\[\int_0^{\frac{Q}{\delta d_1}}\int_0^{\frac{Q}{\delta d_2}}H\left(\frac{Q^2\mathscr{C}(k,m)r}{xy\delta} \right)dxdy=\frac{Q^2\mathscr{C}(k,m)r}{\delta}\int_{r\delta d_1d_2\mathscr{C}(k,m)}^{\Lambda}\frac{H(\lambda)}{\lambda^2}\log\left(\frac{\lambda}{r\delta d_1d_2\mathscr{C}(k,m)} \right)d\lambda. \]
The above identity with \eqref{k57} and \eqref{k4} gives
\begin{align*}
    S_{Q,k}^{(2)}(\chi_0,\chi_0^{\prime})&=\frac{6Q^2\mathscr{C}(k,m)}{\pi^2}\sum_{\substack{d_1, d_2,\delta, r\geq 1\\\delta d_1d_2r<{\mathcal{C}(\Lambda,k,m)}\\(d_1d_2\delta,m)=1}}r\mu_k(\delta)^2\mu(d_1)\mu(d_2)P_{m,m}^{k}(d_1\delta,d_2\delta)\int_{r\delta d_1d_2\mathscr{C}(k,m)}^{\Lambda}\frac{H(\lambda)}{\lambda^2}\\&\times\log\left(\frac{\lambda}{r\delta d_1d_2\mathscr{C}(k,m)} \right)d\lambda+\BigOm{Q^{1+\frac{1}{k}}\log^2Q}\\
    &=\frac{6Q^2\mathscr{C}(k,m)}{\pi^2}\sum_{1\leq n<\mathcal{C}(\Lambda,k,m)}\int_{n\mathscr{C}(k,m)}^{\Lambda}\frac{H(\lambda)}{\lambda^2}\log\left(\frac{\lambda}{n\mathscr{C}(k,m)} \right)d\lambda\\&\times\sum_{\substack{\delta d_1d_2r=n\\(d_1d_2\delta,m)=1}}r\mu_k(\delta)^2\mu(d_1)\mu(d_2)P_{m,m}^{k}(d_1\delta,d_2\delta)+\BigOm{Q^{1+\frac{1}{k}}\log^2Q}\\
    &=\frac{6Q^2\mathscr{C}(k,m)}{\pi^2}\int_{0}^{\Lambda}\frac{H(\lambda)}{\lambda^2}\sum_{1\leq n<\mathcal{C}(\Lambda,k,m)}F_k(n)\log\left(\frac{\lambda}{n \mathscr{C}(k,m)} \right)d\lambda+\BigOm{Q^{1+\frac{1}{k}}\log^2Q},\numberthis\label{k62}
\end{align*}
where
\begin{align*}
    F_k(n)=\sum_{\substack{\delta d_1d_2r=n\\(d_1d_2\delta,m)=1}}r\mu_k(\delta)^2\mu(d_1)\mu(d_2)P_{m,m}^{k}(d_1\delta,d_2\delta).
\end{align*}

Case-II: Suppose at least one of $\chi$ or $\chi^{\prime}$ is non-principal.
\begin{align*}
  S_{Q,k}^{(2)}(\chi,\chi^{\prime})&= \sum_{\substack{d_1, d_2,\delta, r\geq 1\\\delta d_1d_2r<{\mathcal{C}(\Lambda,k,m)}}}\delta {\chi}(\delta d_1\bar{b}){\chi^{\prime}}(\delta d_2\bar{b})\mu_k(\delta)^2\mu(d_1)\mu(d_2)\sum_{\substack{q_1\leq\frac{Q}{\delta d_1}, q_2\leq\frac{Q}{\delta d_2}\\ (q_1,q_2)=1 }}\chi(q_1)\chi^{\prime}(q_2)\\&\times \mu_k(q_1d_1\delta)^2\mu_k(q_2d_2\delta)^2H\left(\frac{Q^2\mathscr{C}(k,m)r}{q_1q_2\delta} \right)\\
  &= \sum_{\substack{d_1, d_2,\delta, r\geq 1\\\delta d_1d_2r<{\mathcal{C}(\Lambda,k,m)}}}\delta {\chi}(\delta d_1\bar{b}){\chi^{\prime}}(\delta d_2\bar{b})\mu_k(\delta)^2\mu(d_1)\mu(d_2)\sum_{\substack{q_1\leq\frac{Q}{\delta d_1}}}\chi(q_1)\mu_k(q_1d_1\delta)^2\\&\times\sum_{\substack{q_2\leq\frac{Q}{\delta d_2}\\(q_2, q_1)=1}}\chi^{\prime}(q_2)\mu_k(q_2d_2\delta)^2 H\left(\frac{Q^2\mathscr{C}(k,m)r}{q_1q_2\delta} \right).\numberthis\label{k42}
\end{align*}
In order to estimate the inner sum in the above identity, we use Proposition \ref{prop15} with $f\left(\frac{M}{q_2} \right)=H\left(\frac{Q^2\mathscr{C}(k,m)r}{q_1q_2\delta} \right)$ and obtain
\[\sum_{\substack{q_2\leq\frac{Q}{\delta d_2}\\(q_2, q_1)=1}}\chi^{\prime}(q_2)\mu_k(q_2d_2\delta)^2 H\left(\frac{Q^2\mathscr{C}(k,m)r}{q_1q_2\delta} \right)\ll_{m} \tau(q_1)\left(\frac{Q}{\delta d_2}\right)^{\frac{1}{k}}\log \frac{Q}{\delta d_2}, \]
and this in conjunction with \eqref{k42} yields
\[S_{Q,k}^{(2)}(\chi,\chi^{\prime})\ll_{m} Q^{1+\frac{1}{k}+\epsilon}\log Q. \numberthis\label{k27}\]
We collect the estimates from \eqref{k62} and \eqref{k27} and insert them into \eqref{k61}. We obtain
\begin{align*}
   S_{Q,k}^{(2)} &=\frac{6Q^2\mathscr{C}(k,m)}{\pi^2\phi^2(m)}\int_{0}^{\Lambda}\frac{H(\lambda)}{\lambda^2}\sum_{1\leq n<\mathcal{C}(\Lambda,k,m)}F_k(n)\log\left(\frac{\lambda}{n \mathscr{C}(k,m)} \right)d\lambda+\BigOm{Q^{1+\frac{1}{k}+\epsilon}\log^2Q}\\
   &=Q^2\mathscr{C}(k,m)\int_{0}^{\Lambda}H(\lambda)\mathfrak{g}_{m,k}(\lambda)d\lambda+\BigOm{Q^{1+\frac{1}{k}+\epsilon}\log^2Q}.
\end{align*}
Therefore
\[\frac{S_{Q,k}^{(2)}}{\mathcal{N}(Q,k)}=\int_{0}^{\Lambda}H(\lambda)\mathfrak{g}_{m,k}(\lambda)d\lambda+\BigOm{Q^{-1+\frac{1}{k}+\epsilon}\log^2Q}. \]
Next, we use the standard approximation argument to approximate the characteristic function of the interval $(0,\Lambda)$ from above and below by the smooth functions with compact support in $(0,\Lambda)$ and obtain
\[\mathcal{S}^{2}((0,\Lambda))=\lim_{Q\to\infty}S_{\mathfrak{F}_{Q,k}^{(m)}}^{(2)}(\Lambda)=\int_0^{\Lambda}\mathfrak{g}_{m,k}(\lambda)d\lambda.\]
This completes the proof of Theorem \ref{main result}.
\begin{data availability}
No data was used for the research described in the article.
\end{data availability}
\begin{conflict}
The authors have no conflicts of interest to declare. All co-authors have seen and agree with the contents of the article and there is no financial interest to report.
\end{conflict}

\bibliographystyle{plain}
\bibliography{reference}

@book {Tenenbaum,
    AUTHOR = {Tenenbaum, G.},
     TITLE = {Introduction to analytic and probabilistic number theory},
    SERIES = {Graduate Studies in Mathematics},
    VOLUME = {163},
   EDITION = {Third},
      NOTE = {Translated from the 2008 French edition by Patrick D. F. Ion},
 PUBLISHER = {American Mathematical Society, Providence, RI},
      YEAR = {2015},
     PAGES = {xxiv+629},
      ISBN = {978-0-8218-9854-3},
   MRCLASS = {11-02 (11Kxx 11Mxx 11Nxx)},
  MRNUMBER = {3363366},
       DOI = {10.1090/gsm/163},
       URL = {https://doi.org/10.1090/gsm/163},
}

@book {Davenport,
    AUTHOR = {Davenport, H.},
     TITLE = {Multiplicative number theory},
    SERIES = {Graduate Texts in Mathematics},
    VOLUME = {74},
   EDITION = {Third},
      NOTE = {Revised and with a preface by Hugh L. Montgomery},
 PUBLISHER = {Springer-Verlag, New York},
      YEAR = {2000},
     PAGES = {xiv+177},
      ISBN = {0-387-95097-4},
   MRCLASS = {11-02 (11-01 11Mxx 11Nxx)},
  MRNUMBER = {1790423},
}

@article {Franel,
    AUTHOR = {Franel, J.},
     TITLE = {Les suites de Farey et le problème des nombres premiers},
   JOURNAL = {Göttinger Nachr.},
  FJOURNAL = {Nachrichten von der Gesellschaft der Wissenschaften zu Göttingen, Mathematisch-Physikalische Klasse},
      YEAR = {1924},
     PAGES = {198--201},
       URL = {http://eudml.org/doc/59156},
}

@article {Landau,
  AUTHOR = {Landau, E.},
   TITLE = {Bemerkungen zu der vorstehenden Abhandlung von Herrn Franel},
 JOURNAL = {Göttinger Nachr.},
FJOURNAL = {Nachrichten von der Gesellschaft der Wissenschaften zu Göttingen, Mathematisch-Physikalische Klasse},
    YEAR = {1924},
   PAGES = {202--206},
}

@article {BocaF,
    AUTHOR = {Boca, F. P. and Zaharescu, A.},
     TITLE = {The correlations of {F}arey fractions},
   JOURNAL = {J. London Math. Soc. (2)},
  FJOURNAL = {Journal of the London Mathematical Society. Second Series},
    VOLUME = {72},
      YEAR = {2005},
    NUMBER = {1},
     PAGES = {25--39},
      ISSN = {0024-6107},
   MRCLASS = {11J71 (11B57 11K36)},
  MRNUMBER = {2145726},
MRREVIEWER = {Dmitry Y. Kleinbock},
       DOI = {10.1112/S0024610705006629},
       URL = {https://doi.org/10.1112/S0024610705006629},
}

@article {Xiong,
    AUTHOR = {Xiong, M. and Zaharescu, A.},
     TITLE = {Pair correlation of rationals with prime denominators},
   JOURNAL = {J. Number Theory},
  FJOURNAL = {Journal of Number Theory},
    VOLUME = {128},
      YEAR = {2008},
    NUMBER = {10},
     PAGES = {2795--2807},
      ISSN = {0022-314X},
   MRCLASS = {11J71 (11K06)},
  MRNUMBER = {2441078},
MRREVIEWER = {Geumlan Choi},
       DOI = {10.1016/j.jnt.2007.12.006},
       URL = {https://doi.org/10.1016/j.jnt.2007.12.006},
}

@article {Zaharescu,
    AUTHOR = {Xiong, M. and Zaharescu, A.},
     TITLE = {Correlation of fractions with divisibility constraints},
   JOURNAL = {Math. Nachr.},
  FJOURNAL = {Mathematische Nachrichten},
    VOLUME = {284},
      YEAR = {2011},
    NUMBER = {2-3},
     PAGES = {393--407},
      ISSN = {0025-584X},
   MRCLASS = {11J71 (11B57 11K36 11L05)},
  MRNUMBER = {2790897},
MRREVIEWER = {Alan Haynes},
       DOI = {10.1002/mana.200710185},
       URL = {https://doi.org/10.1002/mana.200710185},
}

@article {Siskaki,
    AUTHOR = {Boca, F. P. and Siskaki, M.},
     TITLE = {A note on the pair correlation of {F}arey fractions},
   JOURNAL = {Acta Arith.},
  FJOURNAL = {Acta Arithmetica},
    VOLUME = {205},
      YEAR = {2022},
    NUMBER = {2},
     PAGES = {121--135},
      ISSN = {0065-1036},
   MRCLASS = {11K38 (11B57 11J71)},
  MRNUMBER = {4493802},
       DOI = {10.4064/aa210926-30-8},
       URL = {https://doi.org/10.4064/aa210926-30-8},
}

@article {Chaubey,
    AUTHOR = {Chahal, B. and Chaubey, S.},
     TITLE = {Pair correlation of {F}arey fractions with square-free
              denominators},
   JOURNAL = {Acta Arith.},
  FJOURNAL = {Acta Arithmetica},
    VOLUME = {215},
      YEAR = {2024},
    NUMBER = {4},
     PAGES = {289--307},
      ISSN = {0065-1036},
   MRCLASS = {11B57 (11J71)},
  MRNUMBER = {4799171},
       DOI = {10.4064/aa230329-6-6},
       URL = {https://doi.org/10.4064/aa230329-6-6},
}

@article {Neville,
    AUTHOR = {Neville, E. H.},
     TITLE = {The structure of {F}arey series},
   JOURNAL = {Proc. London Math. Soc. (2)},
  FJOURNAL = {Proceedings of the London Mathematical Society. Second Series},
    VOLUME = {51},
      YEAR = {1949},
     PAGES = {132--144},
      ISSN = {0024-6115},
   MRCLASS = {10.0X},
  MRNUMBER = {29924},
MRREVIEWER = {W. H. Simons},
       DOI = {10.1112/plms/s2-51.2.132},
       URL = {https://doi.org/10.1112/plms/s2-51.2.132},
}

@article {Niederreiter,
    AUTHOR = {Niederreiter, H.},
     TITLE = {The distribution of {F}arey points},
   JOURNAL = {Math. Ann.},
  FJOURNAL = {Mathematische Annalen},
    VOLUME = {201},
      YEAR = {1973},
     PAGES = {341--345},
      ISSN = {0025-5831},
   MRCLASS = {10F05},
  MRNUMBER = {332666},
MRREVIEWER = {R. A. Rankin},
       DOI = {10.1007/BF01428199},
       URL = {https://doi.org/10.1007/BF01428199},
}

@article {Dress,
    AUTHOR = {Dress, F.},
     TITLE = {Discr\'{e}pance des suites de {F}arey},
   JOURNAL = {J. Th\'{e}or. Nombres Bordeaux},
  FJOURNAL = {Journal de Th\'{e}orie des Nombres de Bordeaux},
    VOLUME = {11},
      YEAR = {1999},
    NUMBER = {2},
     PAGES = {345--367},
      ISSN = {1246-7405},
   MRCLASS = {11K38 (11B57)},
  MRNUMBER = {1745884},
MRREVIEWER = {Gerhard Larcher},
       URL = {http://jtnb.cedram.org/item?id=JTNB_1999__11_2_345_0},
}

@article {MR2273359,
    AUTHOR = {Alkan, E. and Ledoan, A. H. and V\^{a}j\^{a}itu, M. and
              Zaharescu, A.},
     TITLE = {Discrepancy of fractions with divisibility constraints},
   JOURNAL = {Monatsh. Math.},
  FJOURNAL = {Monatshefte f\"{u}r Mathematik},
    VOLUME = {149},
      YEAR = {2006},
    NUMBER = {3},
     PAGES = {179--192},
      ISSN = {0026-9255},
   MRCLASS = {11K38 (11B57 11N25)},
  MRNUMBER = {2273359},
MRREVIEWER = {Peter A. Kritzer},
       DOI = {10.1007/s00605-006-0383-y},
       URL = {https://doi.org/10.1007/s00605-006-0383-y},
}

@article {MR2275343,
    AUTHOR = {Alkan, E. and Ledoan, A. H. and V\^{a}j\^{a}itu, M. and Zaharescu, A.},
     TITLE = {Discrepancy of sets of fractions with congruence constraints},
   JOURNAL = {Rev. Roumaine Math. Pures Appl.},
  FJOURNAL = {Revue Roumaine de Math\'{e}matiques Pures et Appliqu\'{e}es. Romanian
              Journal of Pure and Applied Mathematics},
    VOLUME = {51},
      YEAR = {2006},
    NUMBER = {3},
     PAGES = {265--276},
      ISSN = {0035-3965},
   MRCLASS = {11K38 (11B57 11N25)},
  MRNUMBER = {2275343},
MRREVIEWER = {Vilius Stakenas},
}

@article {MR3871604,
    AUTHOR = {Ledoan, A. H.},
     TITLE = {The discrepancy of {F}arey series},
   JOURNAL = {Acta Math. Hungar.},
  FJOURNAL = {Acta Mathematica Hungarica},
    VOLUME = {156},
      YEAR = {2018},
    NUMBER = {2},
     PAGES = {465--480},
      ISSN = {0236-5294},
   MRCLASS = {11K38 (11B57 11N25)},
  MRNUMBER = {3871604},
MRREVIEWER = {Robert F. Tichy},
       DOI = {10.1007/s10474-018-0868-x},
       URL = {https://doi.org/10.1007/s10474-018-0868-x},
}

@book {MR882550,
    AUTHOR = {Titchmarsh, E. C.},
     TITLE = {The theory of the {R}iemann zeta-function},
   EDITION = {Second},
      NOTE = {Edited and with a preface by D. R. Heath-Brown},
 PUBLISHER = {The Clarendon Press, Oxford University Press, New York},
      YEAR = {1986},
     PAGES = {x+412},
      ISBN = {0-19-853369-1},
   MRCLASS = {11M06},
  MRNUMBER = {882550},
MRREVIEWER = {Matti Jutila},
}

@book {Vaughan,
    AUTHOR = {Montgomery, H. L. and Vaughan, R. C.},
     TITLE = {Multiplicative number theory. {I}. {C}lassical theory},
    SERIES = {Cambridge Studies in Advanced Mathematics},
    VOLUME = {97},
 PUBLISHER = {Cambridge University Press, Cambridge},
      YEAR = {2007},
     PAGES = {xviii+552},
      ISBN = {978-0-521-84903-6; 0-521-84903-9},
   MRCLASS = {11-02 (11-01 11M06 11M26 11M45 11N05 11N37)},
  MRNUMBER = {2378655},
MRREVIEWER = {Wolfgang Schwarz},
}

@article {MR551704,
AUTHOR = {Kolesnik, G.},
     TITLE = {On the order of {D}irichlet {$L$}-functions},
   JOURNAL = {Pacific J. Math.},
  FJOURNAL = {Pacific Journal of Mathematics},
    VOLUME = {82},
      YEAR = {1979},
    NUMBER = {2},
     PAGES = {479--484},
      ISSN = {0030-8730},
   MRCLASS = {10H08 (10G10)},
  MRNUMBER = {551704},
MRREVIEWER = {D. R. Heath-Brown},
       URL = {http://projecteuclid.org/euclid.pjm/1102784888},
}

@book {Edwards,
    AUTHOR = {Edwards, H. M.},
     TITLE = {Riemann's zeta function},
    SERIES = {Pure and Applied Mathematics, Vol. 58},
 PUBLISHER = {Academic Press [Harcourt Brace Jovanovich, Publishers], New
              York-London},
      YEAR = {1974},
     PAGES = {xiii+315},
   MRCLASS = {10H05},
  MRNUMBER = {466039},
MRREVIEWER = {Bruce C. Berndt},
}

@article {MR4732955,
    AUTHOR = {Khale, T.},
     TITLE = {An explicit {V}inogradov-{K}orobov zero-free region for
              {D}irichlet {$L$}-functions},
   JOURNAL = {Q. J. Math.},
  FJOURNAL = {The Quarterly Journal of Mathematics},
    VOLUME = {75},
      YEAR = {2024},
    NUMBER = {1},
     PAGES = {299--332},
      ISSN = {0033-5606},
   MRCLASS = {11M26},
  MRNUMBER = {4732955},
       DOI = {10.1093/qmath/haae010},
       URL = {https://doi.org/10.1093/qmath/haae010},
}

@article {Erdos,
    AUTHOR = {Erd\H{o}s, P. and Kac, M. and van Kampen, E. R. and Wintner, A.},
     TITLE = {Ramanujan sums and almost periodic functions},
   JOURNAL = {Studia Math.},
  FJOURNAL = {Polska Akademia Nauk. Instytut Matematyczny. Studia
              Mathematica},
    VOLUME = {9},
      YEAR = {1940},
     PAGES = {43--53},
      ISSN = {0039-3223},
   MRCLASS = {10.0X},
  MRNUMBER = {4846},
MRREVIEWER = {S. Bochner},
       DOI = {10.4064/sm-9-1-43-53},
       URL = {https://doi.org/10.4064/sm-9-1-43-53},
}

@article {Weyl,
    AUTHOR = {Weyl, H.},
     TITLE = {\"{U}ber die {G}leichverteilung von {Z}ahlen mod. {E}ins},
   JOURNAL = {Math. Ann.},
  FJOURNAL = {Mathematische Annalen},
    VOLUME = {77},
      YEAR = {1916},
    NUMBER = {3},
     PAGES = {313--352},
      ISSN = {0025-5831},
   MRCLASS = {DML},
  MRNUMBER = {1511862},
       DOI = {10.1007/BF01475864},
       URL = {https://doi.org/10.1007/BF01475864},
}

@article {Fujii,
    AUTHOR = {Fujii, A.},
     TITLE = {On the {F}arey series and the {R}iemann hypothesis},
   JOURNAL = {Comment. Math. Univ. St. Pauli},
  FJOURNAL = {Commentarii Mathematici Universitatis Sancti Pauli},
    VOLUME = {54},
      YEAR = {2005},
    NUMBER = {2},
     PAGES = {211--235},
      ISSN = {0010-258X},
   MRCLASS = {11M26 (11B57)},
  MRNUMBER = {2199580},
MRREVIEWER = {J\"{o}rn Steuding},
}

@article {Cobeli,
    AUTHOR = {Boca, F. P. and Cobeli, C. and Zaharescu, A.},
     TITLE = {A conjecture of {R}. {R}. {H}all on {F}arey points},
   JOURNAL = {J. Reine Angew. Math.},
  FJOURNAL = {Journal f\"{u}r die Reine und Angewandte Mathematik. [Crelle's
              Journal]},
    VOLUME = {535},
      YEAR = {2001},
     PAGES = {207--236},
      ISSN = {0075-4102},
   MRCLASS = {11N37 (11B57)},
  MRNUMBER = {1837099},
MRREVIEWER = {Dmitry Y. Kleinbock},
       DOI = {10.1515/crll.2001.049},
       URL = {https://doi.org/10.1515/crll.2001.049},
}

@article {Aistleitner,
    AUTHOR = {Aistleitner, C. and Baker, S. and Technau, N.
              and Yesha, N.},
     TITLE = {Gap statistics and higher correlations for geometric
              progressions modulo one},
   JOURNAL = {Math. Ann.},
  FJOURNAL = {Mathematische Annalen},
    VOLUME = {385},
      YEAR = {2023},
    NUMBER = {1-2},
     PAGES = {845--861},
      ISSN = {0025-5831},
   MRCLASS = {11K06 (60G55)},
  MRNUMBER = {4542733},
MRREVIEWER = {Agamemnon Zafeiropoulos},
       DOI = {10.1007/s00208-022-02362-3},
       URL = {https://doi.org/10.1007/s00208-022-02362-3},
}

@article {Technau,
    AUTHOR = {Technau, N. and Yesha, N.},
     TITLE = {On the correlations of {$n^\alpha\bmod1$}},
   JOURNAL = {J. Eur. Math. Soc. (JEMS)},
  FJOURNAL = {Journal of the European Mathematical Society (JEMS)},
    VOLUME = {25},
      YEAR = {2023},
    NUMBER = {10},
     PAGES = {4123--4154},
      ISSN = {1435-9855},
   MRCLASS = {11K06 (11K31)},
  MRNUMBER = {4634690},
MRREVIEWER = {Sam Chow},
       DOI = {10.4171/jems/1281},
       URL = {https://doi.org/10.4171/jems/1281},
}

@article {Shubin,
    AUTHOR = {Radziwill , M. and Shubin, A.},
     TITLE = {Poissonian pair correlation for {$\alpha n^\theta \bmod 1$}},
   JOURNAL = {Int. Math. Res. Not. IMRN},
  FJOURNAL = {International Mathematics Research Notices. IMRN},
      YEAR = {2024},
    NUMBER = {9},
     PAGES = {7654--7679},
      ISSN = {1073-7928},
   MRCLASS = {11K06 (11L07)},
  MRNUMBER = {4742839},
MRREVIEWER = {Christian Frederik Wei\ss },
       DOI = {10.1093/imrn/rnad289},
       URL = {https://doi.org/10.1093/imrn/rnad289},
}

@article {Munsch,
    AUTHOR = {Aistleitner, C. and El-Baz, D. and Munsch, M.},
     TITLE = {Difference sets and the metric theory of small gaps},
   JOURNAL = {Int. Math. Res. Not. IMRN},
  FJOURNAL = {International Mathematics Research Notices. IMRN},
      YEAR = {2023},
    NUMBER = {5},
     PAGES = {3848--3884},
      ISSN = {1073-7928},
   MRCLASS = {11J71 (11B05 11J83 11K06 11K38 11M06)},
  MRNUMBER = {4565657},
MRREVIEWER = {Aled Walker},
       DOI = {10.1093/imrn/rnab354},
       URL = {https://doi.org/10.1093/imrn/rnab354},
}

@article {MR4458561,
    AUTHOR = {Lutsko, C.},
     TITLE = {Farey sequences for thin groups},
   JOURNAL = {Int. Math. Res. Not. IMRN},
  FJOURNAL = {International Mathematics Research Notices. IMRN},
      YEAR = {2022},
    NUMBER = {15},
     PAGES = {11642--11689},
      ISSN = {1073-7928},
   MRCLASS = {11B57 (11K55 20H05 20H10 22E40)},
  MRNUMBER = {4458561},
MRREVIEWER = {Thomas Garrity},
       DOI = {10.1093/imrn/rnab036},
       URL = {https://doi.org/10.1093/imrn/rnab036},
}

@article {MR4843309,
    AUTHOR = {Br\"{u}dern, J. and Wooley, T. D.},
     TITLE = {Estimates for smooth {W}eyl sums on major arcs},
   JOURNAL = {Int. Math. Res. Not. IMRN},
  FJOURNAL = {International Mathematics Research Notices. IMRN},
      YEAR = {2024},
    NUMBER = {24},
     PAGES = {14662--14688},
      ISSN = {1073-7928},
   MRCLASS = {11L15 (11P05)},
  MRNUMBER = {4843309},
       DOI = {10.1093/imrn/rnae252},
       URL = {https://doi.org/10.1093/imrn/rnae252},
}

@article {MR4542717,
    AUTHOR = {Chen, C. and Kerr, B. and Maynard, J. and
              Shparlinski, I. E.},
     TITLE = {Metric theory of {W}eyl sums},
   JOURNAL = {Math. Ann.},
  FJOURNAL = {Mathematische Annalen},
    VOLUME = {385},
      YEAR = {2023},
    NUMBER = {1-2},
     PAGES = {309--355},
      ISSN = {0025-5831},
   MRCLASS = {11L15 (28A78 42A38)},
  MRNUMBER = {4542717},
MRREVIEWER = {Scott T. Parsell},
       DOI = {10.1007/s00208-021-02352-x},
       URL = {https://doi.org/10.1007/s00208-021-02352-x},
}

@article {MR4298525,
    AUTHOR = {Chen, C. and Shparlinski, I. E.},
     TITLE = {New bounds of {W}eyl sums},
   JOURNAL = {Int. Math. Res. Not. IMRN},
  FJOURNAL = {International Mathematics Research Notices. IMRN},
      YEAR = {2021},
    NUMBER = {11},
     PAGES = {8451--8491},
      ISSN = {1073-7928},
   MRCLASS = {11K38 (11L15)},
  MRNUMBER = {4298525},
MRREVIEWER = {Scott T. Parsell},
       DOI = {10.1093/imrn/rnz293},
       URL = {https://doi.org/10.1093/imrn/rnz293},
}

@article {MR2186997,
    AUTHOR = {Boca, F. P. and Zaharescu, A.},
     TITLE = {On the correlations of directions in the {E}uclidean plane},
   JOURNAL = {Trans. Amer. Math. Soc.},
  FJOURNAL = {Transactions of the American Mathematical Society},
    VOLUME = {358},
      YEAR = {2006},
    NUMBER = {4},
     PAGES = {1797--1825},
      ISSN = {0002-9947},
   MRCLASS = {11J71},
  MRNUMBER = {2186997},
MRREVIEWER = {Takao Komatsu},
       DOI = {10.1090/S0002-9947-05-03783-9},
       URL = {https://doi.org/10.1090/S0002-9947-05-03783-9},
}

@article {MR2180456,
    AUTHOR = {Vaughan, R. C.},
     TITLE = {A variance for {$k$}-free numbers in arithmetic progressions},
   JOURNAL = {Proc. London Math. Soc. (3)},
  FJOURNAL = {Proceedings of the London Mathematical Society. Third Series},
    VOLUME = {91},
      YEAR = {2005},
    NUMBER = {3},
     PAGES = {573--597},
      ISSN = {0024-6115},
   MRCLASS = {11N25 (11P55)},
  MRNUMBER = {2180456},
MRREVIEWER = {S. W. Graham},
       DOI = {10.1112/S0024611505015352},
       URL = {https://doi.org/10.1112/S0024611505015352},
}

@article {MR43120,
    AUTHOR = {Halberstam, H. and Roth, K. F.},
     TITLE = {On the gaps between consecutive {$k$}-free integers},
   JOURNAL = {J. London Math. Soc.},
  FJOURNAL = {The Journal of the London Mathematical Society},
    VOLUME = {26},
      YEAR = {1951},
     PAGES = {268--273},
      ISSN = {0024-6107},
   MRCLASS = {10.0X},
  MRNUMBER = {43120},
MRREVIEWER = {L. Carlitz},
       DOI = {10.1112/jlms/s1-26.4.268},
       URL = {https://doi.org/10.1112/jlms/s1-26.4.268},
}

@article {MR4102722,
    AUTHOR = {Chen, C. and Shparlinski, I. E.},
     TITLE = {On large values of {W}eyl sums},
   JOURNAL = {Adv. Math.},
  FJOURNAL = {Advances in Mathematics},
    VOLUME = {370},
      YEAR = {2020},
     PAGES = {107216, 48},
      ISSN = {0001-8708},
   MRCLASS = {11K38 (11L15 28A78 28A80)},
  MRNUMBER = {4102722},
MRREVIEWER = {Ben Joseph Green},
       DOI = {10.1016/j.aim.2020.107216},
       URL = {https://doi.org/10.1016/j.aim.2020.107216},
}

@article {MR2926988,
    AUTHOR = {Duke, W. and Friedlander, J. B. and Iwaniec, H.},
     TITLE = {Weyl sums for quadratic roots},
   JOURNAL = {Int. Math. Res. Not. IMRN},
  FJOURNAL = {International Mathematics Research Notices. IMRN},
      YEAR = {2012},
    NUMBER = {11},
     PAGES = {2493--2549},
      ISSN = {1073-7928},
   MRCLASS = {11L15 (11E16 11F72 11L05 11L07)},
  MRNUMBER = {2926988},
MRREVIEWER = {Scott T. Parsell},
       DOI = {10.1093/imrn/rnr112},
       URL = {https://doi.org/10.1093/imrn/rnr112},
}

@article {MR922425,
    AUTHOR = {Codec\`a, P. and Perelli, A.},
     TITLE = {On the uniform distribution {$({\rm mod}\,1)$} of the {F}arey
              fractions and {$l^p$} spaces},
   JOURNAL = {Math. Ann.},
  FJOURNAL = {Mathematische Annalen},
    VOLUME = {279},
      YEAR = {1988},
    NUMBER = {3},
     PAGES = {413--422},
      ISSN = {0025-5831},
   MRCLASS = {11K06 (11M06)},
  MRNUMBER = {922425},
MRREVIEWER = {Robert F. Tichy},
       DOI = {10.1007/BF01456278},
       URL = {https://doi.org/10.1007/BF01456278},
}

@article {MR4137069,
    AUTHOR = {Dunn, A. and Kerr, B. and Shparlinski, I. E. and
              Zaharescu, Alexandru},
     TITLE = {Bilinear forms in {W}eyl sums for modular square roots and
              applications},
   JOURNAL = {Adv. Math.},
  FJOURNAL = {Advances in Mathematics},
    VOLUME = {375},
      YEAR = {2020},
     PAGES = {107369, 58},
      ISSN = {0001-8708},
   MRCLASS = {11L07 (11A15 11E25 11K38 11L20 11N32)},
  MRNUMBER = {4137069},
MRREVIEWER = {Ping Xi},
       DOI = {10.1016/j.aim.2020.107369},
       URL = {https://doi.org/10.1016/j.aim.2020.107369},
}

@article {MR1839285,
    AUTHOR = {Rudnick, Z. and Sarnak, P. and Zaharescu, A.},
     TITLE = {The distribution of spacings between the fractional parts of
              {$n^2\alpha$}},
   JOURNAL = {Invent. Math.},
  FJOURNAL = {Inventiones Mathematicae},
    VOLUME = {145},
      YEAR = {2001},
    NUMBER = {1},
     PAGES = {37--57},
      ISSN = {0020-9910},
   MRCLASS = {11K06 (11K60 81Q50)},
  MRNUMBER = {1839285},
MRREVIEWER = {Jens Marklof},
       DOI = {10.1007/s002220100141},
       URL = {https://doi.org/10.1007/s002220100141},
}

@article {MR1628282,
    AUTHOR = {Rudnick, Z. and Sarnak, P.},
     TITLE = {The pair correlation function of fractional parts of
              polynomials},
   JOURNAL = {Comm. Math. Phys.},
  FJOURNAL = {Communications in Mathematical Physics},
    VOLUME = {194},
      YEAR = {1998},
    NUMBER = {1},
     PAGES = {61--70},
      ISSN = {0010-3616},
   MRCLASS = {11J54 (11J71)},
  MRNUMBER = {1628282},
MRREVIEWER = {John B. Friedlander},
       DOI = {10.1007/s002200050348},
       URL = {https://doi.org/10.1007/s002200050348},
}

@article {MR2018926,
    AUTHOR = {Marklof, J.},
     TITLE = {Pair correlation densities of inhomogeneous quadratic forms},
   JOURNAL = {Ann. of Math. (2)},
  FJOURNAL = {Annals of Mathematics. Second Series},
    VOLUME = {158},
      YEAR = {2003},
    NUMBER = {2},
     PAGES = {419--471},
      ISSN = {0003-486X},
   MRCLASS = {11F72 (11E45)},
  MRNUMBER = {2018926},
MRREVIEWER = {Ze\'{e}v Rudnick},
       DOI = {10.4007/annals.2003.158.419},
       URL = {https://doi.org/10.4007/annals.2003.158.419},
}

@article {MR1793613,
    AUTHOR = {Boca, F. P. and Zaharescu, A.},
     TITLE = {Pair correlation of values of rational functions (mod {$p$})},
   JOURNAL = {Duke Math. J.},
  FJOURNAL = {Duke Mathematical Journal},
    VOLUME = {105},
      YEAR = {2000},
    NUMBER = {2},
     PAGES = {267--307},
      ISSN = {0012-7094},
   MRCLASS = {11J71 (11J54 11J83)},
  MRNUMBER = {1793613},
MRREVIEWER = {Ze\'{e}v Rudnick},
       DOI = {10.1215/S0012-7094-00-10524-8},
       URL = {https://doi.org/10.1215/S0012-7094-00-10524-8},
}

@article {Bchahal,
    AUTHOR = {Chahal, B. and Chaubey, S.},
     TITLE = {On the distribution of polynomial {F}arey points and
              {C}hebyshev's bias phenomenon},
   JOURNAL = {Math. Z.},
  FJOURNAL = {Mathematische Zeitschrift},
    VOLUME = {312},
      YEAR = {2026},
    NUMBER = {1},
     PAGES = {Paper No. 8, 44},
      ISSN = {0025-5874,1432-1823},
   MRCLASS = {11B57 (11J71 11K38 11L07 11R42)},
  MRNUMBER = {4984353},
       DOI = {10.1007/s00209-025-03886-5},
       URL = {https://doi.org/10.1007/s00209-025-03886-5},
}

@article {Bittu,
    AUTHOR = {Chahal, B.},
     TITLE = {Chebyshev's bias for irrational factor function},
   JOURNAL = {Ramanujan J.},
  FJOURNAL = {Ramanujan Journal. An International Journal Devoted to the
              Areas of Mathematics Influenced by Ramanujan},
    VOLUME = {69},
      YEAR = {2026},
    NUMBER = {2},
     PAGES = {Paper No. 50, 29},
      ISSN = {1382-4090,1572-9303},
   MRCLASS = {11N37 (11R42 11R59)},
  MRNUMBER = {5025490},
       DOI = {10.1007/s11139-026-01321-9},
       URL = {https://doi.org/10.1007/s11139-026-01321-9},
}

@article {Lu,
    AUTHOR = {Lu, M.},
     TITLE = {Correlations of {F}arey fractions with square-free
              denominators},
   JOURNAL = {Acta Arith.},
  FJOURNAL = {Acta Arithmetica},
    VOLUME = {220},
      YEAR = {2025},
    NUMBER = {4},
     PAGES = {323--344},
      ISSN = {0065-1036,1730-6264},
   MRCLASS = {11B57 (11J71 11K06)},
  MRNUMBER = {4968752},
MRREVIEWER = {Jonas\ \v{S}iaulys},
       DOI = {10.4064/aa240524-19-6},
       URL = {https://doi.org/10.4064/aa240524-19-6},
}

@article {Lu1,
    AUTHOR = {Lu, M.},
     TITLE = {Correlations of fractions whose denominators are products of
              primes},
   JOURNAL = {Math. Nachr.},
  FJOURNAL = {Mathematische Nachrichten},
    VOLUME = {298},
      YEAR = {2025},
    NUMBER = {7},
     PAGES = {2263--2281},
      ISSN = {0025-584X,1522-2616},
   MRCLASS = {11K31 (11B57)},
  MRNUMBER = {4939750},
MRREVIEWER = {Vilius\ Stakenas},
       DOI = {10.1002/mana.12027},
       URL = {https://doi.org/10.1002/mana.12027},
}
\end{document}